\documentclass[11pt,leqno]{amsart}

\usepackage{tikz}
\usetikzlibrary{shapes,arrows}
\RequirePackage[colorlinks,citecolor=blue,urlcolor=blue]{hyperref}
\usepackage{pkgfile}

\newtheorem{claim}[thm]{Claim}

\def \a {\alpha}
\def \b {\beta}

\newcommand{\Graph}{{\sf G}}

\def\corAB{}

\begin{document}

\tikzstyle{decision} = [diamond, draw, fill=gray!20, 
    text width=4.5em, text badly centered, node distance=3cm, inner sep=0pt]
\tikzstyle{block} = [rectangle, draw, fill=gray!20, 
    text width=10em, text centered, rounded corners, minimum height=3em]
\tikzstyle{line} = [draw, -latex']
\tikzstyle{cloud} = [draw, ellipse,fill=gray!20, node distance=3cm,
    minimum height=2em]

\title[{Large deviations for subgraph counts}]{{Upper tail large deviations of regular subgraph counts in Erd\H{o}s-R\'{e}nyi graphs} in the full localized regime}

\author{Anirban Basak}
\address{Anirban Basak, International Centre for Theoretical Sciences, Tata Institute of Fundamental Research, Bangalore, India}
\email{anirban.basak@icts.res.in}
\author{Riddhipratim Basu}
\address{Riddhipratim Basu, International Centre for Theoretical Sciences, Tata Institute of Fundamental Research, Bangalore, India}
\email{rbasu@icts.res.in}

%
%
%
%

\date{\today}



\maketitle

\begin{abstract}
For a  $\Delta$-regular connected graph ${\sf H}$ the problem of determining the upper tail large deviation for the number of copies of ${\sf H}$ in $\mathbb{G}(n,p)$, an Erd\H{o}s-R\'{e}nyi graph on $n$ vertices with edge probability $p$, has generated significant interests. For $p\ll 1$ and $np^{\Delta/2} \gg (\log n)^{1/(v_{\sf H}-2)}$, where $v_{\sf H}$ is the number of vertices in ${\sf H}$, the upper tail large deviation event is believed to occur due to the presence of localized structures. In this regime the large deviation event that the number of copies of ${\sf H}$ in $\mathbb{G}(n,p)$ exceeds its expectation by a constant factor is predicted to hold at a  speed $n^2 p^{\Delta} \log (1/p)$ and the rate function is conjectured to be given by the solution of a mean-field variational problem. After a series of developments in recent years, covering progressively broader ranges of $p$, the upper tail large deviations for cliques of fixed size was proved by Harel, Mousset, and Samotij \cite{hms} in the entire localized regime. This paper establishes the conjecture for {\em all} connected regular graphs in the {\em whole} localized regime. 
%
%
%
%
%
%
%
%
%
%
%
%
\end{abstract}

\section{Introduction}\label{sec:intro}
Let $\G_n$ be a random graph on $n$ vertices and ${\sf H}$ be a fixed graph. In recent years the study of large deviations for the number of copies of ${\sf H}$ in $\G_n$ has received a paramount interest. See the recent monograph \cite{Chbook} for a comprehensive overview on this topic. Consider the simplest non-trivial set up:~$\G_n= \G(n,p)$ is the Erd\H{o}s-R\'{e}nyi graph on $n$ vertices with edge connectivity probability $p=p(n) \in (0,1)$, and ${\sf H}=K_3$ is a triangle. After a series of works \cite{aug, BGLZ, chd, cod, eld, LZ}, large deviations bounds for the upper tail of   triangle counts in $\G(n,p)$ for all $p \ll 1$ satisfying $np \gg \log n$ is established by Harel, Mousset, and Samotij \cite{hms} (for two sequences of positive reals $\{a_n\}$ and $\{b_n\}$ we write $b_n=o(a_n)$, $a_n \gg b_n$, and $b_n \ll a_n$ to denote $\limsup_{n \to \infty} b_n/a_n =0$). In this regime the large deviation event is due to the presence of localized structures in $\G(n,p)$. Whereas, in the complement regime, i.e. $1 \ll np \ll \log n$, as shown in \cite{hms}, the large deviation is given by the large deviation of a Poisson random variable with an appropriate mean. Thus these two regimes can be termed as the {\em localized regime} and the {\em Poisson regime}, respectively. 

Moving to more general subgraph counts it was established in the series of works mentioned above that for any $\Delta$-regular connected graph ${\sf H}$, the upper tail large deviation occurs due to the presence of localized structures for $1\gg p\gg n^{-\alpha_{\sf H}}$ for successively improved values of $\alpha_{\sf H} \in (0,1)$. 

On the other hand, it was shown in \cite{hms} that for any $\Delta$-regular connected graph ${\sf H}$ the Poisson regime is characterized by the threshold {$1\ll np^{\Delta/2} \ll (\log n)^{1/(v_{\sf H} -2)}$} and the exponent $1/(v_{\sf H} -2)$ is optimal (see \cite[Section 8]{hms} for a discussion on the optimality), where $v_{\sf H}$ is the number of vertices in ${\sf H}$. It naturally leads to the conjecture that for any $\Delta$-regular connected graph for the entire regime $ (\log n)^{1/(v_{\sf H}-2)} \ll {np^{\Delta/2}} \ll n$ the large deviation for the upper tail of the number of copies of ${\sf H}$ in $\G(n,p)$ is due to the presence of localized structures, where the speed is predicted to be $n^2 p^\Delta \log(1/p)$ with the rate function to be given by a mean-field variational problem as in \cite{chd}. 

When ${\sf H}$ is a clique this conjecture was proved in \cite{hms}. For general regular graphs the best previously known result in this direction is again due to \cite{hms}, where the authors derived upper tail large deviations for all $\Delta$-regular connected {\em non-bipartite graphs} ${\sf H}$ in the regime of $p$ satisfying $np^{\Delta/2} \gg (\log n)^{\Delta v_{\sf H}^2}$, whereas for $\Delta$-regular connected {\em bipartite graphs} this is done for the regime $p^{\Delta/2}\ge n^{-1/2-o(1)}$ with the $o(1)$ term decays to zero at any arbitrary rate as $n \to \infty$.  
%
The goal of this paper is to derive the upper tail large deviation for the subgraph count of {\em any} regular connected graph in the {\em entire} localized regime. 


To state our main result we need to introduce some notation. For any graph $\Graph$ we write $V(\Graph)$ and $E(\Graph)$ to denote its vertex and edge sets, respectively. Next we define the notion of labelled copies of a given graph in another graph. 
\begin{dfn}\label{dfn:labelled-density}
Given graphs $\Graph$ and ${\sf H}$ we write $N({\sf H}, {\sf G})$ to denote the number of labelled copies of ${\sf H}$ in $\Graph$. That is,
\beq\label{eq:labelled-density}
N({\sf H}, \Graph):= \sum_{\varphi: V({{\sf H}}) \hookrightarrow V({\Graph})}  \prod_{(x,y) \in E({\sf H})} a^\Graph_{\varphi(x),\varphi(y)},
\eeq
where the sum is over all {\em injective} maps $\varphi$ from $V({\sf H})$ to $V(\Graph)$, and $(a^\Graph_{i,j})_{i,j \in V(\Graph)}$ is the adjacency matrix of the graph $\Graph$. 
\end{dfn}

Next for any graph ${\sf H}$ we denote its independence polynomial by $P_{\sf H}(\cdot)$. That is,
\[
P_{\sf H}(x):= \sum_{k\ge 0} i_{\sf H}(k) x^k,
\]
where $i_{\sf H}(k)$ is the number of $k$-element independent subsets of ${\sf H}$ and set $i_{\sf H}(0):=\corAB{1}$. We further denote $\theta_{\sf H}$ to be unique positive solution to $P_{\sf H}(\theta)=1+\delta$. ($P_{\sf H}$ being a polynomial with nonnegative coefficients, hence increasing on $[0,\infty)$, with $P_{\sf H}(0)=1$, the existence of a unique solution is guaranteed.)

{Let us also note that for any graph ${\sf H}$ we have $\E[N({\sf H}, \Graph)] = (1+o(1)) n^{v_{\sf H}} p^{e({\sf H})}$, where $v_{\sf H}$ and $e({\sf H})$ denote the number of vertices and edges of ${\sf H}$, respectively. Thus for any graph ${\sf H}$ and $\delta >0$ the upper tail event is given by 
\beq\label{eq:UT}
{\sf UT}({\sf H}, \delta):= \left\{ N({\sf H}, \G(n,p) ) \ge (1+\delta) n^{v_{\sf H}} p^{{ e}({\sf H})}\right\}.
\eeq}
The following is the main result of this paper.


\begin{thm}\label{thm:cycle-main}
For any $\Delta$-regular connected graph ${\sf H}$, with $\Delta \ge 2$, and $\delta >0$,
\beq
\label{eq:cycle-ldp}
\lim_{n \to \infty} -\f{\log {\P({\sf UT}({\sf H}, \delta))}}{n^2 p^\Delta \log (1/p)} = \left\{\begin{array}{ll}
\min\{\theta_{\sf H}, \f12 \delta^{2/v_{\sf H}}\} & \mbox{ if } n^{1/2} \ll np^{\Delta/2} \ll n,\\
\f12 \delta^{2/v_{\sf H}} & \mbox{ if } (\log n)^{1/(v_{\sf H}-2)} \ll n p^{\Delta/2} \ll n^{1/2}.
\end{array}
\right.
\eeq
\end{thm}

Let us add that, as the edges in an Erd\H{o}s-R\'{e}nyi graph are independent, the case of $\Delta=1$ is trivial and hence omitted from the setup of Theorem  \ref{thm:cycle-main}. 

We point to the reader that the upper tail large deviations of {\em unlabelled copies} of ${\sf H}$ in $\G(n,p)$ has been considered in \cite{hms}. As the number of labelled and unlabelled copies only differ from each other by a factor $|{\rm Aut}({\sf H})|$, the number of automorphisms of ${\sf H}$, the large deviation speed as well as the rate function are identical in these two cases. For convenience we have chosen to work with the labelled copies of ${\sf H}$ in $\G(n,p)$. 

{We remark that for $p \gg n^{-1/\Delta}$ Theorem \ref{thm:cycle-main} follows from \cite{hms}. There it was shown that the \abbr{LHS} of \eqref{eq:cycle-ldp} equals the limit of a mean-field variational problem (see \eqref{eq:var-prblm-labelled} below), whereas the limiting behavior of that variational problem was studied in \cite{BGLZ}. In this paper, in the other regime of $p$ we directly show \eqref{eq:cycle-ldp}, bypassing the variational problem.} 

\begin{rmk}
The case of irregular graphs is more involved. The  upper tail large deviations for the number of copies of an irregular graph ${\sf H}$ in an Erd\H{o}s-R\'{e}nyi graph have been derived in \cite{chd, cod, eld} for $p \gg n^{-c_{\sf H}}$, where $c_{\sf H} \in (0,1)$ is some constant depending on the graph ${\sf H}$. On the other hand \cite{jor} finds the order of the upper and lower bounds on the $\log$ of the probability of ${\sf UT}({\sf H}, \cdot)$ for a much wider range of $p$. However the order of these two bounds differ by a factor of $\log (1/p)$. For $p \ll n^{-1/\Delta}$, where $\Delta$ is the maximum degree in an irregular graph ${\sf H}$, even the order of the $\log$ probability of ${\sf UT}({\sf H}, \cdot)$ is not well understood. In fact, in this case the conjectured bound of \cite{DK1} have been disproved in some examples (see  \cite{SW}). 
\end{rmk}


\begin{rmk}
A related problem of interest is to study upper tail large deviations of the {\em homomorphism count}. The homomorphism count of ${\sf H}$ in $\Graph$, denoted by ${\rm Hom}({\sf H}, \Graph)$, is defined to be sum in \eqref{eq:labelled-density} when $\varphi$ varies over {\em all} maps from $V({\sf H})$ to $V(\Graph)$. For $\ell \ge 3$, denoting $C_\ell$ to be cycle graph of length $\ell$, it follows from \cite{aug, cod, hms} that the upper tail large deviations for $N(C_\ell, \G(n,p))$ and ${\rm Hom}(C_\ell, \G(n,p))$ are the same for $n^{-1/2}(\log n)^2 \ll p \ll 1$. The same phenomenon {should} hold for a wider range of $p$. 

A natural way to derive the upper tail large deviations of ${\rm Hom}(C_\ell, \G(n,p))$ from that of labelled copies of subgraphs in $\G(n,p)$ is to write the former as a sum of $N({\sf H}_\star, \G(n,p))$, where ${\sf H}_\star$ is a {\em quotient subgraph} of $C_\ell$ (see \cite[Chapter 5]{L} for more details on this representation), and derive the upper tail large deviations of $N({\sf H}_\star, \G(n,p))$ {for each such ${\sf H}_\star$}. As the quotient graphs of $C_\ell$ involves {\em star} graphs this route would in particular need an understanding of the upper tail large deviations for such irregular graphs. The best known result in this direction is due to \cite{SW1} where the speed of the large deviation is identified.
\end{rmk}

\begin{rmk}
It is immediate to note that for any graph $\Graph$ and $t \ge 2$,
\[
{\rm Hom}(C_{2t}, \Graph) = \tr [{\rm Adj}(\Graph)^{2t}]= \sum_{i=1}^n \lambda_i^{2t},
\]
where ${\rm Adj}(\Graph)$ is the {\em adjacency matrix} of $\Graph$, $\tr(\cdot)$ denotes the trace, and $\lambda_1 \ge \lambda_2 \ge \cdots \ge \lambda_n$ are the eigenvalues of ${\rm Adj}(\Graph)$ arranged in a non-increasing order. So upper tail large deviations of ${\rm Hom}(C_{2t}, \G(n,p))$ for all $t \ge 2$ would yield the same for the top eigenvalue of ${\rm Adj}(\G(n,p))$. Using results of \cite{cod} this has been achieved in \cite{BG} for $p \gg n^{-1/2}$. Extending the same for a sparser regime would need an understanding of the upper tail large deviations of ${\rm Hom}(C_{2t}, \G(n,p))$. We postpone it to a future work.

{Let us also add that for $p \sim 1$ ({for two sequences of positive reals $\{a_n\}$ and $\{b_n\}$ we write $a_n \sim b_n$ if $0< \liminf_{n \to \infty} a_n / b_n \le \limsup_{n \to \infty} a_n /b_n < \infty$}) the upper tail large deviation of the top eigenvalue was studied in \cite{LZ0}. There the large deviation was derived by showing that the {\em operator norm} of a {\em graphon} is a {\em nice graph parameter}, associating the upper tail event for such graph parameters to a variational problem, and then by analyzing that variational problem. See \cite[Theorem 1.2]{LZ0} and its proof for more details.}
\end{rmk}

\begin{rmk}\label{rmk:rate-fn}
Note that Theorem \ref{thm:cycle-main} does not discuss the nature of the large deviation when $p \sim n^{-1/\Delta}$. It follows from \cite{hms} that for such $p$ the large deviation speed continues to be $n^2 p^\Delta \log (1/p)$. The rate function turns out to be the limit of an $n$-dependent constrained optimization problem defined over the space of graphs on $n$ vertices (see \eqref{eq:var-prblm-labelled} below). 
The solution of this optimization problem depends on the graph ${\sf H}$. 
To highlight the difficulty and generality of this problem we refer the reader to Remark \ref{rmk:rate-fn-c4} where we provide a description of the rate function for $C_4$. 
\end{rmk}

\begin{rmk}
It is of interest to study the typical structure of $\G(n,p)$ conditioned on the upper tail event that $N({\sf H}, \G(n,p))$ exceeds its expectation by a constant factor. When ${\sf H}$ is a clique graph it has been shown in \cite[Theorem 1.8]{hms} that conditioned on the upper tail event the Erd\H{o}s-R\'{e}nyi graph typically has either a {\em clique-like} or a {\em hub-like} structure. In the setting of Theorem \ref{thm:cycle-main} one {can} modify its proof appropriately and proceed similarly along the lines of the proofs of \cite[Propositions 6.4 and 6.6]{hms} to deduce the same. To retain the clarity of the proof of Theorem \ref{thm:cycle-main} this extension is not carried out in detail.
\end{rmk}

\begin{rmk}
Much less is known about the large deviations of subgraph counts in random graph models beyond $\G(n,p)$. This has been studied in the context of {\em random $d_n$-regular graphs} \cite{BD} and {\em random hypergraphs} \cite{LiZ} when the (hyper)-graphs are not too sparse. It is worthwhile to investigate whether the ideas of \cite{hms} and this paper can be adapted to these problems to treat sparser regimes. 
\end{rmk}

\begin{rmk}
In an independent work, Cohen \cite{Co} studies upper tail large deviations of the number of {\em induced} copies of $C_4$ in an Erd\H{o}s-R\'{e}nyi graph, for all $p =o(1)$ satisfying $p \gg n^{-1+\updelta}$ for any $\updelta >0$. Let us add that the upper tail large deviations of $N(C_4, \G(n,p))$ and the number of induced copies $C_4$ in $\G(n,p)$ are quite different because, for $p \ll n^{-1/2}$, the large deviation event in the first setup is due to planting a clique, whereas in the other case it is due to planting complete bipartite graphs of appropriate sizes depending on the sparsity parameter $p$. For further details we refer the reader to \cite{Co}.
\end{rmk}

\subsection{Background and related results}\label{sec:discuss}
As alluded to already, the study of upper tail large deviations of subgraph counts in an Erd\H{o}s-R\'{e}nyi graph has a long and rich history. It can be traced back to the work of Janson and Ruci\'{n}ski \cite{JR} where the problem is described as {\em the infamous upper tail} problem. When ${\sf H}=K_3$, a triangle, it was shown in \cite{jor, KV} that for any $p {\ge} \log n/n$ and $\delta>0$ one has the bounds
\beq\label{eq:upper-tail-conc-bd}
\exp(-c_1(\delta)n^2 p^2 \log(1/p)) \le \P({\sf UT}(K_3, \delta)) \le \exp(-c_2(\delta) n^2 p^2),
\eeq
for some constants $0< c_1(\delta), c_2(\delta) < \infty$, where {we recall the definition of ${\sf UT}(\cdot, \cdot)$ from \eqref{eq:UT}.} 

About a decade later the discrepancy between the exponents of the upper and lower bounds in \eqref{eq:upper-tail-conc-bd}  was resolved in \cite{Ch, DK} by showing that the exponent in the upper bound can be tightened to $c_2(\delta) n^2 p^2\log (1/p)$. These left open the problem of determining the asymptotic dependence of the constants $c_1(\delta)$ and $c_2(\delta)$ in $\delta$ and showing that they differ from each other only by ${o(1)}$. 

For $p \in (0,1)$ fixed, i.e.~not depending on $n$, this problem has been resolved by Chatterjee and Varadhan \cite{CV} where a key ingredient was Szemer\'{e}di's regularity lemma \cite{Sz} for dense graphs. 
{In particular, they showed that the $\log$ probability of the upper tail event ${\sf UT}(K_3,\delta)$, when properly normalized, is asymptotically equivalent to a mean-field variational problem. Lubetzky and Zhao  \cite{LZ0} derived the solution to that variational problem (for general regular graphs) in the {\em replica symmetry} region, thereby settling this question in that regime for $p \sim 1$.} 

Due to the poor quantitative bound of Szemer\'{e}di's regularity lemma the approach of \cite{CV} cannot be adopted when $p\sim n^{-c}$ for some $c >0$. Recently there was a series of breakthroughs in this area. Chatterjee and Dembo \cite{chd} developed a general framework to treat the upper tail large deviation of any nonlinear smooth function $f(\cdot)$ of i.i.d.~$\dBer(p)$ random variables, where $f(\cdot)$ is of {\em low-complexity} characterized by the existence of a net of small cardinality of the image of the unit hypercube under the map $\nabla f$, the gradient of $f$. {Similar to the dense regime, \cite{chd} showed that the $\log$ probability of the upper tail event $\{f \ge(1+\delta) \E [f]\}$ can be associated to a mean-field variational problem. 
This, together with solution of the variational problem derived in \cite{LZ},} when applied to the problem of the upper tail of $N(K_3, \G(n,p))$ yields large deviations for $n^{-1/42} (\log n)^{11/14} \ll p \ll 1$.

Eldan in \cite{eld} derived upper tail large deviations of a nonlinear Lipschitz function $f(\cdot)$ of i.i.d.~$\dBer(p)$ variables when the {\em Gaussian width} of the image of the discrete unit hypercube under the gradient map $\nabla f$ is small. This, {together with \cite{LZ}}, improved the range of $p$ where the large deviations of ${\sf UT}(K_3, \cdot)$ could be established to $ n^{-1/18} (\log n) \ll p \ll 1$. {In a related work Austin \cite{aus} shows that Gibbs measures on arbitrary product spaces can be approximated by mixtures of product measures, if those Gibbs measures are induced by Hamiltonians of low complexity. There the low complexity is characterized by a bound on the covering number of the range of the discrete gradient of the Hamiltonian.}

The work of \cite{eld} was further improved by Augeri \cite{aug}, and Cook and Dembo \cite{cod} to show that the large deviations of the upper tail events ${\sf UT}(C_\ell, \cdot)$, for any $\ell \ge 3$ (including in particular the case of $K_3$), is due to the presence of localized structures in the regime $ n^{-1/2}(\log n)^2 \ll p \ll 1$. {As before, to derive a result analogous to \eqref{eq:cycle-ldp}, for this regime of $p$, one needs to understand the asymptotic behavior of the corresponding variational problem. That was derived in \cite{BGLZ}}. 
The work of Augeri \cite{aug} is an advancement of \cite{chd} where one can now consider non-smooth convex Lipschitz functions and it provides a cleaner error bound suitable for use in a wider sparse regime.  Whereas the key to \cite{cod} is the derivation of a new quantitative version of Szemer\'{e}di’s regularity lemma and the {\em counting lemma} tailored for the sparse regime.

Let us add that, in the context of cycle graphs, the approaches used in \cite{aug, cod} require approximating symmetric square matrices of dimension $n$ with entries in $[0,1]$ in the Hilbert-Schmidt norm. This is done by using standard nets for the large eigenvalues {of such matrices} and the eigenvectors corresponding to those large eigenvalues. As seen from Theorem \ref{thm:cycle-main}, for $p \ll (\log n)^{-1/2}n^{-1/2}$ the speed of the large deviations for ${\sf UT}(C_\ell, \cdot)$ is $o(n)$. {Observe that $\log$ of the cardinality of any net of constant mesh-size of a unit vector in dimension $n$ is at least of order $n$. Thus for $p \ll (\log n)^{-1/2}n^{-1/2}$ using a standard net even for one eigenvector will be too expensive to deduce the large deviation.} Therefore, to be able to use the machinery of \cite{aug, cod} for that sparser regime of $p$ one needs to a priori show that the eigenvectors corresponding to large eigenvalues of ${\rm Adj}(\G(n,p))$ are {\em localized} {with respect to some appropriately chosen collection of basis vectors} with sufficiently large probability. 

The recent most breakthrough in the context of upper tail large deviations is due to Harel, Mousset, and Samotij \cite{hms} where a novel idea is put forward. Their general approach can be described as follows: Using an adaptation of the classical moment argument of \cite{jor} it is shown that the probability of ${\sf UT}({\sf H}, \cdot)$ can be bounded, upto a multiplicative factor of two, by that of the existence of a subgraph $ \Graph_\star$ of $\G(n,p)$ such that it does not have too many edges and has an adequate number of copies of subgraphs of ${\sf H}$. Such subgraphs will be termed as {\em seed graphs} (see Definitions \ref{dfn:seed-0} and \ref{dfn:seed}). Thus, by the union bound 
\beq\label{eq:seed-union}
\P({\sf UT}({\sf H}, \delta)) \le 2 \sum_{\Graph_\star: \, \Graph_\star \text{ is seed}} p^{e(\Graph_\star)}.
\eeq
However, even for ${\sf H}=K_3$ the upper bound in \eqref{eq:seed-union} turns out to be too large to yield the correct large deviation estimate. A natural way to improve the above upper bound is to first {\em approximate} the seed graphs, or in other words construct a net for the seed graphs, and then apply the union bound over the graphs in that net to derive an upper bound on the probability of the existence of a seed graph $\Graph_\star \subset \G(n,p)$. 

At a high level, this was indeed the approach taken in \cite{hms}. In particular, by peeling off edges from any  seed graph $\Graph_\star$, without losing too many copies of subgraphs of ${\sf H}$, they obtain a subgraph $\Graph \subset  \Graph_\star$ such that each edge in $\Graph$ participates in a large number of copies of subgraphs of ${\sf H}$. Following \cite{hms} we term any such graph $\Graph$ to be a {\em core graph} (see Definition \ref{dfn:core-graph} for a precise formulation). 
{Thus the probability of ${\sf UT}({\sf H}, \cdot)$ can now be bounded by that of the existence of a core subgraph of $\G(n,p)$ (again up to a multiplicative factor of two).} This reduction from the seed graphs to the core graphs can be thought of constructing a net for the set of seed graphs which is given by the set of core  graphs. 

Having found the net one then proceeds to apply a union bound, as in \eqref{eq:seed-union}, where now the sum is over all core graphs. To complete the argument one is then left with the task of finding a bound on $\gN_{\bm e}$, the number of core graphs with a given number of edges ${\bm e}$. If 
\beq\label{eq:gN-e}
\gN_{\bm e} \le \exp({\bm e} \log(1/p)\cdot o(1)),
\eeq
then $\gN_{\bm e}$ being sufficiently small compared to $(1/p)^{\bm e}$, the inverse of the probability of observing any graph with ${\bm e}$ edges, one can take the union bound to derive that {the probability} of ${\sf UT}({\sf H}, \cdot)$ is bounded by $\exp(-(1-o(1))\cdot {\bm e}_0\log(1/p))$, where ${\bm e}_0$ is the minimum number of edges a core graph must possess. This results in the correct upper bound on the probability of ${\sf UT}({\sf H}, \cdot)$. 

In \cite{hms}, this general scheme is successfully employed for cliques to derive the upper tail large deviations in the entire localized regime (and also the upper tail large deviations of {\em $k$-term arithmetic progressions}). For $\Delta$-regular bipartite graphs (e.g.~$C_4$) this scheme could only derive the upper tail large deviations when $p \ge n^{-1/\Delta-o(1)}$. 

The obstacle of extending the above {for a sparser regime} stems from the fact that the bound \eqref{eq:gN-e} breaks down for bipartite graphs {when $p \le n^{-1/\Delta -o(1)}$}. Indeed, as already noted in \cite[Section 10]{hms}, for any $C>0$ the number of labelled copies of $K_{2, C n^2 p^2}$ ({for $a, b \in \N$ the standard notation $K_{a,b}$ denotes the complete bipartite graph with two parts having $a$ and $b$ vertices respectively}) in the complete graph on $n$ vertices exceeds the \abbr{RHS} of \eqref{eq:gN-e}. For $C$ sufficiently large the graph $K_{2, C n^2 p^2}$ becomes a core graph, for ${\sf H}=C_4$, and hence one cannot proceed as in \cite{hms}.

Instead, in this paper we obtain appropriate nets for the set of core graphs with sufficiently small  cardinalities so that the union bound produces a meaningful upper bound. As will be elaborated below, the strategy for procuring such nets is delicate and it depends on the number of edges in a core graph, as well as on the sparsity level of the edge probability $p$. 
Formulating this approach requires a couple of new key ideas. 
\begin{enumerate}
\item[(1)]
We show that for any core graph $\Graph$ with $e(\Graph) = O(n^2 p^\Delta)$ ({for two sequences of positive reals $\{a_n\}$ and $\{b_n\}$ hereafter we use the standard notation $b_n=O(a_n)$ to denote that $\limsup_{n \to \infty} b_n /a_n < \infty$}) there exists a subgraph $\wh \Graph \subset \Graph$ (which we term as a {\em strong-core} graph, see Definition \ref{dfn:strong-core}) for which one can extract a bipartite subgraph $\wh \Graph_\gb:=\wh \Graph_\gb(\wh \Graph)$ of it that has an arbitrarily long {\em block path} structure (equivalently {\em blow up of a path}), as shown in Figure \ref{fig:path-str} {below}. The set of all strong-core graphs serves as the right net in this subcase. 

{Using combinatorial arguments we deduce that $\wh\Graph_\gb$ and $\wh\Graph\setminus \wh\Graph_\gb$ are individually {\em entropically stable}. We use the term entropic stability to broadly mean that if $\wh\Graph_\gb$ (or $\wh\Graph\setminus \wh \Graph_\gb$) is assumed to contain adequate numbers of copies of ${\sf H}$ then there exist appropriate lower bounds on their edges and an upper bound on the number of such graphs suitable for the union bound yielding the correct large deviation probability. For example, see Lemmas \ref{lem:small-count} and \ref{lem:large-count}.

Let us point to the reader that in \cite{hms} an upper tail event is termed to be entropically stable if \eqref{eq:gN-e} holds for all ${\bm e}$. As already discussed above that it does not hold for non-bipartite graphs for $p \le n^{-1/\Delta-o(1)}$. Therefore we adopt the above notion of entropic stability that is somewhat different, weaker, and broader than the one in \cite{hms}.

Next, continuing the description of the key ideas of the proof, the block path structure of Figure \ref{fig:path-str} also allows us to choose $\Graph_\gb$ in such a way so that almost all copies of ${\sf H}$ in $\Graph$ must either be contained in $\Graph_\gb$ or $\Graph\setminus \Graph_\gb$. This, in turn, implies that showing entropic stability of $\Graph_\gb$ and $\Graph\setminus \Graph_\gb$ separately guarantees the same for the whole graph $\Graph$.}

\item[(2)] 
To treat core graphs with $e(\Graph) \gg n^2 p^\Delta$ we focus at its subgraph induced by the edges with at least one end point having a {\em low} degree. If $np^{\Delta/2} \ge (\log n)^{ v_{\sf H}}$ then this subgraph can be shown to be bipartite. Using a simple combinatorial argument we show that this case is {\em entropically non-viable} (or equivalently {\em entropically suboptimal}), {i.e.~the probability is much smaller than the large deviation probability}. 

However, for $np^{\Delta/2} \le (\log n)^{v_{\sf H}}$ the bipartite structure described above ceases to exist. For example, consider the graph $\mathscr{K}_{2, C n^3 p^3}$ which consists of $C n^3 p^3$ edges and two vertices such that the two end points of these edges are attached to those vertices. {It can be checked that for $np \le \log(1/p)$} the graph $\mathscr{K}_{2, C n^3 p^3}$ is indeed a core graph for $C_6$. To tackle this additional difficulty our strategy on obtaining a net for core graphs 
depends on whether $\cN_{1,1}({\sf H}, \Graph)$, the number of copies of ${\sf H}$ in $\Graph$ that uses at least one edge whose both end points are of low degree, is large or small. To run this procedure effectively and cover the full localized regime we further split the range of $e(\Graph)$ into dyadic scales. Hence, this part of the proof requires a {\em chaining-type argument}. A schematic representation of the chaining argument is presented in Figure \ref{fig:chaining}.
\end{enumerate}
For a more elaborate description of these two key ingredients we refer the reader to Section \ref{sec:outline}. 

Let us add that during the chaining argument the condition $np^{\Delta/2} \gg (\log n)^{1/(v_{\sf H}-2)}$ is used {\em only} to argue that $\wt \cN_{1,1}({\sf H}, \Graph)$, the number of copies of ${\sf H}$ in $\Graph $ that uses only those edges for which both of their end points are of bounded degree, is negligible compared to $N({\sf H}, \Graph)$. {If $np^{\Delta/2} \sim (\log n)^{1/(v_{\sf H}-2)}$ and $e(\Graph) =O(n^2 p^\Delta \log (1/p))$} then this is no longer true. Furthermore in the Poisson regime the large deviation is expected to be driven by $\wt\cN_{1,1}({\sf H}, \Graph)$. 
Therefore we believe that with some additional efforts the ideas of this paper may be used to show that the {localized and the Poisson behaviors coexist} when $np^{\Delta/2} \sim (\log n)^{1/(v_{\sf H}-2)}$. 

We end this section with the following remark: In \cite{hms} the localized nature of the upper tail large deviation event for non-bipartite $\Delta$-regular graphs ${\sf H}$  is derived in \cite{hms} when $np ^{\Delta/2} \gg (\log n)^{\Delta v_{\sf H}^2}$. The suboptimality in the exponent of $\log n$ is most likely due to the absence of the chaining procedure which as will be seen below is crucial to derive the large deviations in the whole localized regime.

\subsection*{Acknowledgements} We thank Wojciech Samotij for several helpful comments and, Sourav Chatterjee and Yufei Zhao for pointing to relevant references. AB is partially supported by a MATRICS Grant (MTR/2019/
001105)  and a Start-up Research Grant 
(SRG/2019/001376) from Science and Engineering Research Board of Govt.~of India, and an Infosys--ICTS Excellence Grant. RB is partially supported by a Ramanujan Fellowship (SB/S2/RJN-097/2017) from Govt.~of India, and an ICTS--Simons Junior Faculty Fellowship. We acknowledge the support of DAE, Govt.~of India, to ICTS under project no.~12-R\&D-TFR-5.10-1100. This research is carried out in part  as members of the Infosys-Chandrasekharan Virtual Center for Random Geometry, supported by a grant from the Infosys Foundation.

\section{Preliminaries and proof outline}\label{sec:outline}

In this section we describe the idea behind the proof of Theorem \ref{thm:cycle-main} in some detail and introduce relevant definitions and notation. Let us recall that for $p \ge n^{-1/\Delta-o(1)}$, where $o(1)$ is any term converging to zero as $n \to \infty$, the conclusion of Theorem \ref{thm:cycle-main} already follows from \cite{hms}. So, in the rest of the paper we will only consider the case $p \le n^{-1/\Delta -o(1)}$, where $o(1)$ is an appropriately chosen term converging to zero as $n \to \infty$. 

The upper bound in \eqref{eq:cycle-ldp} {(i.e.~the \abbr{LHS} upper bounded by the \abbr{RHS}) or equivalently the lower bound on the large deviations probability} essentially follows by planting a clique of appropriate size. Hence the main work is the derivation of the lower bound in \eqref{eq:cycle-ldp}. 

{In all other previous works, to derive the relevant lower bound the common route has been to show that the logarithm of the upper tail probability is bounded above by the solution of an $n$-dependent variational problem. That variational problem is studied in \cite{BGLZ}. Combining these two pieces together the lower bound on the negative of the $\log$ probability of the upper tail event follows. Let us state the relevant variational problem.} It needs some notation. 

For $x \in [0,1]$ we define the binary entropy
\[
I_p(x):= x \log \f{x}{p}+ (1-x) \log \f{1-x}{1-p},
\]
and for a vector ${\bm x}:=(x_1,x_2,\ldots,x_N) \in [0,1]^N$, with $N \in \N$, we let
\[
I_p({\bm x}):= \sum_{i=1}^N I_p(x_i).
\]
By setting $N:=\binom{n}{2}$, and identifying $\{(i,j): 1 \le i < j \le n\}$ with $\llbracket N \rrbracket := \llbracket 1, N \rrbracket:=\{1,2,\ldots, N\}$ ({{for brevity, throughout the rest of the paper, we continue using the somewhat nonstandard notation $\llbracket N \rrbracket$ for the discrete interval of integers from $1$ to $N$}}), any vector ${\bm x}$ of length $N$ can be associated to a unique weighted simple graph, denoted hereafter by $\Graph[{\bm x}]$. Similar to an unweighted graph one can define the homomorphism count ${\rm Hom}({\sf H}, \Graph[{\bm x}])$ for such weighted graphs. It has been established in \cite{aug,chd,cod,eld} that given a graph of fixed size ${\sf H}$ if $p \gg n^{-\alpha_{\sf H}}$, for certain $\alpha_{\sf H} \in (0,1)$, then the negative of the logarithm of the probability of the upper tail event ${\sf UT}({\sf H}, \delta)$  equals, upon excluding smaller order terms,  
\beq\label{eq:var-prblm-0}
\varphi_{n, {\sf H}}(\delta):=\inf \left\{I_p({\bm x}): {\rm Hom}({\sf H}, \Graph[{\bm x}]) \ge (1+\delta) n^{v_{\sf H}} p^{e({\sf H})} \right\},
\eeq
where we remind the reader that the notation $v_{\sf H}$ and $e({\sf H})$ denote the number of vertices and edges of ${\sf H}$, respectively.

Let  $\gA:=(\ga_{i,j})_{i,j=1}^n$ be the adjacency matrix of $\G(n,p)$. So $\{\ga_{i,j}\}_{i < j}$ are i.i.d.~$\dBer(p)$. Let us introduce one more notation: For any graph $\Graph \subset K_n$, the complete graph on $n$ vertices,  by a slight abuse of notation we write 
\[
\E_\Graph[N({\sf H}, \G(n,p))]:=  \E[N({\sf H},\G(n,p)) |\ga_{i,j} =1, (i,j) \in E(\Graph)].
\]
Equipped with the above notation and upon restricting ${\bm x} \in \{p, 1\}^N$ one can note that the variational problem \eqref{eq:var-prblm-0} transforms to the following variational problem:
\beq\label{eq:var-prblm-labelled}
\Phi_{n,{\sf H}}(\delta):= \min \left\{e(\Graph)\log(1/p): \Graph \subset K_n, \, \E_\Graph[N({\sf H}, \G(n,p))] \ge (1+\delta)n^{v_{\sf H}} p^{e({\sf H})} \right\}.
\eeq
{From \cite[Remark 8.3]{BGLZ} it follows that for $p \ll n^{-1/\Delta}$ the solution to the variational problem in \eqref{eq:var-prblm-labelled} is given by a clique on $\delta ^{1/v_{\sf H}} np^{\Delta/2}$ vertices. 
In this paper we show that the $\log$ of the probability of ${\sf UT}({\sf H}, \delta)$, up to the leading order, equals $\f12 \delta ^{2/v_{\sf H}} n^2p^{\Delta}$. Thereby, with the help of \cite{BGLZ}, we establish that in the entire localized regime the large deviation of the upper tail event ${\sf UT}({\sf H},\cdot)$ (both the speed and the rate function) is determined by the variational problem $\Phi_{n,{\sf H}}(\cdot)$. 

Having stated the variational problem we proceed to describe the idea in showing that the negative of the $\log$ of the probability of ${\sf UT}({\sf H}, \cdot)$ is upper bounded by $\f12 \delta ^{2/v_{\sf H}} n^2p^{\Delta}$ upto a factor of $1+o(1)$.}

The initial part of the proof proceeds as in \cite{hms}. Indeed, using \cite{hms} we show that the probability of ${\sf UT}({\sf H}, \delta)$ is bounded by that of the existence of subgraphs in $\G(n,p)$ which are near-optimizers  of the variational problem in \eqref{eq:var-prblm-labelled}. More precisely, such subgraphs are defined as follows.


\begin{dfn}[Pre-seed graph]\label{dfn:seed-0}
Fix $\vep >0$ sufficiently small. Let $\bar C:=\bar C(\delta,\vep)$ be a sufficiently large constant. A graph $\Graph \subset K_n$ is said to be a pre-seed graph if the followings hold:
\begin{enumerate}
\item[(PS1)] $\E_\Graph[N({\sf H}, \G(n,p))] \ge (1+\delta(1 -\vep)) n^{v_{\sf H}} p^{e({\sf H})}$.
\item[(PS2)] $e(\Graph) \le \bar C n^2 p^\Delta \log(1/p)$.
\end{enumerate}
%
\end{dfn}
The choice of the constant $\bar C$ will be made precise in Section \ref{sec:proof-main-thm} during the course of the proof of Theorem \ref{thm:cycle-main}. Hereafter, we fix a sufficiently small but arbitrary $\vep>0$. With that choice of $\vep$ we will show that the lower bound in \eqref{eq:cycle-ldp} holds with an additional factor $(1-\gf_{\sf H}(\vep))$ on its \abbr{RHS}, for some function $\gf_{\sf H}(\cdot)$ satisfying $\lim_{\vep \downarrow 0} \gf_{\sf H}(\vep)=0$. Thus sending $\vep$ to zero afterwards would yield \eqref{eq:cycle-ldp}.

The upper bound $p \le n^{-1/\Delta -o(1)}$ allows us to consider only those subgraphs for which the property (PS1) of Definition \ref{dfn:seed-0} can be replaced by a simpler condition as stated below.

\begin{dfn}[Seed graph]\label{dfn:seed}
Let $\vep$ and $\bar C$ be as in Definition \ref{dfn:seed-0}. A graph $\Graph \subset K_n$ is said to be a seed graph if the followings hold: 
\begin{enumerate}
\item[(S1)] $N({\sf H}, \Graph) \ge \delta(1 -2\vep) n^{v_{\sf H}} p^{e({\sf H})}$.
\item[(S2)] $e(\Graph) \le \bar C n^2 p^\Delta \log(1/p)$.
\end{enumerate}
\end{dfn}

{We remark that the pre-seed graphs of Definition \ref{dfn:seed-0} is termed as seed graphs in \cite{hms}. Since the assumption $p \le n^{-1/{\Delta}-o(1)}$ allows us to obtain a simpler description of seed graphs of \cite{hms} we have chosen to deviate from the terminology of \cite{hms}.}

We then show that any seed graph must have a subgraph containing most of its copies of ${\sf H}$ such that each of the edges of that subgraph participates in a large number copies of ${\sf H}$ as well. Following \cite{hms} we term these graphs as {\em core graphs}. 

\begin{dfn}[Core graph]\label{dfn:core-graph}
With $\vep$ and $\bar C$ as in Definition \ref{dfn:seed-0} we define a graph $\Graph \subset K_n$ to be a core graph if
\begin{enumerate}
\item[(C1)] $ N({\sf H}, \Graph) \ge \delta  (1-3 \vep) n^{v_{\sf H}} p^{e({\sf H})}$,
\item[(C2)] $e(\Graph) \le \bar C n^2 p^\Delta \log(1/p)$,
\end{enumerate}
and
\begin{enumerate}[resume]
\item[(C3)] \(\min_{e \in E(\Graph)} N({\sf H}, \Graph, e) \ge \delta \vep n^{v_{\sf H}} p^{e({\sf H})} /(\bar C n^2 p^\Delta \log(1/p))\),
\end{enumerate}
where for an edge $e \in E(\Graph)$ the notation $N({\sf H}, \Graph,e)$ denotes the number of labelled copies of ${\sf H}$ in $\Graph$ that contain the edge $e$.
\end{dfn}

Let us remind the reader that \cite{hms} shows that the probability of the upper tail event can be bounded, upto a factor of two, by that of the existence of core graphs. Then \cite{hms} proceeds by finding bounds of core graphs (recall \eqref{eq:gN-e}) and thereafter the union bound yields an adequate bound on the existence of a core subgraph in $\G(n,p)$. As already mentioned in Section \ref{sec:intro}, for $p \le n^{-1/\Delta -o(1)}$ the number of core graphs is too huge to follow the above route. So, to find a bound on the probability of the existence of a core subgraph in $\G(n,p)$ we find nets of such graphs so that the union bounds produces an effective bound.  

To carry out this scheme we split the set of core graphs into two subsets: graphs with $e(\Graph) = O(n^2 p^\Delta)$ and $e(\Graph) \gg n^2 p^\Delta$. We note that for a core graph $\Graph$ with $e(\Graph)=O(n^2 p^\Delta)$ one can find a further subgraph of it so that every edge of that subgraph participates in an even larger number of copies of ${\sf H}$. These subgraphs of $K_n$ will be termed as {\em strong-core} graphs. 

\begin{dfn}[Strong-core graph]\label{dfn:strong-core}
Let $\vep$ be as in Definition \ref{dfn:seed} and $\bar C_\star:= \bar C_\star(\delta) < \infty$ be a large constant, depending only on $\delta$.  We define a graph $\Graph \subset K_n$ to be a strong-core graph if 
\begin{enumerate}
\item[(SC1)] $ N({\sf H}, \Graph) \ge \delta  (1-{6}\vep) n^{v_{\sf H}} p^{e({\sf H})}$,
\item[(SC2)] $e(\Graph) \le \bar C_\star n^2 p^\Delta$,
\end{enumerate}
and
\begin{enumerate}[resume]
\item[(SC3)] \(\min_{e \in E(\Graph)} N({\sf H}, \Graph, e) \ge (\delta\vep/ \bar C_\star) \cdot (np^{\Delta/2})^{v_{\sf H}-2}\).
\end{enumerate}
\end{dfn}

Notice the difference in the lower bounds in (C3) and (SC3) in Definitions \ref{dfn:core-graph} and \ref{dfn:strong-core}, respectively. We will see below that any $\bar C_\star = C_\star \delta^{2/{v_{\sf H}}}$ with $C_\star \ge 32$ will suffice for the proof of Theorem \ref{thm:cycle-main}.  

To treat the core graphs with edges at most $\bar C_\star n^2 p^\Delta$ our approach will be to bound the number of strong-core graphs with a given number of edges. We illustrate the idea behind this step by demonstrating bounds on certain specific collections of strong-core graphs for ${\sf H}=C_4$.

Consider 
$K_{2, C n^2 p^2}$ with $C \ge \wh C_\delta:= \f12 \delta^{1/2}$. It can be verified that for such choices of $C$ the graph $K_{2, C n^2 p^2}$ is indeed a strong-core graph when ${\sf H}= C_4$. On the other hand it is easy to see that any clique $K_{\lceil \delta^{1/2} np\rceil}$ is also a strong-core graph. Same reasoning yields that so is the disjoint union of $K_{2, \wh{C}_{\delta_1} n^2 p^2}$ and $K_{\lceil \delta_2^{1/2} np \rceil}$, to be denoted by ${\sf K}_{\delta_1,\delta_2}$, whenever $\delta_1, \delta_2 \ge 0$ are such that $\delta_1+\delta_2 \ge \delta$. Using Stirling's approximation it is immediate to deduce that, for $\delta_1, \delta_2$ as above, the $\log $ probability of the existence of a copy of ${\sf K}_{\delta_1, \delta_2}$ in $\G(n,p)$ is bounded by that of ${\sf UT}(C_4, \delta)$ upon excluding a negligible factor. This hints that if every strong-core graph can be decomposed as a {\em disjoint} union of a {\em complete bipartite like graph} and a {\em clique like graph} then one may be able to repeat the above argument. This observation is the motivation behind the argument that follows. 

Indeed, we show that there exists a bipartite subgraph $ \Graph_\star \subset \Graph$ with vertices $\mathscr{V} \cup \tilde{\mathscr{V}}$ that have a block path structure as shown in Figure \ref{fig:path-str}, where $\mathscr{V}:=\{V_i\}_{i=1}^{\wt C_3}$ and $\tilde{ \mathscr{V}:}=\{\widetilde V_i\}_{i=1}^{\wt C_3}$ are the two {partite sets} (i.e.~the maximal independent sets), and $\wt C_3$ is an arbitrarily large constant. As the block path structure of Figure \ref{fig:path-str} is of arbitrarily large length one can use (reverse) pigeonhole principle to deduce there indeed exists a subgraph $\Graph_\gb\subset \Graph_\star$ such that almost all copies of ${\sf H}$ are either contained $\Graph_\gb$ or in $\Graph\setminus \Graph_\gb$. That is, almost no copy of ${\sf H}$ use edges from both $\Graph_\gb$ and $\Graph\setminus \Graph_\gb$. 
Heuristically, the reader may view the decomposition of $\Graph$ into $\Graph_\gb$ and $\Graph\setminus \Graph_\gb$ as a separation of copies of $K_{2, \wh C_{\delta_1} n^2 p^2}$ and $K_{\lceil \delta_2^{1/2} n p \rceil}$ that may be present in a strong-core graph, {when ${\sf H}=C_4$. For a general ${\sf H}$ this decomposition of $\Graph$ can be thought of a separation of its complete bipartite like subgraphs and clique like subgraphs.} 

\begin{figure}
\centering
\includegraphics[width=9.5cm, height=4cm]{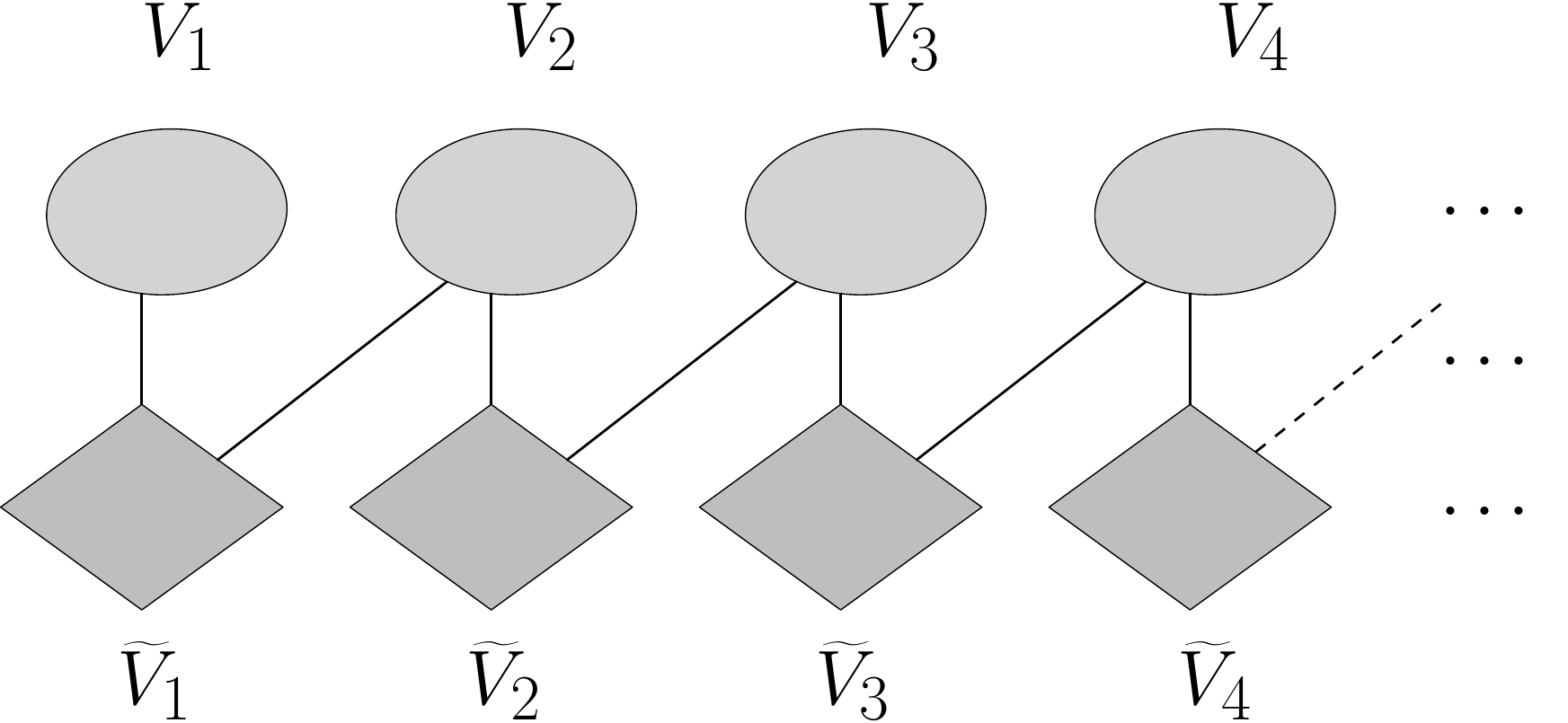}
\caption{{Schematic representation of the block path structure of the bipartite subgraph $\Graph_\star$ in a strong-core graph $\Graph$. Vertices in $V_1$ are only connected to that in $\wt V_1$. Vertices in $V_2$ can only be connected to that in $\wt V_1$ and $\wt V_2$, and so on.}}
\label{fig:path-str}
\end{figure}

Equipped with this observation the desired bound on the probability of the existence of a strong-core subgraph of $\G(n,p)$ follows once we show that $\Graph_\gb$ and $\Graph\setminus \Graph_\gb$ are individually entropically stable. This is done using combinatorial arguments. The details of this step is carried out in Section \ref{sec:strong-core}.

Now it remains to consider core graphs with at least $\bar C_\star n^2 p^\Delta$ edges. In this case, if we additionally assume that $np^{\Delta/2} \ge (\log n)^{v_{\sf H}}$ then one has a lower bound on the product of the degrees of the end points of any edge in a core graph (see Lemma \ref{lem:bad-graph-bd}) which in turn implies that the subgraph of $\Graph$ induced by edges that are incident to at least one vertex of low degree must be bipartite. Observe that as ${\sf H}$ is a $\Delta$-regular graph and by definition every edge of a core graph must participate in at least one copy of ${\sf H}$. Therefore the minimum degree in a core graph must be at least $\Delta$. These last two facts together allow one to derive a bound on the number of such graphs adequate for a union bound showing their entropic {suboptimality}. 

As already explained in Section \ref{sec:intro} one may not have the above bipartite structure in a core graph when $np^{\Delta/2} \le (\log n)^{v_{\sf H}}$. To derive the entropic stability in this case we carry out a {delicate} chaining type argument which is the second key idea of this paper. This chaining argument allows us to find an optimal net for core graphs with large number of edges when the edge probability is near the critical threshold. During this process we also need to discard several sub events that are entropically nonviable. First let us describe the heuristic behind this step. 

For a graph $\Graph$ let us denote $E_{1,1}(\Graph)$ to be the subset of its edges consisting of edges with both end points of {\em low} degree. Notice that this subset of edges is possibly now non empty. Consider the sub event that there exists a core graph $\Graph$ with a large number of $\cN_{1,1}({\sf H}, \Graph)$, the number of copies of ${\sf H}$ in $\Graph$ that use at least one edge belonging to $E_{1,1}(\Graph)$. Using a combinatorial argument (see Proposition \ref{l:seqcounting}) we deduce that this scenario can only happen if there are {\em many} edges in $\Graph$ whose at least one end point is of {\em high} degree. Using the latter observation one can then deduce that this scenario is entropically non-viable. 

This allows us to move our focus to the complement event, i.e.~ the case when core graphs with edges at least $\bar C_\star n^2 p^\Delta$ have a small $\cN_{1,1}({\sf H}, \Graph)$. In this case the most natural strategy would be to delete all edges in $E_{1,1}(\Graph)$ so that in the resulting graph the subgraph induced by edges that are incident to vertices of low degree is now bipartite and thus one may hope repeat the argument of the case $np ^{\Delta/2} \ge (\log n)^{v_{\sf H}}$. However, the resulting graph may possess vertices of degree strictly less than $\Delta$. So one can further iteratively remove vertices of degree strcitly less than $\Delta$ to obtain a subgraph $\Graph_0:=\Graph_0(\Graph) \subset \Graph$ that has minimum degree $\Delta$. As $\cN_{1,1}({\sf H}, \Graph)$ is small and ${\sf H}$ is a $\Delta$-regular graph this entire process of removal of edges results in a loss of only small fraction of copies of ${\sf H}$ in $\Graph$. If the resulting subgraph $\Graph_0$ has a large $E(\Graph_0)\setminus E_{1,1}(\Graph_0)$ then one can again argue using another combinatorial argument (see Lemma \ref{l:core2jb}) to derive entropic sub optimality of this case.
However, if $E_{1,1}(\Graph_0)$ is large then one has to repeat the entire procedure described above. That is, one need to investigate whether $\cN_{1,1}({\sf H}, \Graph_0)$ is large or small and proceed accordingly until all subsequent sub cases are exhausted. 

To build a proof out of the above heuristic we proceed as follows: We break the discrete interval $\llbracket \bar C_\star n^2 p^\Delta, \bar C n^2 p^\Delta \log(1/p) \rrbracket$, the allowable range of the number of edges of core graphs to be considered, into dyadic subintervals $\{\wt\cJ_j\}$, where $\wt\cJ_j:= \llbracket 2^{-j} \bar C n^2 p^\Delta \log(1/p), 2^{-(j-1)} \bar C n^2 p^\Delta \log(1/p)\rrbracket$. For $\Graph$ with $e(\Graph) \in \wt\cJ_j$ the criterion that $E(\Graph_0)\setminus E_{1,1}(\Graph_0)$ is large is then characterized by the event that $e(\Graph_0) \in \wt \cJ_j$. As mentioned above, this case can be shown to be entropically non-viable using combinatorial arguments. Consider the other case, i.e.~when $e(\Graph_0) \in \cup_{j' \ge j+1}\wt \cJ_{j'}$. Since in this case $e(\Graph_0) \le 2^{-j} \bar C n^2 p^\Delta \log(1/p)$ it can be treated similarly to the case of the core graphs with edges at most $2^{-j} \bar C n^2 p^\Delta \log(1/p)$. Therefore, starting with $j=1$ and using the ideas explained  in the previous two paragraphs one can proceed iteratively to establish the entropic stability of the core graphs with at least $\bar C_\star n^2 p^\Delta$ edges. This iterative procedure may remind the reader of Kolmogorov's chaining technique of approximating a space by a nested sequence of nets with meshes that become increasingly coarse. 

Let us also add that during each step of the removal of the edges from $\Graph$ to procure the subgraph $\Graph_0(\Graph)$ we lose a small fraction of the number of copies of ${\sf H}$ in $\Graph$. To ensure that we do not lose a lot of copies ${\sf H}$ during the entire chaining procedure the threshold for {$\cN_{1,1}({\sf H}, \Graph)$} to be declared to be large has to be adaptive. That is, this threshold must depend on the index $j$ for which $e(\Graph) \in \cJ_j$. The larger the index $j$ is the smaller the threshold becomes. This adaptive threshold worsens the lower bound on the number of edges with at least one end point of high degree (see Proposition \ref{l:seqcounting} again), resulting another technical difficulty. It prevents us to carry out the chaining argument for all graphs having at least $\bar C_\star n^2 p^\Delta$. To overcome this obstacle we apply the above chaining procedure only for core graphs that have at least $\bar C n^2 p^\Delta \log(1/p)/(\log \log n)^2$ edges. For the remaining case we use a different argument which is essentially a combination of two key ideas that are described above. This finally completes the outline of the proof of Theorem \ref{thm:cycle-main}. 

For convenience to the reader we have included a schematic diagram of the main steps of the proof the main result (see Figure \ref{fig:proof-main}), as well as a schematic representation of the chaining argument in Figure \ref{fig:chaining}.

\begin{center}
\begin{figure}
\centering
\begin{tikzpicture}[node distance = 3.3cm, auto]
    \node [block] (seed) {{\scriptsize Seed graphs}\\
    {\scriptsize $e(\Graph)\leq \bar{C}n^2p^\Delta\log(1/p)$}\\
    {\scriptsize many copies of ${\sf H}$}};
    \node [block, below of=seed, node distance= 2 cm] (core) {{\scriptsize Core graphs}\\
    {\scriptsize each edge in many copies of ${\sf H}$}};
    \node [block, below left of= core] (strongcore) {{\scriptsize Strong-core graphs}\\
    {\scriptsize $e(\Graph)\leq \bar{C}_{\star}n^2p^\Delta$}};
    \node [block, below right of=core] (weakcore) {{\scriptsize Core graphs with}\\
    {\scriptsize $e(\Graph)\geq \bar{C}_{\star}n^2p^\Delta$}};
    \node [block, below left of= strongcore, node distance= 3cm] (bd-strongcore-1) {{\scriptsize Existence of a }\\ {\scriptsize bipartite subgraph $\Graph_\star$}};
    \node [block, below of= bd-strongcore-1, node distance= 2cm] (bd-strongcore-2) {{\scriptsize Both $\Graph_\star$ and}\\ 
    {\scriptsize $\Graph\setminus \Graph_\star$ entropically stable}\\
    {\scriptsize Proposition \ref{prop:strong-core}}};
       \node [cloud, below right of=weakcore, node distance=4 cm] (largep) {$np^{\Delta/2}\geq (\log n)^{v_{\sf H}}$};
     \node [cloud, below left of=weakcore, node distance=4.3 cm] (smallp) {$np^{\Delta/2}\leq (\log n)^{v_{\sf H}}$};
      \node [block, below of=largep, node distance=2 cm] (subopt-large-1) {{\scriptsize Existence of}\\
    {\scriptsize a bipartite subgraph}};
      \node [block, below of=subopt-large-1, node distance=2 cm] (subopt-large-2) {{\scriptsize Entropically  non-viable}};
     \node [block, below of=bd-strongcore-2, node distance=3 cm] (core1) {{\scriptsize $e(\Graph) \leq \frac{\bar{C}n^2p^\Delta\log(1/p)}{(\log \log n)^2}$}};
    \node [block, below of=smallp, node distance=2.8 cm] (core2) {{\scriptsize $e(\Graph) \geq \frac{\bar{C}n^2p^\Delta\log(1/p)}{(\log \log n)^2}$}};
    \node [block, below of=core1, node distance=2 cm] (subopt-core1) {{\scriptsize Entropic suboptimality}\\
    {\scriptsize Proposition \ref{p:core--1} } };
    \node [block, below right of=core2, node distance=3.5 cm] (subopt-core2) {{\scriptsize Chaining argument}\\
    {\scriptsize Proposition \ref{p:core2}}};
    \path [line] (seed) -- (core);
    \path [line] (core) -- (strongcore);
    \path [line] (core) -- (weakcore);
     \path [line] (strongcore) -- (bd-strongcore-1);
      \path [line] (bd-strongcore-1) -- (bd-strongcore-2);
    \path [line] (weakcore) -- (largep);
    \path [line] (weakcore) -- (smallp);
    \path [line] (largep) -- (subopt-large-1);
     \path [line] (subopt-large-1) -- (subopt-large-2);
     \path [line] (smallp) -- (core1);
     \path [line] (smallp) -- (core2);
     \path [line] (core1) -- (subopt-core1);
     \path [line] (core2) -- (subopt-core2);
\end{tikzpicture}
\caption{Schematic representation of the mains steps in the proof of Theorem \ref{thm:cycle-main}.}
\label{fig:proof-main}
\end{figure}
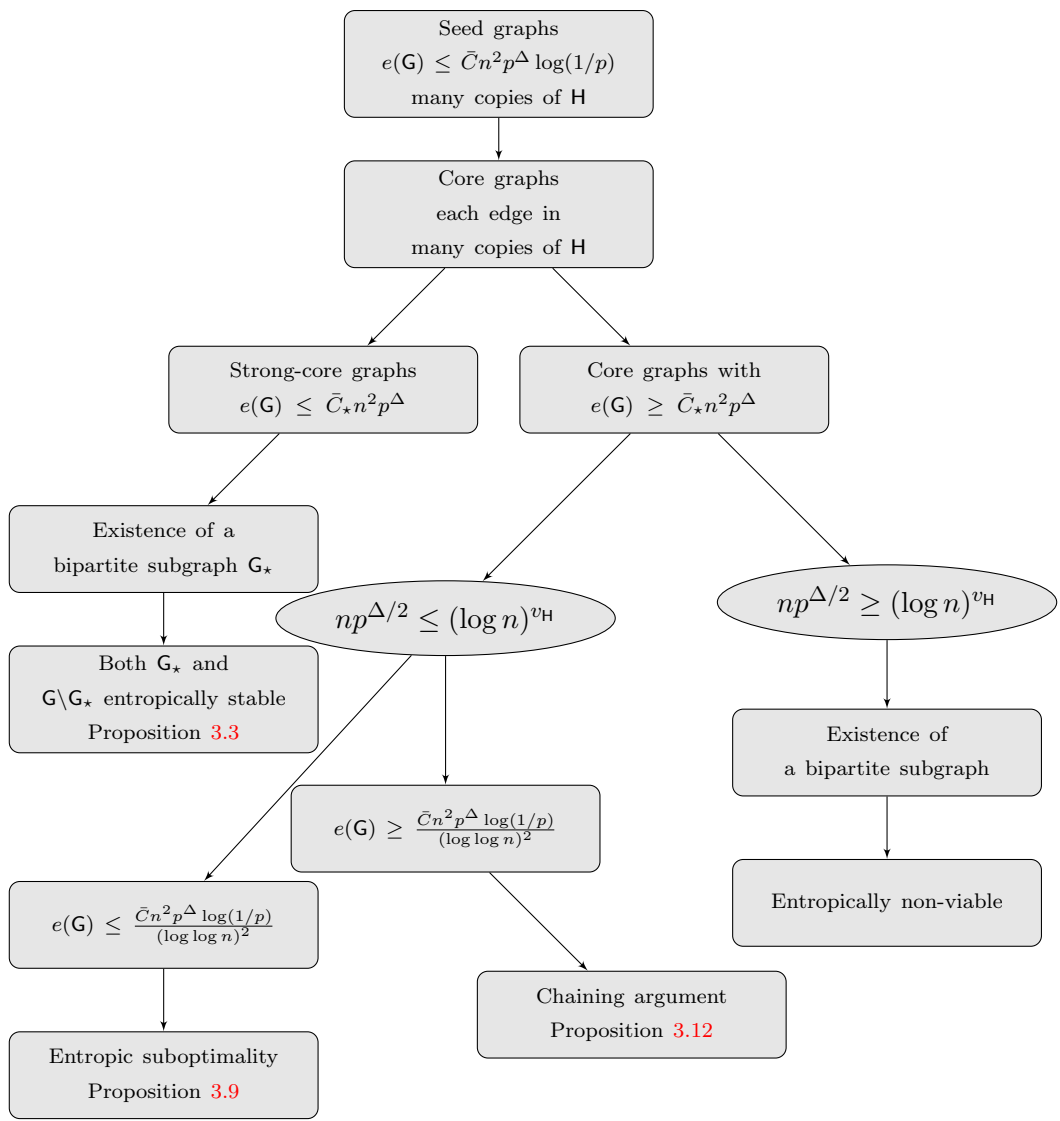
\end{center}

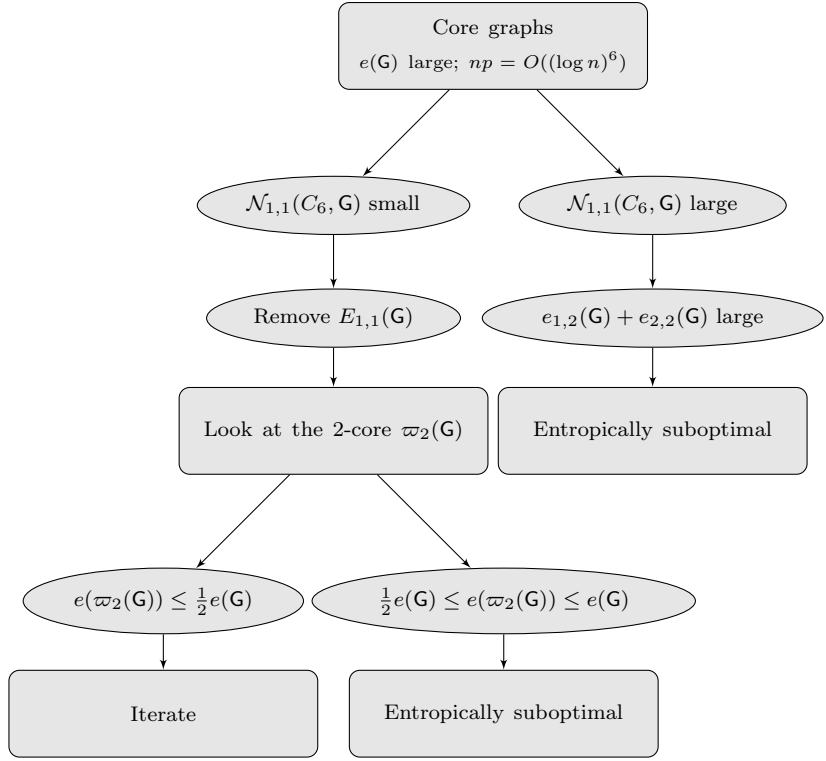
\begin{figure}
\centering
\begin{tikzpicture}[node distance = 3.3cm, auto]
    \node [block] (seed) {{\scriptsize Core graphs}\\
    {\tiny $e({\sf G})$ large; $ np = O((\log n)^6)$}};
    \node [cloud, below right of=seed, node distance= 3 cm] (core) {\scriptsize{$\cN_{1,1}(C_6, {\sf G})$ large}};
     \node [cloud, below left of=seed, node distance= 3 cm] (coren) {\scriptsize{$\cN_{1,1}(C_6, {\sf G})$ small}};
    \node [cloud, below of=core, node distance= 1.5cm] (weakcore) {{\scriptsize $e_{1,2}({\sf G}) + e_{2,2}({\sf G})$ large}};
     \node [cloud, below of=coren, node distance= 1.5cm] (weakcoren) {{\scriptsize Remove $E_{1,1}({\sf G})$}};
      \node [block, below of=weakcore, node distance= 1.5cm] (weakcorea) {{\scriptsize Entropically suboptimal}};
      \node [block, below of=weakcoren, node distance= 1.5cm] (weakcoreb) {{\scriptsize Look at the $2$-core $\varpi_2({\sf G})$}};
        \node [cloud, below left of=weakcoreb, node distance= 3.2cm] (weakcorec) {{\scriptsize $e(\varpi_2({\sf G})) \le \f12 e({\sf G})$}};
          \node [cloud, below right of=weakcoreb, node distance= 3.2cm] (weakcored) {{\scriptsize $\f12 e({\sf G})\le e(\varpi_2({\sf G})) \le e({\sf G})$}};
            \node [block, below of=weakcored, node distance= 1.5cm] (weakcoref) {{\scriptsize Entropically suboptimal}};
              \node [block, below of=weakcorec, node distance= 1.5cm] (weakcoree) {{\scriptsize Iterate}};
    \path [line] (seed) -- (core);
    \path [line] (seed) -- (coren);
    \path [line] (core) -- (weakcore);
     \path [line] (weakcore) -- (weakcorea);
     \path [line] (coren) -- (weakcoren);
      \path [line] (weakcoren) -- (weakcoreb);
       \path [line] (weakcoreb) -- (weakcorec);
        \path [line] (weakcoreb) -- (weakcored);
         \path [line] (weakcorec) -- (weakcoree);
          \path [line] (weakcored) -- (weakcoref);
\end{tikzpicture}
\caption{Schematic representation of the chaining argument for ${\sf H}=C_6$.
}
\label{fig:chaining}
\end{figure}

We remind the reader that our Theorem \ref{thm:cycle-main} does not discuss the upper tail large deviations when $p \sim n^{-1/\Delta}$. As mentioned in Remark \ref{rmk:rate-fn}, thanks to \cite{hms}, one only needs to identify the rate function for such $p$. In the remark below we provide a short outline of the derivation of the rate function for $C_4$.  

\begin{rmk}\label{rmk:rate-fn-c4}
\corAB{
To describe the rate function for $p \sim n^{-1/2}$ and ${\sf H}=C_4$ let us introduce a few notation. For any $U \subset V(\Graph)$ we write $\Graph[U]$ to be the graph spanned by the vertices in $U$ and $\Graph[U, \bar U]$ to be the bipartite subgraph of $\Graph$ induced by the two disjoint subsets of vertices $U$ and $\bar U:=V(\Graph)\setminus U$. We let $N_{U}(K_{1,2}, \Graph)$ to denote the number of labelled copies of $K_{1,2}$ in $\Graph$ for which the center vertex is in $U$ and the leaf vertices are in $\bar U$.

It can be checked that the near-optimizers of the variational problem \eqref{eq:var-prblm-labelled} are the graphs $\Graph$ for which there exists a partition of $V(\Graph)={\sf V}_1 \cup {\sf V}_2$ such that  
\beq\label{eq:rate-fn-1}
N(C_4, \Graph[{\sf V}_1]) \ge \delta (1-\vep) \cdot x \cdot n^4 p^4 
\eeq
and
\beq\label{eq:rate-fn-2}
 \upmu(\kappa):=N(C_4, \Graph[{\sf V}_1, {\sf V}_2])+ 4\kappa N_{{\sf V}_2}(K_{1,2}, \Graph) \ge  \delta(1-\vep) \cdot (1-x) \cdot n^4 p^4,
\eeq
for some $x \in [0,1]$, and any $\vep >0$ sufficiently small, where $\kappa:=\lim_{n \to \infty} n p^2$.  Thus the near minimizers are the graphs with the minimal number of edges satisfying \eqref{eq:rate-fn-1}-\eqref{eq:rate-fn-2}.

To identify such graphs we note the following: If $\deg_\Graph(v)$ denotes the degree of vertex $v$ in graph $\Graph$ then $N(C_4, \Graph[{\sf V}_1, {\sf V}_2])$ is maximized, among all graphs with a given number of edges, if the differences between pairs $\deg_\Graph(v_i)$ and $\deg_\Graph(v_j)$ are minimized for every pair of vertices $v_i, v_j \in {\sf V}_1$. On the other hand, $N_{{\sf V}_2}(K_{1,2}, \Graph)$ is maximized when $\deg_\Graph(v)$ are set to be the maximum value subject to the natural constraint that $\deg_\Graph(v) \le |{\sf V}_2|$, where $|\cdot|$ denotes the cardinality of a set. Due to convexity it can be further deduced that the maximizer for $\upmu(\kappa)$, among all graphs with a given number of edges, has to be one of the two maximizers described above. However, which one of them will dominate depends on $\kappa$. 

Using this observation and minimizing over $x \in [0,1]$ one can find the near minimizers and hence also the rate function. {For a general $C_\ell$} this picture is more intricate. We refrain from fleshing out the detail for the general case. This may be considered elsewhere.}
\end{rmk}

\subsection*{{Outline of the rest of the paper} } In Section \ref{sec:proof-main-thm} we provide a proof of Theorem \ref{thm:cycle-main} assuming that we have the necessary bounds on the probabilities of various subevents of the event that $\G(n,p)$ contains a core graph. 
The proofs of these bounds being combinatorial in nature are pushed to Sections \ref{sec:strong-core}, \ref{sec:core-many-edges}, and \ref{s:comb-reg}. In Section \ref{sec:strong-core-odd} we provide a short proof showing the entropic stability of strong-core graphs for non-bipartite graphs. While Sections \ref{sec:G_b} and \ref{sec:G-G_b} are devoted in deriving bounds that enable us in proving entropic stability of $\Graph_\gb$ and $\Graph\setminus \Graph_\gb$, respectively. Combining results from these two sections we then, in Section \ref{sec:proof-sc-even}, finish the proof of the entropic stability of strong-core graphs in the case of bipartite graphs. 

In Section \ref{s:bipartite} we treat core graphs with large number of edges when $np^{\Delta/2} \ge (\log n)^{v_{\sf H}}$. The next two sections provide proofs of probability bounds necessary for carrying out the chaining argument. In particular, in Section \ref{sec:combin-graph-bd} we derive bounds showing entropic stability of core graphs for which the number of edges with at least one of its end points is of high degree, is large. Whereas in Section \ref{sec:largecN-11} using a combinatorial result (see Proposition \ref{l:seqcounting}) we derive entropic sub optimality of the remaining case. In Section \ref{s:comb-reg} we prove Proposition \ref{l:seqcounting}. This uses a graph decomposition result whose its proof deferred to Appendix \ref{sec:decompose}. 
Finally, in Appendix \ref{app:app1} we derive bounds on the product of the degrees of the end points of most of the edges in a strong-core graph that plays an important role in showing that such graphs are entropically stable.

\section{Proof of Theorem \ref{thm:cycle-main}}\label{sec:proof-main-thm}
In this section we provide the proof of our main result Theorem \ref{thm:cycle-main}. As already mentioned, we will only focus on the case $p\le n^{-1/\Delta-o(1)}$, {where $o(1)$ is an appropriately chosen term decaying to zero as $n \to \infty$} since the other case is proved in \cite{hms}. First let us state the result showing the upper bound in \eqref{eq:cycle-ldp}. 
\begin{prop}
\label{p:lower}
Fix $\delta>0$ and a $\Delta$-regular connected graph ${\sf H}$. For $p=p_n\in (0,1)$ such that $p\ll n^{-1/\Delta}$ we have
\[
\limsup_{n\to \infty} -\dfrac{\log \P\left(N({\sf H}, \G(n,p)) \ge (1+\delta)n^{v_{\sf H}} p^{e({\sf H})}\right)}{n^{2}p^\Delta\log(1/p)}\leq \frac{1}{2}\delta^{2/v_{\sf H}}.
\]
\end{prop}

The proof of Proposition \ref{p:lower} is standard. We include it for completeness. 

\begin{proof}[Proof of Proposition \ref{p:lower}]
We aim to apply \cite[Lemma 3.5]{hms}.  As the edges in $\G(n,p)$ are independent and each edge occurs with probability $p$ we find that for any $\Graph \subset K_n$,
\beq\label{eq:cond-exp-ubd}
\E_\Graph[N({\sf H}, \G(n,p))] - \E [N({\sf H}, \G(n,p))] = \sum_{ \emptyset \ne {\sf H}_\star \subset {\sf H}} N({\sf H}_\star, \Graph) \cdot (1-p^{e({\sf H}_\star)}) \cdot n^{v_{\sf H}- v_{{\sf H}_\star}} p^{e({\sf H}) - {e}({{\sf H}_\star})},
\eeq
where the sum is taken over all subgraphs ${\sf H}_\star$ of ${\sf H}$ with no isolated vertices, and as before $v_{{\sf H}_\star}$ and $e({{\sf H}_\star})$ denote the number of vertices and edges of ${\sf H}_\star$, respectively. This, in particular, implies that
\beq\label{eq:ps-to-s-00}
\E_\Graph[N({\sf H}, \G(n,p))] - \E [N({\sf H}, \G(n,p))] \ge N({\sf H}, \Graph)\cdot (1-p^{e({\sf H})}).
\eeq
We also observe that any clique on $m$ vertices contains $(m)_{v_{\sf H}}:=m (m-1) \cdots (m-v_{\sf H}+1)$ labelled copies of ${\sf H}$. Fix $\wt \vep >0$. Taking $\Graph$ to be the clique on  $\lceil (\delta + 2\wt \vep)^{1/v_{\sf H}} np^{\Delta/2} \rceil$ vertices in \eqref{eq:ps-to-s-00} we see that the \abbr{LHS} of \eqref{eq:ps-to-s-00} is bounded below by 
\(
(\delta+\f32\wt \vep) n^{v_{\sf H}}p^{e({\sf H})},
\)
for all large $n$, where we use the fact that ${\sf H}$ being a $\Delta$-regular graphs we have $e({\sf H}) = (\Delta/ 2) \cdot  v_{\sf H}$. As $\E[N({\sf H}, \G(n,p))]= n^{v_{\sf H}} p^{e({\sf H})}(1+o(1))$, recalling the definition of the variational problem $\Phi_{n,{\sf H}}(\delta)$ from \eqref{eq:var-prblm-labelled}, we therefore deduce that for any $\wt \vep >0$ and $n$ sufficiently large,
\beq\label{eq:phi-ubd}
\Phi_{n, {\sf H}}(\delta +\wt \vep) \le \f12 (\delta + 3\wt \vep)^{2/v_{\sf H}}n^2 p^\Delta \log (1/p),
\eeq
where we also used the fact that $p=o(1)$. Now we apply \cite[Lemma 3.5]{hms} 
to deduce
\[
\limsup_{n \to \infty}  -\dfrac{\log \P\left(N({\sf H}, \G(n,p)) \ge (1+\delta)n^{v_{\sf H}} p^{e({\sf H})}\right)}{n^{2}p^\Delta\log(1/p)} \le \lim_{\wt \vep \downarrow 0} \limsup_{n \to \infty} \dfrac{\Phi_{n, {\sf H}}(\delta+\wt \vep)}{n^2 p^\Delta \log (1/p)} \le \f12 \delta^{2/v_{\sf H}},
\]
where the rightmost inequality is due to \eqref{eq:phi-ubd}. This completes the proof. 
 \end{proof}

Let us now move to the proof of the upper bound which takes up the rest of the paper. As outlined in Section \ref{sec:outline} the initial step in this direction is to show that the probability of ${\sf UT}({\sf H}, \delta)$ can be  bounded above by that of existence of a core graph as defined in Definition \ref{dfn:core-graph}. 
\begin{lem}
\label{l:seed-markov}
Fix $\delta >0$ and a $\Delta$-regular connected graph ${\sf H}$ with $\Delta \ge 2$. If $np^{\Delta}\ll (\log n)^{-\Delta v_{\sf H}}$ then for every $\vep>0$ and all large $n$ we have
\beq\label{eq:ps-to-s}
\P(N({\sf H}, \G(n,p)) \ge (1+\delta)n^{v_{\sf H}} p^{e({\sf H})}) \leq (1+\vep)\P( \G(n,p)~\text{contains a core graph}).
\eeq
\end{lem}

The proof of Lemma \ref{l:seed-markov} will use \cite{hms} and the bound $np^{\Delta/2} \le n^{1/2-o(1)}$. Let us remark that in \cite[Lemma 3.6, Lemma 3.7]{hms}, a similar statement was established for a slightly different definition of core graphs (that corresponds to replacing $N({\sf H}, \Graph)$ and $N({\sf H}, \Graph, e)$ in (C1) and (C3) of Definition \ref{dfn:core-graph} by $\E_\Graph[N({\sf H}, \G(n,p)] - \E[N({\sf H}, \G(n,p))]$ and $\E_\Graph[N({\sf H}, \G(n,p))] - \E_{\Graph\setminus \{e\}}[N({\sf H}, \G(n,p)]$, respectively).  

\begin{proof}[Proof of Lemma \ref{l:seed-markov}]
We first claim that 
\beq\label{eq:claim-pre-seed}
\P(N({\sf H}, \G(n,p)) \ge (1+\delta)n^{v_{\sf H}} p^{e({\sf H})}) \leq (1+\vep)\P( \G(n,p)~\text{contains a pre-seed graph}).
\eeq
To this end, we apply \cite[Lemma 3.6]{hms} with $N({\sf H}, \G(n,p))$ and $\bar C n^2 p^\Delta \log(1/p)$ taking the roles of $X$ and $\ell$ there, respectively. Applying \cite[Lemma 3.6]{hms} with these choices from the definition of pre-seed graphs one has that
\begin{multline}\label{eq:ps-to-s-1}
\P\left(\left\{N({\sf H}, \G(n,p)) \ge (1+\delta)n^{v_{\sf H}} p^{e({\sf H})} \right\}\cap \left\{\G(n,p) \text{ contains a pre-seed graph}\right\}^c\right)\\
\le \left(\f{1+\delta(1-\vep)}{1+\delta}\right)^{\bar C n^2 p^\Delta \log(1/p)}.
\end{multline}
On the other hand, as a {clique} on $m$ vertices contains $(m)_{v_{\sf H}}$ labelled copies of ${\sf H}$ it is immediate from Definition \ref{dfn:seed-0} and \eqref{eq:ps-to-s-00} that, for all large $n$, 
\begin{align}
  \P(\G(n,p) \text{ contains a pre-seed graph})  \ge & \, \P(\G(n,p) \text{ contains a clique on }\lceil \left(\delta\left(1-\vep/2\right)\right)^{1/v_{\sf H}} n p^{\Delta/2}\rceil \text{ vertices})\notag \\
\ge & \, \exp\left( - \f12 \left(\delta\left(1-\dfrac\vep4\right)\right)^{2/v_{\sf H}} n^2 p^\Delta \log (1/p)\right). \label{eq:ps-to-s-2}
\end{align}
Therefore, for $\bar C$ sufficiently large, depending only on $\delta$ and $\vep$, we deduce from \eqref{eq:ps-to-s-1}-\eqref{eq:ps-to-s-2} that the \abbr{LHS} of \eqref{eq:ps-to-s-1} is at most
\[
\vep \cdot \P(\G(n,p) \text{ contains a pre-seed graph}).
\]
This proves the claim \eqref{eq:claim-pre-seed}.  

We now proceed to prove that if $np^\Delta \ll (\log n)^{-\Delta v_{\sf H}}$ then for all sufficiently large $n$ the existence of a pre-seed subgraph of $\G(n,p)$ guarantees the existence of a seed subgraph. Once we have a seed graph $\Graph$ then we peel off its edges iteratively that participate in strictly less than $\delta \vep n^{v_{\sf H}} p^{e({\sf H})} / (\bar C n^2 p^\Delta \log(1/p))$ labelled copies of ${\sf H}$ to produce a subgraph $\Graph_0 \subset \Graph$ so that 
\[
\min_{{\bm e} \in E(\Graph_0)} N({\sf H}, \Graph_0, {\bm e}) \ge \delta \vep n^{v_{\sf H}} p^{e({\sf H})} / (\bar C n^2 p^\Delta \log(1/p)). 
\]
Note that, as $e(\Graph) \le \bar C n^2 p^\Delta \log(1/p)$, by triangle inequality it follows that this peeling procedure loses at most $\delta \vep n^{v_{\sf H}} p^{e({\sf H})}$ labelled copies of ${\sf H}$ in $\Graph$. Thus $\Graph_0$ satisfies the condition (C1) of Definition \ref{dfn:core-graph}. The condition (C2) is automatic. Therefore $\Graph_0$ is indeed a core graph.\footnote{This peeling procedure is similar to that in \cite[Lemma 3.7]{hms}. We refer the reader to there for more details on this step.} This then yields the desired conclusion.  

So it now remains to show that the existence of a pre-seed graph implies the same for a seed graph. 
Turning to do this task we use \cite[Theorem 5.7]{hms} to find that 
\beq\label{eq:copy-H-star-bd}
N({\sf H}_\star, \Graph) \le (2 e(\Graph))^{ \upalpha_{{\sf H}_\star}^\star},
\eeq
for any ${\sf H}_\star \subset {\sf H}$ without any isolated vertices, where $ \upalpha_{{\sf H}_\star}^\star$ is the fractional independence number\footnote{A fractional independent set in a graph ${\sf J}$ is an assignment $\upalpha: V({\sf J}) \mapsto [0,1]$ such that $\upalpha_u+\upalpha_v \le 1$ for all edges $e=(u,v) \in E({\sf J})$. The fractional independence number of ${\sf J}$, denoted by $\upalpha_J^\star$, is the fractional independence set $\upalpha$ that maximizes $\sum_{v \in V({\sf J})} \upalpha_v$.} of ${\sf H}_\star$. As ${\sf H}_\star$ is a subgraph of the connected $\Delta$-regular graph ${\sf H}$ from \cite[Lemma 5.3]{hms} we also have that
\beq\label{eq:frac-ind-bd}
\upalpha_{{\sf H}_\star}^\star \le v_{{\sf H}_\star} - e({\sf H}_\star)/\Delta
\eeq
For a proper subgraph ${\sf H}_\star$ of the $\Delta$-regular ${\sf H}$ we further note that 
\beq\label{eq:v-e-d-ineq}
\Delta v_{{\sf H}_\star} > 2 e({\sf H}_\star). 
\eeq
Therefore, using \eqref{eq:copy-H-star-bd}-\eqref{eq:frac-ind-bd}, for $p \ll n^{-1/\Delta}$, we find that
\begin{multline*}
N(H_\star,\Graph) \cdot n^{-v_{{\sf H}_\star}} p^{-e({\sf H}_\star)} \le (2 \bar C)^{v_{\sf H}} \cdot (\log (1/p))^{v_{\sf H}} \cdot (n^2 p^\Delta)^{v_{{\sf H}_\star} - e({\sf H}_\star)/\Delta} \cdot n^{-v_{{\sf H}_\star}} p^{-e({\sf H}_\star)}\\
= (2 \bar C)^{v_{\sf H}} \cdot (\log (1/p))^{v_{\sf H}} \cdot (n p^\Delta)^{v_{{\sf H}_\star} - 2 e({\sf H}_\star)/\Delta}  \le (2 \bar C)^{v_{\sf H}} \cdot (\log (1/p))^{v_{\sf H}} \cdot (n p^\Delta)^{1/\Delta},
\end{multline*}
where in the last step we have used the fact $\Delta v_{{\sf H}_\star}$ and $2 e({\sf H}_\star)$ both being integer the inequality \eqref{eq:v-e-d-ineq} implies that $v_{{\sf H}_\star} - 2e({\sf H}_\star)/\Delta \ge 1/\Delta$. Thus continuing from above, as $np ^\Delta \ll (\log n)^{-\Delta v_{\sf H}}$, we derive that 
 \[
\gT:= \sum_{ \emptyset \ne {\sf H}_\star \subsetneq {\sf H}} N({\sf H}_\star, \Graph) \cdot n^{v_{\sf H}- v_{{\sf H}_\star}} p^{e({\sf H}) - {e}({{\sf H}_\star})} =o(n^{v_{\sf H}} p^{e({\sf H})}).
 \]
 As $\E[N({\sf H}, \G(n,p)] = n^{v_{\sf H}} p^{e({\sf H})}(1+o(1))$ it is now immediate from Definition \ref{dfn:seed-0} and \eqref{eq:cond-exp-ubd} that for any pre-seed graph $\Graph$ one must have that
 \[
 N({\sf H}, \Graph) \ge N({\sf H}, \Graph) \cdot (1- p^{e({\sf H})}) \ge \E_\Graph[N({\sf H}, \G(n,p)] - \E[N({\sf H}, \G(n,p)] - \gT \ge \delta(1-2\vep) n^{v_{\sf H}} p^{e({\sf H})}, 
 \]
 for all large $n$. Thus $\Graph$ is indeed a seed graph as well. The proof of the lemma is now complete.
 \end{proof}

Equipped with Lemma \ref{l:seed-markov} the proof of upper bound of the $\log$ probability of ${\sf UT}({\sf H}, \delta)$  now splits into two parts.


%

\begin{prop}\label{prop:strong-core}
 Fix $\delta >0$ and a $\Delta$-regular connected graph ${\sf H}$ with $\Delta \ge 2$. Then, for any {$p \in (0,1)$ satisfying $1 \ll np^{\Delta/2} \le n^{1/2}$}, and $\vep >0$ sufficiently small, we have that
\[
\liminf_{n \to \infty} -\f{ \log \P(\exists \Graph \subset \G(n,p): \Graph \text{ is a strong-core graph})}{n^2 p^\Delta \log (1/p)} \ge \f12 \delta^{2/v_{\sf H}} (1-\gf_{\sf H}(\vep)),
\]
for some nonnegative function $\gf_{\sf H}(\cdot)$ such that $\lim_{\vep \downarrow 0} \gf_{\sf H}(\vep) =0$.
\end{prop}

\begin{prop}
\label{p:upper-notstrong}
Let $\delta, \Delta$, $\vep$, ${\sf H}$, and $\gf_{\sf H}(\cdot)$ be as in Proposition \ref{prop:strong-core}. Let $p \in (0,1)$ be such that for $ (\log n)^{1/(v_{\sf H}-2)} \ll np^{\Delta/2} \le n^{1/2}$. Then 
\[
\liminf_{n\to \infty} -\dfrac{\log \P( \exists \Graph \subset \G(n,p): \Graph \text{ is a core graph with } e(\Graph) \ge \bar C_\star n^2 p^\Delta )}{n^2p^\Delta\log (1/p)} \geq \frac{1}{2}\delta^{2/v_{\sf H}}(1-\gf_{\sf H}(\vep)).
\]
\end{prop}


Let us now complete the proof of Theorem \ref{thm:cycle-main} using Propositions \ref{prop:strong-core} and \ref{p:upper-notstrong}.

\begin{proof}[Proof of Theorem \ref{thm:cycle-main} using Propositions \ref{prop:strong-core}-\ref{p:upper-notstrong}]
Recalling Proposition \ref{p:lower} we note that it only remains to prove the lower bound in \eqref{eq:cycle-ldp}. For brevity let us also write
\[
\mathscr{C}_1:= \left\{\exists \Graph \subset \G(n,p): \Graph \text{ is a core graph with } e(\Graph) \le \bar C_\star n^2 p^\Delta\right\}
\]
and 
\[
\mathscr{C}_2:=\left\{\exists \Graph \subset \G(n,p): \Graph \text{ is a core graph with } e(\Graph) \ge \bar C_\star n^2 p^\Delta \right\}.
\]
Using Lemma \ref{l:seed-markov} we have that 
\begin{align}
\limsup_{n \to \infty} \dfrac{\log \P\left({\sf UT}({\sf H}, \delta)\right)}{n^2 p^\Delta \log (1/p)}& =\limsup_{n \to \infty} \f{\log \P(\exists \Graph \subset \G(n,p): \Graph \text{ is a core graph})}{n^2 p^\Delta \log(1/p)}\notag\\
& \le \max\left\{ \limsup_{n \to \infty} \dfrac{\log \P(\mathscr{C}_1)}{n^2p^\Delta\log (1/p)},  \limsup_{n \to \infty} \dfrac{\log \P(\mathscr{C}_2)}{n^2p^\Delta\log (1/p)}\right\}. \label{eq:UT-split}
\end{align}
We claim that any core graph $\Graph$ with $e(\Graph) \le  \bar C_\star n^2 p^\Delta$ contains a strong-core subgraph $\Graph' \subset \Graph$. To see this for any such core graph $\Graph$ we iteratively remove edges from $E(\Graph)$ that participate in less than ${(\delta \vep/\bar C_\star)} \cdot (np^{\Delta/2})^{v_{\sf H}-2}$ copies of ${\sf H}$. This ensures that subgraph $\Graph'$ obtained at the end of this peeling procedure have the desired lower bound (SC3) of Definition \ref{dfn:strong-core}. The upper bound on $e(\Graph')$ is automatic and (SC1) follows from triangle inequality. This proves that $\Graph'$ is a strong-core graph establishing the claim. Hence
\[
\P(\mathscr{C}_1) \le \P(\exists \Graph \subset \G(n,p): \Graph \text{ is a strong-core graph}).
\]
Therefore, continuing from \eqref{eq:UT-split}, and using Propositions \ref{prop:strong-core}-\ref{p:upper-notstrong}, we deduce that
\begin{align*}
\limsup_{n \to \infty} \dfrac{\log \P\left({\sf UT}({\sf H}, \delta)\right)}{n^2 p^\Delta \log (1/p)}&   \le -\f12 \delta^{2/v_{\sf H}}(1-\gf_{\sf H}(\vep)),
\end{align*}
for any sufficiently small $\vep >0$. As $\lim_{\vep \downarrow 0} \gf_{\sf H}(\vep)=0$, sending $\vep$ to zero the proof is completed.
\end{proof}

The rest of this paper will be devoted to proving Propositions \ref{prop:strong-core} and \ref{p:upper-notstrong}. The proof of Proposition \ref{prop:strong-core} is deferred to Section \ref{sec:strong-core}. To prove Proposition \ref{p:upper-notstrong} we treat two regimes $n^2p^\Delta \ge (\log n)^{2v_{\sf H}}$ and $n^2p^\Delta \le (\log n)^{2v_{\sf H}}$ separately. First let us consider the easier case $n^2p^\Delta \ge (\log n)^{2v_{\sf H}}$.





\subsection{Proposition \ref{p:upper-notstrong} in large $p$ regime}\label{sec:notstronglarge-p}


As already outlined in Section \ref{sec:outline} the key here is to show that the subgraph of a core graph $\Graph$ induced  by the edges that are adjacent to vertices of low degree is a bipartite graph. To make this idea precise let us consider the following set of low degree vertices
\beq\label{eq:cW}
\cW:= \cW(\Graph):=\{ v \in V(\Graph): \deg_\Graph(v) \le D \},
\eeq
where 
{\beq\label{eq:D}
D:=D(\vep):= \lceil 16 \Delta/\vep\rceil.
\eeq}
Further let $\Graph_\cW \subset \Graph$ to be the subgraph induced by edges adjacent to vertices in $\cW$. For $v \in V(\Graph)$ we write $\deg_\Graph(v)$ to denote the degree of vertex $v$ in graph $\Graph$. 
Finally, for ${\bm e} \geq {\bm e}_\star:= \bar C_\star n^2p^{\Delta}$ we let 
\[
\sA_{{\bm e}}:=\left\{\exists \Graph \subset \G(n,p): \Graph \text{ is a core graph}, \, \Graph_\cW \text{ is bipartite}, \text{ and } e(\Graph) = {\bm e} \right\}.
\]
Equipped with the above set of notation let us state the lemma that yields the entropic stability of the set of all core graphs $\Graph$ with a large number of edges for which $\Graph_\cW$ is bipartite.  

\begin{lem}\label{lem:core-large-edge}
Fix ${\sf H}$ a $\Delta$-regular graph and $\delta >0$. If $n^2p^\Delta \ge (\log n)^{2v_{\sf H}}$, then for any $\vep \in (0,\f{1}{8})$,
\[
\limsup_{n \to \infty} -\f{\log \P(\cup_{{\bm e}\ge {\bm e}_\star}\sA_{{\bm e}})}{n^2 p^\Delta \log(1/p)} \ge  {\f{1}{16}} \bar C_\star.
\]
\end{lem}

Lemma \ref{lem:core-large-edge} follows from an easy combinatorial argument bounding the number of potential graphs participating in the event $\cup_{{\bm e}\ge {\bm e}_\star}\sA_{{\bm e}}$. Its proof is postponed to Section \ref{s:bipartite}. To complete the proof of  Proposition \ref{p:upper-notstrong} we need the following lower bound on the product of the degrees of the end points of edges in  core graphs which will show that for such graphs $\Graph_\cW$ is indeed bipartite.

\begin{lem}\label{lem:bad-graph-bd}
Fix ${\sf H}$ a $\Delta$-regular graph. Let $\Graph$ be a core graph. If $n^2p^\Delta \gg (\log n)^{v_{\sf H}}$ then for every edge $e = (u,v) \in E(\Graph)$
\beq\label{eq:prod-lbd}
\deg_\Graph(u) \cdot \deg_\Graph(v) \ge \f{\wt c_0(\vep) \cdot e(\Graph)}{(\log n)^{v_{\sf H}}},
\eeq
for some constant $\wt c_0(\vep) >0$. 
\end{lem}
Bounds same as above have been derived in \cite{hms} (see Claim 7.5 there). We include a short outline of the proof of Lemma \ref{lem:bad-graph-bd} in Appendix \ref{app:app1} for the reader's convenience. The lower bound on $p$ in Lemma \ref{lem:bad-graph-bd} is added because otherwise the lower bound \eqref{eq:prod-lbd} would become useless for the graphs for which we will apply this result.

\begin{proof}[Proof of Proposition \ref{p:upper-notstrong} for $n^2p^\Delta \ge (\log n)^{2v_{\sf H}}$ (assuming Lemma \ref{lem:core-large-edge})]
Let $w \in \cW \subset V(\Graph)$ and $(w, w') \in E(\Graph)$, for some $w' \in V(\Graph)$, where $\Graph$ is a core graph. As $n^2p^\Delta \ge (\log n)^{2v_{\sf H}}$, Lemma \ref{lem:bad-graph-bd} implies that 
\beq\label{eq:deg-lbd}
\deg_\Graph(w') \ge \f1D \cdot \deg_\Graph(w) \cdot \deg_\Graph(w') \ge \f1D \cdot \f{\wt c_0(\vep) \cdot e(\Graph) }{(\log n)^{v_{\sf H}}} \ge \f{1}{2D} (\delta(1-3\vep))^{\frac{2}{v_{\sf H}}} \cdot \f{\wt c_0(\vep) \cdot n^2 p^{\Delta} }{(\log n)^{v_{\sf H}}} \ge 2D,
\eeq
for all large $n$, where the penultimate step follows from the fact that $\Graph$ being a core graph must possess at least $\delta(1-3\vep) n^{v_{\sf H}} p^{e({\sf H})}$ copies of ${\sf H}$ (or equivalently $\delta(1-3\vep) (n^2 p^\Delta)^{v_{\sf H}/2}$ copies of ${\sf H}$) and hence the lower bound on $e(\Graph)$ follows from that fact that $N({\sf H}, \Graph) \le (2e(\Graph))^{v_{\sf H}/2}$ (use \eqref{eq:copy-H-star-bd} and \cite[Lemma 5.3]{hms} to deduce the fact that the fractional independence number of the regular graph ${\sf H}$ is equal to  $v_{\sf H}/2$). 

The lower bound \eqref{eq:deg-lbd} in particular implies that the subgraph $\Graph_\cW$ is bipartite. 
Therefore
\[
\P( \exists \Graph \subset \G(n,p): \Graph \text{ is a core graph with } e(\Graph) \ge \bar C_\star n^2 p^{\Delta} ) \le \P(\cup_{{\bm e}\ge {\bm e}_\star}\sA_{{\bm e}}).
\]
The proof is now completed upon using Lemma \ref{lem:core-large-edge} and setting $\bar{C}_{\star}\geq 32\delta^{2/v_{\sf H}}$.
\end{proof}


\subsection{Proposition \ref{p:upper-notstrong} in small $p$ regime}
We begin by introducing the following set of notation. First we split the set of all core graphs with at least $\bar C_\star n^2 p^{\Delta}$ edges into two subsets: 
\[
{\rm Core}_1:= \left\{ \exists \Graph\subset \G(n,p): \Graph \text{ is a core graph with } \bar{C}_{\star} n^2 p^{\Delta} \leq e(\Graph) \leq (\log \log n)^{-2}\bar{C}n^2p^{\Delta} \log (1/p)\right\}
\]
and
\begin{multline*}
{\rm Core}_2:= \big\{ \exists \Graph\subset \G(n,p): \Graph \text{ is a core graph with }\\
 (\log \log n)^{-2}\bar{C}n^2p^{\Delta} \log (1/p) \leq e(\Graph) \leq \bar{C}n^2p^{\Delta} \log (1/p)\big\}.
\end{multline*}
As discussed in Section \ref{sec:outline} obtaining the desired probability bound on ${\rm Core}_2$ requires a chaining type argument. Whereas, to find a bound on the probability ${\rm Core}_1$ we employ a different method. 
So, first let us proceed to show that the set of core graphs participating in the event ${\rm Core}_1$ are entropically stable. This demands a further subdivision of ${\rm Core}_1$. We let 
{\begin{multline*}
{\rm Core}_{1,1}:= \big\{ \exists \Graph\subset \G(n,p): \Graph \text{ is a core graph with } \\
\bar C_\star n^2 p^{\Delta} \le {\bm e}_{1,2}(\Graph)+ {\bm e}_{2,2}(\Graph) \le e(\Graph) \le (\log \log n)^{-2}\bar{C}n^2p^{\Delta} \log (1/p) \big\} ,
\end{multline*}}
where
\beq\label{eq:e-121}
{\bm e}_{1,2}(\Graph):= |E_{1,2}(\Graph)|, \qquad {\bm e}_{2,2}(\Graph):= |E_{2,2}(\Graph)|,
\eeq
\beq\label{eq:e-122}
E_{1,2}(\Graph):=\{e=(u_1,u_2) \in E(\Graph): \text{ one of } u_1 \text{ and } u_2 \text{ is in } \cW(\Graph)\},
\eeq
\beq\label{eq:e-123}
E_{2,2}(\Graph):=\{e=(u_1,u_2) \in E(\Graph): u_1, u_2 \notin \cW(\Graph)\},
\eeq
and we recall the definition of $\cW(\Graph)$ from \eqref{eq:cW}. 
Setting
\beq\label{eq:E11}
E_{1,1}:= E_{1,1}(\Graph):=\{e =(u_1, u_2) \in E(\Graph): u_1, u_2 \in \cW(\Graph)\},
\eeq
we define $\cN_{1,1}({\sf H}, \Graph)$ be the number of labelled copies of ${\sf H}$ in $\Graph$ that use at least one edge from $E_{1,1}(\Graph)$. In words $\cN_{1,1}({\sf H}, \Graph)$ is the number of copies of ${\sf H}$ in $\Graph$ that uses at least an edge with both end points of low degree. Now define 
{\begin{multline*}
{\rm Core}_{1,2}:= \bigg\{ \exists \Graph\subset \G(n,p): \Graph \text{ is a core graph with } \cN_{1,1}({\sf H}, \Graph) \ge \vep \delta n^{v_{\sf H}} p^{e({\sf H})} \\
\text{ and } \bar{C}_{\star} n^2 p^{\Delta} \leq e(\Graph) \leq (\log \log n)^{-2}\bar{C}n^2p^{\Delta} \log (1/p) \bigg\}.
\end{multline*}
}
The next two results yield upper bounds on the probabilities of ${\rm Core}_{1,1}$ and ${\rm Core}_{1,2}$. For a later use during the chaining procedure we will in fact bound probabilities of events that are somewhat larger than ${\rm Core}_{1,1}$ and ${\rm Core}_{1,2}$. Let us define these events. We set $d_{\min}(\Graph):= \min_{v \in V(\Graph)} \deg_\Graph(v)$. Then denote
\begin{multline}\label{eq:wtcore11}
\wt{\rm Core}_{1,1}:= \big\{\exists \Graph \subset \G(n,p): d_{\min}(\Graph) \ge \Delta \text{ and } \\
\bar C_\star n^2 p^{\Delta} \le {\bm e}_{1,2}(\Graph)+ {\bm e}_{2,2}(\Graph) \le e(\Graph) \le \bar C n^2 p^{\Delta} \log(1/p)/(\log \log n)^2\big\}
\end{multline}
and
\begin{multline}\label{eq:wtcore12}
\wt{\rm Core}_{1,2}:= \big\{\exists \Graph \subset \G(n,p): d_{\min}(\Graph)\ge \Delta, \, \cN_{1,1}({\sf H} ,\Graph) \ge \vep \delta n^{v_{\sf H}} p^{e({\sf H})}, \text{ and } \\
   e(\Graph) \le \bar C n^2 p^{\Delta} \log(1/p)/(\log \log n)^2\big\}.
\end{multline}
Note the differences between $\wt{\rm Core}_{1,i}$ and ${\rm Core}_{1,i}$, for $i=1,2$. The former does not require the graph $\Graph$ to be a core graph. It requires a mild condition that the minimum degree should be at least $\Delta$. This mild requirement will suffice to obtain the desired probability bounds. We now state the results. 

\begin{lem}
\label{l:core11}
Fix $\delta, \vep > 0$, and a $\Delta$-regular connected graph ${\sf H}$ with $\Delta \ge 2$. Let $p \in (0,1)$ be such that $1 \ll np^{\Delta/2} \le (\log n)^{v_{\sf H}}$. Then 
\[
\limsup_{n\to \infty} \dfrac{\log \P(\wt{\mathrm{Core}_{1,1}})}{n^2p^{\Delta}\log (1/p)} \leq -{16}\delta^{2/v_{\sf H}}.
\]
\end{lem}

\begin{lem}
\label{l:core12}
Let $\delta, \vep$, and ${\sf H}$ be as in Lemma \ref{l:core11}. If $p \in (0,1)$ such that $ np^{\Delta/2} \gg (\log n)^{1/(v_{\sf H}-2)}$ then we have 
\[
\limsup_{n\to \infty} \dfrac{\log \P(\wt{\mathrm{Core}_{1,2}})}{n^2p^{\Delta}\log (1/p)}\, {=-\infty}.
\]
\end{lem}

The proofs of these two results being combinatorial in nature are moved to Sections \ref{sec:combin-graph-bd} and \ref{sec:largecN-11}, respectively. Combining Lemmas \ref{l:core11} and \ref{l:core12} we now have the following upper bound on the probability of ${\rm Core}_1$ establishing its entropic stability.

\begin{prop}
\label{p:core--1}
Fix $\delta, \vep >0$, and a $\Delta$-regular connected graph ${\sf H}$ with $\Delta \ge 2$. For $\vep$ sufficiently small and $(\log n)^{1/(v_{\sf H}-2)} \ll np^{\Delta/2} \leq (\log n)^{v_{\sf H}}$ we have 
\beq\label{eq:pcore--1}
\limsup_{n\to \infty} \dfrac{\log \P(\mathrm{Core}_1)}{n^2p^{\Delta}\log (1/p)} \leq -\frac{1}{2}\delta^{2/{v_{\sf H}}} (1-\gf_{\sf H}(\vep)),
\eeq
where $\gf_{\sf H}(\cdot)$ is as in Proposition \ref{prop:strong-core}. 
\end{prop}

\begin{proof}[Proof (assuming Lemma \ref{l:core11} and \ref{l:core12})]
We begin by claiming that
\beq\label{eq:core-1-minus-1112}
{\rm Core}_1 \setminus ({\rm Core}_{1,1} \cup {\rm Core}_{1,2}) \subset \left\{\exists \Graph \subset \G(n,p): \Graph \text{ is a strong-core graph}\right\}.
\eeq
To see this we consider a core graph $\Graph'$ for which
\beq\label{eq:e-n11}
\bar e(\Graph'):= e_{1,2}(\Graph') + e_{2,2}(\Graph') \le \bar C_\star n^2 p^{\Delta} \qquad \text{ and } \qquad \cN_{1,1}({\sf H}, {\Graph'}) \le \vep \delta n^{v_{\sf H}} p^{e({\sf H})}.
\eeq
Let $\Graph_0$ be the {$\Delta$-core} of the subgraph $\Graph_0'$ obtained from $\Graph'$ by removing all edges in $E_{1,1}(\Graph')$. Since $\Graph_0$ is a {$\Delta$-core} of $\Graph_0'$ it is straightforward to note that 
\[
N({\sf H}, \Graph_0) = N({\sf H}, \Graph_0') \ge N({\sf H}, \Graph') - \delta \vep n^{v_{\sf H}} p^{e({\sf H})} \ge \delta(1-4\vep) n^{v_{\sf H}} p^{e({\sf H})},
\]
where the penultimate step is due to \eqref{eq:e-n11} and the last step is due to the fact that $\Graph'$ is a {core graph}. Since $e(\Graph_0) \le \bar e(\Graph') \le \bar C_\star n^2 p^{\Delta}$ we now run a peeling procedure as in the proof of Lemma \ref{l:seed-markov} to extract a further subgraph $\Graph$ which is a strong-core graph. As
\[
{\rm Core}_1 \setminus ({\rm Core}_{1,1} \cup {\rm Core}_{1,2}) \subset \left\{\exists \Graph \subset \G(n,p): \Graph \text{ is a core graph and obeys \eqref{eq:e-n11}}\right\},
\] 
we deduce \eqref{eq:core-1-minus-1112}. Therefore, upon noting that 
\[
{\rm Core}_{1,1} \subset \wt{\rm Core}_{1,1} \qquad \text{ and } \qquad {\rm Core}_{1,2} \subset \wt{\rm Core}_{1,2},
\]
using Lemmas \ref{l:core11} and \ref{l:core12}, \eqref{eq:core-1-minus-1112}, and Proposition \ref{prop:strong-core} we now arrive at \eqref{eq:pcore--1}. This completes the proof. 
\end{proof}

In the next section we derive the entropic stability of core graphs with at least $ \bar C n^2 p^{\Delta} \log (1/p) \cdot (\log \log n)^{-2}$ edges, i.e.~we find an appropriate bound on the probability of ${\rm Core}_2$.

\subsection{Core graphs with larger number of edges}
Finding a suitable bound on the probability of ${\rm Core}_2$ requires a chaining type argument. To run the chaining procedure effectively we need a few more notation. Recall that we need to consider core graphs $\Graph$ for which
\[
(\log \log n)^{-2}\bar{C}n^2p^{\Delta} \log (1/p) \leq e(\Graph) \leq \bar{C}n^2p^{\Delta} \log (1/p).
\] 
We divide this range into dyadic discrete intervals. Set $L_{n}:=\lfloor 2\log _2 (\log \log n) \rfloor +1 $. For $j=1,2, \ldots, {L_{n}}$, define
\[
\cJ_j:= \left\{\Graph \subset K_n: e(\Graph) \in \left( 2^{-j} \bar{C}n^{2}p^{\Delta}\log (1/p), 2^{-(j-1)} \bar{C}n^{2}p^{\Delta}\log (1/p)\right]\right\},
\]
and
\[
{\cJ_{L_n+1}}:= \left\{\Graph \subset K_n: e(\Graph) \leq  2^{-L_n} \bar{C}n^{2}p^{\Delta}\log (1/p)\right\}.
\]
It will be clear from below that during the chaining argument we may end up with graphs with no a priori lower bound on its edges. Therefore, in $\cJ_{{L_n+1}}$ we do not impose any lower bound on the number of edges. The partition $\{\cJ_j\}_{j=1}^{{L_n+1}}$ naturally yields a partition of ${\rm Core}_2$. For $j \in \llbracket {L_n+1} \rrbracket$ we define
 \[
 \mathrm{Core}_{2,j}:= \left\{\exists \Graph \subset \G(n,p): \Graph \text{ is a core graph and } \Graph \in \cJ_j\right\}. 
 \]
 Let us also define the following sequence of events: for $j \in \llbracket {L_n+1} \rrbracket$ we let 
  \[
\wt{ \mathrm{Core}}_{2,j}:= \left\{\exists \Graph \subset \G(n,p): N({\sf H}, \Graph) \ge (1 - 3 \vep - \gs_j \vep)\delta n^{v_{\sf H}} p^{e({\sf H})} \text{ and } \Graph \in \cJ_j\right\},
\]
where for brevity we write $\gs_j:= \sum_{i=1}^{j-1} 2^{-i}$. The difference in ${\rm Core}_{2,j}$ and $\wt{\rm Core}_{2,j}$ lies in the fact that the former event requires $N({\sf H}, \Graph) \ge (1-3\vep) \delta n^{v_{\sf H}} p^{e({\sf H})}$, whereas the latter requires a slightly weaker lower bound on $N({\sf H}, \Graph)$. Furthermore the latter one does not need to obey (C3) of Definition \ref{dfn:core-graph}. Therefore
\beq\label{eq:CwtC}
{\rm Core}_{2,j} \subset \wt{\rm Core}_{2,j} \qquad \text{ for } j \in \llbracket {L_n+1} \rrbracket.
\eeq
The rationale behind defining the events $\{\wt {\rm Core}_{2,j}\}$ is as follows: During the chaining argument we need to iteratively run the peeling procedure similar to the one in the proof Lemma \ref{l:seed-markov}. This results in losing a small fraction of the number of copies of ${\sf H}$ in the graph with which we start the next step in the chaining argument. Hence one not only requires to bound the probabilities of $\{{\rm Core}_{2,j}\}$ but also those of $\{\wt{\rm Core}_{2,j}\}$. 

Bounding probabilities of $\wt {\rm Core}_{2,j}$ requires a further subdivision of it. We define
\begin{multline}
{\rm Core}_{2,j,\alpha}:= \Big\{\exists \Graph \subset \G(n,p):  N({\sf H}, \Graph) \ge (1 - 3 \vep - \gs_j \vep)\delta n^{v_{\sf H}} p^{e({\sf H})}, \\
\cN_{1,1}({\sf H}, \Graph) \ge 2^{-j} \vep \delta n^{v_{\sf H}} p^{e({\sf H})}, \text{ and } \Graph \in \cJ_j \Big\}. \label{eq:core2ja}
\end{multline}
Thus ${\rm Core}_{2,j,\a}$ can be considered to be the subset of {$\wt{\rm Core}_{2,j}$} for which $\cN_{1,1}({\sf H}, \Graph)$ is large. Notice that this threshold for $\cN_{1,1}({\sf H}, \Graph)$ to be considered to be large depends on $j$. This will be crucial for our proof.  

Next for a graph $\Graph$ we denote $\varpi_\Delta(\Graph)$ to be the {$\Delta$-core} of the subgraph of $\Graph$ obtained by removing the edges in $E_{1,1}(\Graph)$ (recall \eqref{eq:E11}). We then denote 
{\corAB{\beq
{\rm Core}_{2,j, \b}:= \left\{\exists \Graph \subset \G(n,p): \varpi_\Delta(\Graph), \Graph \in \cJ_j, \text{ and } N({\sf H}, \Graph) \ge (1 - 3 \vep - \gs_j \vep)\delta n^{v_{\sf H}} p^{e({\sf H})}\right\}. \label{eq:core2jb}
\eeq}}
 The following two lemmas yield bound on the probabilities of ${\rm Core}_{2,j,\a}$ and ${\rm Core}_{2,j,\b}$, respectively. 
 \begin{lem}
\label{l:core2ja}
Fix $\delta, \vep >0$, and a $\Delta$-regular ${\sf H}$ connected graph with $\Delta \ge 2$. If $p\in (0,1)$ is such that $(\log n)^{1/(v_{\sf H}-2)} \ll np^{\Delta/2} \le (\log n)^{v_{\sf H}}$ then we have
\[
{\limsup_{n \to \infty} \f{\max_{j \in \llbracket L_n \rrbracket} \left\{\log \P(\mathrm{Core}_{2,j,\a})\right\}}{n^2 p^{\Delta}\log (1/p)} = -\infty.}
\]
\end{lem}


\begin{lem}
\label{l:core2jb}
Let $\delta, \vep$, and ${\sf H}$ be as in Lemma \ref{l:core2ja}. If $1 \ll np^{\Delta/2} \le (\log n)^{v_{\sf H}}$  then we have
\[
{\limsup_{n \to \infty} \f{\max_{j \in \llbracket L_n \rrbracket} \left\{\log \P(\mathrm{Core}_{2,j,\b})\right\}}{n^2 p^{\Delta}\log (1/p)} =-\infty.}
\]
\end{lem}

Proofs of Lemmas \ref{l:core2ja} and \ref{l:core2jb} are deferred to Sections \ref{sec:largecN-11} and \ref{sec:combin-graph-bd}, respectively. Building on Lemmas \ref{l:core2ja}-\ref{l:core2jb} we now have the result yielding a desired bound on the probability of ${\rm Core}_2$. 
 

%




\begin{prop}
\label{p:core2}
Fix $\delta, \vep >0$, and a $\Delta$-regular ${\sf H}$ connected graph with $\Delta \ge 2$. For $\vep$ sufficiently small and $p \in (0,1)$ such that $(\log n)^{1/(v_{\sf H}-2)} \ll np^{\Delta/2} \le (\log n)^{v_{\sf H}}$ we have 
\beq\label{eq:prob-core20}
\limsup_{n\to \infty} \dfrac{\log \P(\mathrm{Core}_2)}{n^2p^{\Delta}\log (1/p)} \leq -\frac{1}{2}\delta^{2/{v_{\sf H}}}(1-\gf_{\sf H}(\vep)),
\eeq
where $\gf_{\sf H}(\cdot)$ is as in Proposition \ref{prop:strong-core}. 
\end{prop}

\begin{proof}
We claim that for any $j \in \llbracket L_n \rrbracket$
\beq\label{eq:chaining-1} 
\P(\mathrm{Core}_{2,j}) \le \P(\wt{ \mathrm{Core}}_{2,j})\leq {2\exp\left(-\delta^{2/v_{\sf H}}n^2p^\Delta\log(1/p)\right))}+ \sum_{j'=j+1}^{{L_n+1}} \P(\wt{ \mathrm{Core}}_{2,j'}).
\eeq
The first inequality is immediate from \eqref{eq:CwtC}. To see the second inequality, we observe that if $\Graph \in \cJ_j$ is graph such that 
\beq\label{eq:NN11}
N({\sf H}, \Graph) \ge (1-3\vep -\gs_j\vep)\delta n^{v_{\sf H}} p^{e({\sf H})}, \quad \cN_{1,1}({\sf H}, \Graph) \le 2^{-j} \vep \delta n^{v_{\sf H}} p^{e({\sf H})}, \quad \text{ and } \quad \varpi_\Delta(\Graph) \notin \cJ_j,
\eeq
then, as $\varpi_\Delta(\Graph) \subset \Graph$, we have that $\varpi_\Delta(\Graph) \in \cup_{j'=j+1}^{{L_n+1}} \cJ_j$. Furthermore, denoting $\Graph'$ to be the subgraph of $\Graph$ obtained by removing edges in $E_{1,1}(\Graph)$, as ${\sf H}$ is a $\Delta$-regular graph, we note that
\[
N({\sf H}, \varpi_\Delta(\Graph)) = N({\sf H}, \Graph') \ge N({\sf H}, \Graph) - \cN_{1,1}({\sf H}, \Graph) \ge (1 - 3\vep - \gs_{j'}\vep)\delta n^{v_{\sf H}} p^{e({\sf H})}
\]
for any $j'=j+1, \ldots, {L_n+1}$, where the final inequality is a consequence of \eqref{eq:NN11}. Thus the last two observations together imply that $\varpi_\Delta(\Graph) \in \cup_{j'=j+1}^{{L_n+1}} \wt{\rm Core}_{2,j'}$ for any graph $\Graph \in \cJ_j$ satisfying \eqref{eq:NN11}. Thus we have shown that 
\[
\wt{\rm Core}_{2,j} \setminus ({\rm Core}_{2,j,\a} \cup {\rm Core}_{2,j,\b}) \subset \cup_{j'=j+1}^{{L_n+1}} \wt{\rm Core}_{2,j'}.
\]
Now Lemmas \ref{l:core2ja} and \ref{l:core2jb} together with the union bound yield the second inequality of \eqref{eq:chaining-1}. 

To complete the proof of the proposition we use \eqref{eq:chaining-1} with $j=1$ to derive that 
\begin{equation*}
\P({\rm Core}_2) \le \sum_{j=1}^{{L_n+1}} \P(\mathrm{Core}_{2,j}) \leq {2 \exp\left(-\delta^{2/{v_{\sf H}}} n^2 p^{\Delta} \log(1/p)\right)} + 2 \sum_{j=2}^{{L_n+1}} \P(\wt{\rm Core}_{2,j'}).
\end{equation*}
Repeating the same procedure as above iteratively with $j=2,3,\ldots, {L_n}$, we arrive at the bound
\beq\label{eq:prob-core21}
\P({\rm Core}_2) \le {2^{L_n+{2}} \left[\exp\left(-\delta^{2/{v_{\sf H}}} n^2 p^\Delta \log (1/p)\right) + \P(\wt{\rm Core}_{2,{L_n+1}})\right]}.
\eeq
So it now remains to evaluate the probability of $\wt{\rm Core}_{2, {L_n+1}}$. To evaluate the same we recall that the number copies of ${\sf H}$ in any graph $\Graph$ and its {$\Delta$-core} are the same. Therefore recalling the definitions of $\wt{\rm Core}_{1,i}$, $i=1,2$, from \eqref{eq:wtcore11}-\eqref{eq:wtcore12}, we derive that
\begin{multline*}
 \wt{\rm Core}_{2, {L_n+1}} \setminus (\wt{\rm Core}_{1,1} \cup \wt{\rm Core}_{1,2})\\
 \subset   \Big\{\exists \Graph \subset \G(n,p): N({\sf H}, \Graph) \ge (1 - 3 \vep - \gs_{L_n}\vep)\delta n^{v_{\sf H}} p^{e({\sf H})}, \, \cN_{1,1}({\sf H}, \Graph) \le \vep \delta n^{v_{\sf H}} p^{e({\sf H})}, \\
 {d_{\min}(\Graph) \ge \Delta}, \text{ and } e_{1,2}(\Graph)+ e_{2,2}(\Graph) \le \bar C_\star n^2 p^{\Delta}\Big\}\\
 \subset \left\{\exists \Graph \subset \G(n,p): \Graph \text{ is a strong-core graph}\right\},
\end{multline*}
where the last step follows by proceeding similarly as in the proof of \eqref{eq:core-1-minus-1112}. We omit the details. 

Hence, applying Lemmas \ref{l:core11}-\ref{l:core12}, and Proposition \ref{prop:strong-core} we derive that
\beq\label{eq:prob-core22}
\limsup_{n \to \infty} \f{\log \P(\wt{\rm Core}_{2,{L_n+1}})}{n^2 p^{\Delta} \log (1/p)} \le -\f12 \delta^{2/{v_{\sf H}}} (1-\gf_{\sf H}(\vep)).
\eeq
As $L_n = O(\log \log \log n) \ll (\log n)^{1/(v_{\sf H}-2)} \ll np^{\Delta/2}$ combining \eqref{eq:prob-core21}-\eqref{eq:prob-core22} we obtain \eqref{eq:prob-core20}.  
This completes the proof of the proposition.
\end{proof}


\section{Strong-core graphs are entropically stable}\label{sec:strong-core}
In this section we prove Proposition \ref{prop:strong-core}. That is, we show that the set of all strong-core graphs is entropically stable. Recall the definition of the strong-core graphs from Definition \ref{dfn:strong-core}. As outlined in Section \ref{sec:outline} the proof relies on identifying a bipartite subgraph $\Graph_\gb$ of a strong-core graph $\Graph$ so that almost all copies of ${\sf H}$ in $\Graph$ is either contained in $\Graph_\gb$ or $\Graph\setminus \Graph_\gb$. We remind the reader the key to this is block path structure of Figure \ref{fig:path-str}. To establish that block path structure we will show that for any strong-core graph, except possibly a few ``bad'' edges, the product of the degrees of the two end points of any of its ``good'' edges satisfies a strong upper and lower bound. Below we provide a precise formulation of good and bad edges of a strong-core graph. 
\begin{dfn}\label{dfn:bad-graph}
Fix $C_0 < \infty$ and let $\Graph_{{\rm high}} \subset \Graph$ be the subgraph spanned by the edges $e =(u,v) \in E(\Graph)$ for which
\beq\label{eq:sc-ubd}
\deg_\Graph(u) \cdot \deg_\Graph(v) \ge C_0 n^2 p^\Delta.
\eeq
Set $\Graph_{{\rm low}}:= \Graph \setminus \Graph_{{\rm high}}$, i.e.~$\Graph_{{\rm low}}$ is spanned by the edges for which \eqref{eq:sc-ubd} does not hold. Furthermore, we write $\Graph_{{\rm bad}} \subset \Graph$ to denote the subgraph induced by the edges $e \in E(\Graph)$ for which every copy of ${\sf H}$ passing through it uses at least one edge belonging to $\Graph_{{\rm high}}$. 
\end{dfn}

In the following lemma we show that, if $C_0$ in \eqref{eq:sc-ubd} is chosen to be sufficiently large, then the number of edges in $\Graph_{{\rm bad}}$ is only a small desired fraction of that in $\Graph$, and moreover the number of labelled copies of ${\sf H}$ in $\Graph_{{\rm low}}$ is almost same as that in $\Graph$. Furthermore, we will establish a lower bound on the product of the degrees of any pair of adjacent vertices in $\Graph$. These facts together will imply that one can work with $\Graph_{{\rm low}}$ instead of $\Graph$ for which one has tight upper and lower bounds on the products of the degrees of the end points of any edge. 

\begin{lem}\label{lem:bad-edges-bounds}
Fix a $\Delta$-regular graph ${\sf H}$ with $\Delta \ge 2$. Let $\Graph$ be a strong-core graph and $p \in (0,1)$. Then, for any $\vep >0$, there exist $0 < c_0(\vep), C_0(\vep) < \infty$, such that the followings hold:
\begin{enumerate}
\item[(a)] For every edge $e =(u,v) \in E(\Graph)$
\[
\deg_\Graph(u) \cdot \deg_\Graph(v) \ge c_0(\vep) n^2 p^\Delta. 
\]
\item[(b)] Let $\Graph_{{\rm high}}:=\Graph_{{\rm high}}(\vep) \subset \Graph$ be the subgraph spanned by the edges $e \in E(\Graph)$ for which \eqref{eq:sc-ubd} holds with $C_0=C_0(\vep)$. Having defined $\Graph_{{\rm high}}$ we let $\Graph_{{\rm low}}:=\Graph_{{\rm low}}(\vep)$ and $\Graph_{{\rm bad}}:= \Graph_{{\rm bad}}(\vep)$ to be as in Definition \ref{dfn:bad-graph}. Then 
\[
N({\sf H}, \Graph_{{\rm low}}) \ge (1-\vep) N({\sf H}, \Graph),
\]
and
\[
e(\Graph_{{\rm high}}) \le e(\Graph_{{\rm bad}}) \le \vep e(\Graph). 
\] 
\end{enumerate}
\end{lem}

Note that lower bound in Lemma \ref{lem:bad-edges-bounds}(a) is similar to that in Lemma \ref{lem:bad-graph-bd} while the former is sharper. This is due to the stronger bounds on $e(\Graph)$ and on $N({\sf H}, \Graph, e)$ in a strong-core graph $\Graph$ when compared to a core graph. Bounds analogous to Lemma \ref{lem:bad-edges-bounds}(b) have also been derived in \cite{hms} for core graphs. Repeating a same line of argument and using the bounds (SC2)-(SC3) of Definition \ref{dfn:strong-core} one can deduce Lemma \ref{lem:bad-edges-bounds}. We include its proof in Appendix \ref{app:app1} for completeness. 

Equipped with Lemma \ref{lem:bad-edges-bounds} we now proceed to the proof of Proposition \ref{prop:strong-core}. 
\subsection{Proposition \ref{prop:strong-core} for non-bipartite graphs}\label{sec:strong-core-odd} When ${\sf H}$ is a non-bipartite graph using the fact that it must contain a cycle of odd length we provide an alternate and shorter proof than the one outlined above. 

The existence of an odd length cycle in ${\sf H}$ allows us to translate the tight upper and lower bounds on the product of the degree of any two adjacent vertices in $\Graph_{{\rm low}}$ to a lower bound on the degree of the non-isolated vertices of $\Graph\setminus \Graph_{\rm bad}.$\footnote{A similar argument has appeared in \cite[Claim 7.7]{hms} for core graphs when $np \gg (\log n)^{\Delta v_{\sf H}^2}$ and ${\sf H}$ is a $\Delta$-regular non-bipartite graph.} This together with the upper bound on $e(\Graph_{\rm bad})$ yields a bound on the number of strong-core graphs with a given number of edges, which in turn produces an effective bound on the probability of the existence of a strong-core graph with that many edges. Below we carry out the details.

\begin{proof}[Proof of Proposition \ref{prop:strong-core} for non-bipartite graphs]
We begin by claiming that for every vertex $v \in V(\Graph\setminus \Graph_{\rm bad})$ 
\beq\label{eq:deg-v-lbd}
\deg_\Graph(v) \ge c_\star(\vep) np^{\Delta/2},
\eeq
for some $c_\star(\vep)>0$. 

Fix $v \in V(\Graph\setminus \Graph_{\rm bad})$. Turning to prove the above claim we note that, by definition of $\Graph_{\rm bad}$, the vertex $v$ must participate in at least one labelled copy $\mathscr{H}$ of ${\sf H}$ that is contained in $\Graph_{\rm low}$. For ease of writing, let $\varphi$ be the labelling. That is, $\varphi: V({\sf H}) \mapsto V(\mathscr{H})$ is an embedding of ${\sf H}$ in $\mathscr{H}$. 
Since ${\sf H}$ is a non-bipartite graph it must contain a cycle $C_{2t+1}$, the cycle of length $(2t+1)$, for some $t\ge 1$ such that $2t+1 \le v_{\sf H}$. Let $\{v_1, v_2, \ldots, v_{2t+1}\}$ be the subset of the vertices $V(\mathscr{H})$ that gets mapped to the vertices of cycle $C_{2t+1}$ under the labelling $\varphi$. For ease of writing, without loss of generality, let us also assume that the vertices are indexed in a way so that $(v_i, v_{i+1}) \in E(\Graph_{\rm low})$ for $i  \in \llbracket 2t\rrbracket$ and $(v_{2t+1}, v_1) \in E(\Graph_{\rm low})$. It follows that
\beq\label{eq:odd-equality}
2\log \deg_\Graph(v_1) = \sum_{i=1}^{(t+1)} \log\left(\deg_\Graph(v_{2i-1}) \cdot \deg_\Graph(v_{2i})\right) -  \sum_{i=1}^{t} \log\left(\deg_\Graph(v_{2i}) \cdot \deg_\Graph(v_{2i+1})\right),
\eeq  
where for ease of writing we set $v_{2t+2} = v_1$. 

Now using the upper and lower bounds on the product of the degrees of the two end points of an edge in $\Graph_{\rm low}$ derived in Lemma \ref{lem:bad-edges-bounds} we have that
\[
\log c_0(\vep) + \log (n^2 p^\Delta) \le \log (\deg_\Graph(v_{j}) \cdot \deg_\Graph(v_{j+1})) \le \log C_0(\vep) + \log (n^2 p^\Delta) \qquad \text{ for all } j \in \llbracket 2t+1 \rrbracket.
\]
Therefore, setting 
\[
\wt c_\star(\vep):= \left(\min\left\{c_0(\vep), C_0(\vep)^{-1}\right\}\right)^{(t+1)/2},
\]
we find that
\beq\label{eq:deg-v-lbd-1}
\wt c_\star(\vep) n p^{\Delta/2} \le \deg_\Graph(v_1) \le \wt c_\star(\vep)^{-1} n p^{\Delta/2}.
\eeq
Using the cyclic nature of $C_{2t+1}$ and arguing similarly as in \eqref{eq:odd-equality} we further deduce that \eqref{eq:deg-v-lbd-1} continues to hold for the vertex $v_j$ for any $j \in \llbracket 2t+1\rrbracket$. Thus, if our chosen vertex $v=v_j$ for some $j \in \llbracket 2t+1\rrbracket$ this immediately yields \eqref{eq:deg-v-lbd}. 

Otherwise to derive \eqref{eq:deg-v-lbd} for the vertex $v$ we need a couple of lines of additional argument. Indeed, using the fact that ${\sf H}$ is connected we find that there exists a path $\mathscr{P}$ of length $\ell$, for some $\ell \ge 1$, connecting $v$ to $v_{j_0}$ for some $j_0 \in \llbracket 2t+1 \rrbracket$. For ease of writing, let us denote the vertices of $\mathscr{P}$ (in sequential order of their appearance) by $u_0, u_1, \ldots, u_\ell$, where $u_0=v$ and $u_{\ell}=v_{j_0}$.  
Note that
\[
\log \deg_\Graph(v)  + (-1)^{\ell-1}\log \deg_\Graph(v_{j_0}) = \sum_{i=1}^{\ell} (-1)^{i-1}\log (\deg_\Graph(u_{i-1}) \cdot \deg_\Graph(u_i)).
\]
By appealing to the tight upper and lower bounds on the product of the degrees of the end points in an edge in $\Graph_{\rm low}$ we derive from above that
\[
\log \deg_\Graph(v)  + (-1)^{\ell-1}\log \deg_\Graph(v_{j_0})  \ge \left(\left\lfloor {\ell}/{2}\right\rfloor +1\right)\cdot \log \left(\min \left\{ c_0(\vep), C_0(\vep)^{-1}\right\}\right) + \log (n^2 p^\Delta) \cdot {\bf 1}_{\{\ell \text{ is odd}\}}.
\]
As ${\sf H}$ contains a cycle $C_{2t+1}$ and a path of length $\ell$ that are edge disjoint we also have that
$2t+1 + \ell \le v_{\sf H}$. Therefore recalling that \eqref{eq:deg-v-lbd-1} holds for $v_{j_0}$ we now deduce from above that $\deg_\Graph(v) \ge c_\star(\vep) n p^{\Delta/2}$, where
\[
c_\star(\vep) := \left(\min\left\{c_0(\vep), C_0(\vep)^{-1}\right\}\right)^{v_{\sf H}/2}.
\]
This completes the proof of the claim \eqref{eq:deg-v-lbd}. Observe that this lower bound on the degree immediately yields the following upper bound on the number of non-isolated vertices of $\Graph\setminus \Graph_{\rm bad}$: 
\beq\label{eq:V-good-ubd}
|V(\Graph \setminus \Graph_{\rm bad})| \le \f{2e(\Graph)}{c_\star(\vep) np^{\Delta/2}}.
\eeq
Equipped with \eqref{eq:V-good-ubd} we next proceed to bound the number of strong-core graphs $\Graph$ with $e(\Graph)={\bm e}$  as follows:
\begin{enumerate}
\item[(1)] Choose non-isolated vertices of $\Graph\setminus \Graph_{\rm bad}$.
\item[(2)] Choose $e(\Graph_{\rm bad})$ edges arbitrarily out of all possible edges of the complete graph on $n$ vertices to construct $\Graph_{\rm bad}$.
\item[(3)] Choose ${\bm e} - e(\Graph_{\rm bad})$ edges out of $\binom{|V(\Graph\setminus \Graph_{\rm bad})|}{2}$ possible choices to construct $\Graph\setminus \Graph_{\rm bad}$. 
\item[(4)] Finally take a union over $e(\Graph_{\rm bad})$ in the allowable range $\llbracket 0, \vep {\bm e}\rrbracket$. 
\end{enumerate}
To implement steps (1)-(4) we need some bounds for which we make the following observations: 

Using \eqref{eq:copy-H-star-bd} and the fact that for any regular graph ${\sf H}$ its fractional independence number is $v_{\sf H}/2$ we find that for any graph $\Graph$ one has  $N({\sf H}, \Graph) \le (2 e(\Graph))^{v_{\sf H}/2}$. Since for any strong-core graph $\Graph$ we have 
\[
N({\sf H}, \Graph) \ge \delta(1-6\vep) n^{v_{\sf H}} p^{e({\sf H})} =  \delta(1-6\vep) \cdot (n^2 p^\Delta)^{v_{\sf H}/2},
\]
it immediately implies that the minimum number of edges of a strong-core graph, denoted hereafter by $\wh {\bm e}$, satisfies the lower bound 
\beq
\wh {\bm e} \ge \bar{\bm e}_0(\delta(1-6\vep)),
\eeq 
where for any $\delta_0 >0$ we set  
\beq\label{eq:e-0}
\bar {\bm e}_0(\delta_0):= \bar {\bm e}_0(\delta_0, {\sf H}):=\f12 \delta_0^{{2}/{v_{\sf H}}} n^2 p^\Delta.
\eeq
That is, $\bar {\bm e}_0(\delta_0, {\sf H})$ is the minimum number of edges a graph must possess to have at least $\delta_0 n^{v_{\sf H}} p^{e({\sf H})}$ labelled copies of ${\sf H}$. 

Note that for any ${\bm e} \ge \wh {\bm e}$, upon shrinking $c_\star(\vep)$, if necessary, one also has that
\beq\label{eq:e-half}
{\bm e} \le\f13 \cdot  \f{4 {\bm e}^2}{c_\star(\vep)^2 n^2 p^\Delta}.
\eeq
Now denote $\cI_{\bm e}$ to be set of strong-core graphs with ${\bm e}$ edges. Following the steps (1)-(4) to bound the cardinality of $\cI_{\bm e}$, we find that there exists some constant $C< \infty$, depending on $\delta$ and $\vep$, such that
\begin{multline}\label{eq:cI-e-bd}
|\cI_{\bm e}| \le \sum_{\mathfrak{e}\le \vep {\bm e}} \binom{n}{{2 {\bm e}}/({c_\star(\vep) np^{\Delta/2}})} \cdot \binom{n^2}{\mathfrak{e}} \cdot \binom{{4 {\bm e}^2}/({c_\star(\vep) np^{\Delta/2}})^2}{{\bm e} -\mathfrak{e} }\\
\le \sum_{\mathfrak{e} \le \vep {\bm e}}  \left(\f{e c_\star(\vep) n^2 p^{\Delta/2}}{2 {\bm e}}\right)^{\vep {\bm e}} \cdot n^{2\mathfrak{e}} \cdot \left(\f{4 e {\bm e}}{c_\star(\vep)^2 n^2 p^\Delta}\right)^{\bm e}  \le {\bm e} \cdot \left(\f1p\right)^{4 \Delta \vep {\bm e} } \cdot C^{{\bm e}} \le \left(\f1p\right)^{5 \Delta \vep {\bm e} },
\end{multline}
for any ${\bm e} \ge \wh{\bm e}$ and all large $n$. Let us explain the steps in \eqref{eq:cI-e-bd}. As $p =o(1)$ and ${\bm e} \le \bar C_\star n^2 p^\Delta$, using the fact that for nonnegative integers $y \le x$ the binomial coefficient $\binom{x}{y}$ is increasing for $y \le \lfloor x/2\rfloor$ the bound on the number of ways to choose the non-isolated vertices of $\Graph\setminus \Graph_{\rm bad}$ follows from  \eqref{eq:V-good-ubd}. Using   \eqref{eq:e-half} and the same reasoning as above we have the third term in the second step. Moreover, we use Stirling's approximation to obtain the first term there. Whereas, the third step uses that ${\bm e} \le \bar C_\star n^2 p^\Delta$ and $p \le n^{-1/\Delta}$. Finally to obtain the last inequality above we recall that $p=o(1)$. 

Equipped with \eqref{eq:cI-e-bd} we now take a union bound over ${\bm e} \in \llbracket \wh{\bm e}, \bar C_\star n^2 p^2 \rrbracket$ to find that
\begin{multline*}
 \log \P(\exists \text{ a strong-core graph})\\
  \le  \log\left(\sum_{{\bm e} = \wh{\bm e}}^{\bar C_\star n^2 p^\Delta} \P(\Graph \subset \G(n,p): \Graph \in \cI_{\bm e})\right)   = \log \left( \sum_{{\bm e} = \wh{\bm e}}^{\bar C_\star n^2 p^\Delta} p^{{\bm e} (1-5 \Delta \vep)} \right) \\
 \qquad \qquad \qquad \qquad \qquad   \qquad  \qquad \qquad  \le \log 2 - (1-5\Delta \vep ) \cdot \wh{\bm e} \cdot \log (1/p)\\
\le  \log 2 -  \f12 \delta^{{2}/{v_{\sf H}}} \cdot (1-\gf_{{\sf H}}(\vep)) \cdot n^2 p^\Delta  \log(1/p),
\end{multline*}
for all large $n$, where $\gf_{\sf H}(\vep):= 1 - (1-6\vep)^{{2}/{v_{\sf H}}} \cdot(1-5 \Delta \vep)$. Dividing both sides by $n^2 p^\Delta \log (1/p)$ and then sending $n$ to infinity the proof completes for non-bipartite graphs.
 \end{proof}


When ${\sf H}$ is a bipartite graph it does not have any cycle of odd length. Therefore, in that case we lose the identity \eqref{eq:odd-equality}. Hence one cannot repeat the above argument. In fact, as already mentioned earlier, one can have strong-core graphs with many of its vertices having small degrees and in particular \eqref{eq:deg-v-lbd} does not hold for such graphs. Thus one needs to follow the route outlined in Section \ref{sec:outline}. Recall from there that we split a strong-core graph $\Graph$ into two subgraphs: a bipartite subgraph $\Graph_\gb$ and the rest $\Graph\setminus \Graph_\gb$.  

\corAB{Upon assuming a lower bound on the number of copies of ${\sf H}$ in $\Graph_\gb$ we next derive a lower bound on the difference on the number of edges of $\Graph_\gb$ and an appropriately chosen multiple of its number of vertices of low degree (the appropriateness of this choice will be clear later during the proof of Proposition \ref{prop:strong-core}). This bound would lead to establishing that $\Graph_\gb$ is entropically stable. This is the content of the following section.}

\subsection{\corAB{Lower bound on the difference of the number of edges and vertices}}\label{sec:G_b} 
To prove Proposition \ref{prop:strong-core} for bipartite graphs we will need to be able to use the tight upper and lower bounds on the product of the degrees of Lemma \ref{lem:bad-edges-bounds}. Therefore, similar to the last section, we need to work with $\Graph_{\rm low}$ again. It will be verified below that any bipartite subgraph $\bar \Graph$ of it  must satisfy the following property. 

\begin{ass}\label{ass:bipartite-ass}
Let $\bar \Graph$ be a bipartite graph with two parts $U_1$ and $U_2$, i.e.~$E(\bar \Graph) \subset U_1 \times U_2$ for some disjoint set of vertices $U_1$ and $U_2$. Assume that there exists $U_{1,1} \subset U_1$ such that 
\beq\label{eq:deg-2-lbd}
\min_{u \in U_{1,1}} \deg_{\bar \Graph}(u) \ge \Delta,
\eeq
for some $\Delta \ge 2$, and 
\beq\label{eq:deg-1-ubd}
|U_1\setminus U_{1,1}| \le \vep_2 e(\bar \Graph),
\eeq
for some $\vep_2 \in (0,1/\Delta)$.
\end{ass}

Under the above assumption we find a lower bound $e(\bar \Graph) - (\Delta/2) \cdot |U_1|$ which will be used later in showing the entropic stability of $\Graph_\gb$. 

\begin{lem}\label{lem:small-count}
Fix $ \delta_1>0$ and a $\Delta$-regular connected bipartite graph ${\sf H}$ with $\Delta \ge 2$. Let $\bar \Graph$ be a bipartite graph with parts $U_1$ and $U_2$ satisfying Assumption \ref{ass:bipartite-ass}.  If 
\beq\label{eq:four-cycle-bipartite-bd-1}
 N({\sf H}, \bar \Graph) \ge \delta_1 n^{v_{\sf H}} p^{e({\sf H})}
\eeq
 then
\[
e(\bar \Graph) -\f{\Delta}{2} \cdot|U_1| \ge \f12 \delta_1^{2/v_{\sf H}} \cdot(1-\Delta \vep_2) \cdot n^2 p^\Delta.
\]
\end{lem}

The proof of Lemma \ref{lem:small-count} uses the fact that any bipartite graph admits a {\em perfect matching}. Let us define the notion of {matching} in a graph. 

\begin{dfn}[Matching]\label{dfn:matching}
Let ${\sf H}$ be a graph. A subgraph ${\sf M} \subset {\sf H}$ is said to be a matching if the degree of every non-isolated vertex of ${\sf M}$ is one. Thus ${\sf M}$ is a subgraph spanned by some collection of vertex disjoint edges of ${\sf H}$. Furthermore, ${\sf M}$ is said to be a perfect matching if  every vertex of ${\sf H}$ is incident to some edge in ${\sf M}$.
\end{dfn}

In graph theory it is customary to treat a matching as a set of vertex disjoint edges. However, it will be convenient for us to treat a matching as a subgraph of the original graph. The following result shows that any bipartite graph admits a perfect matching that avoids a collection of pre-specified edges. Its proof essentially follows from K\"{o}nig's edge coloring theorem for bipartite graphs and is deferred to Appendix \ref{sec:decompose}.

\begin{lem}\label{lem:perfect-match}
Let ${\sf H}$ be a $\Delta$-regular bipartite graph with $\Delta \ge 2$. Let $\{e_1, e_2, \ldots, e_{\Delta-1}\}$ be any collection of $(\Delta-1)$ edges of ${\sf H}$. Then there exists a perfect matching ${\sf M}$ of ${\sf H}$ such that $e_j \notin E({\sf M})$ for all $j \in \llbracket \Delta -1 \rrbracket$. 
\end{lem}

We are now ready to prove Lemma \ref{lem:small-count}. 

\begin{proof}[Proof of Lemma \ref{lem:small-count}]
Since ${\sf H}$ is a bipartite graph its vertices can be partitioned into $V^{1}({\sf H})$ and $V^{2}({\sf H})$ so that both $V^{1}({\sf H})$ and $V^2({\sf H})$ are independent sets. Furthermore, as ${\sf H}$ is a regular graph it follows that 
\[
|V^1({\sf H})| = V^2({\sf H})| = \f{v_{\sf H}}{2} =: t. 
\]
For convenience, let us write $V^1({\sf H})= \{1,2,\ldots, t\}$ and $V^2({\sf H})= \{\bar 1,\bar 2,\ldots, \bar t\}$. Let $\{\bar s_1, \bar s_2, \ldots, \bar s_\Delta\}$ be the neighbors of the vertex $t$ in the graph ${\sf H}$. By Lemma \ref{lem:perfect-match} there exists a perfect matching ${\sf M}$ of ${\sf H}$ that avoids the edges $\{e_j:= (t, s_{j+1}), \, j \in \llbracket \Delta -1 \rrbracket\}$. Let $\wh{\sf M}$ be the subgraph of ${\sf H}$ obtained from ${\sf M}$ by adding the edges $\{e_j, j \in \llbracket \Delta -1 \rrbracket\}$. By construction one has that 
\beq\label{wh-sf-M}
N({\sf H}, \bar\Graph) \le N(\wh{\sf M}, \bar\Graph). 
\eeq
Therefore, to complete the proof it suffices to show that a lower bound on $N(\wh{\sf M}, \bar \Graph)$ yields the desired lower bound on $e(\bar \Graph) - (\Delta/2) \cdot |U_1|$.

Turning to do the same, we fix ${\bm u} = \{u_1, u_2, \ldots, u_t\} \in U_1^t$, a set of $t$ distinct vertices from $U_1$, and denote $N_1(\wh{\sf M}, \bar \Graph, {\bm u})$ to be the number of copies of $\wh{\sf M}$ in $\Graph$ where the vertex $u_i$ gets mapped to the vertex $i \in V^1({\sf H})$, for $i \in \llbracket t \rrbracket$. We now derive an upper bound $N_1(\wh{\sf M}, \bar \Graph, {\bm u})$ in terms of the degree sequence $\{\deg_{\bar \Graph}(u_i)\}_{i=1}^t$. 

Notice that, as ${\sf M}$ is a perfect matching, there exists a bijection $\pi: V^1({\sf H}) \mapsto V^2({\sf H})$ such that only the edges $(i, \pi(i)) \in E({\sf M})$ for all $i \in \llbracket t \rrbracket$. By construction the edges $e_j \notin E({\sf M})$ for $j \in \llbracket \Delta-1\rrbracket$. Since ${\sf M}$ is a perfect matching we therefore must have that $\pi(t) = \bar s_1$. 

We find a bound on $N_1(\wh{\sf M}, \bar \Graph, {\bm u})$ as follows: Let ${\sf C}$ be any copy of $\wh{\sf M}$ in $\bar \Graph$ with $v_1, v_2, \ldots, v_t \in U_2$ so that $v_j$ is mapped to the vertex $\bar j \in V^2({\sf H})$, for $j \in \llbracket t \rrbracket$. Then
\begin{itemize}
\item For $i \in \llbracket t-1 \rrbracket$ the vertex $v_{\pi(i)}$ must be connected to $u_i$. Hence, for each such $i$ the choice of the vertex $v_{\pi(i)}$ is at most $\deg_{\bar \Graph}(u_i)$.

\item The vertex $v_{s_1}$ must be connected to $u_t$. As the edges $(t, \bar s_{j+1}) \in E(\wh{\sf M})$ for $j \in \llbracket \Delta -1 \rrbracket$ the vertices $v_{s_j}$ must also be neighbors of $u_t$, for $j=2,3,\ldots, \Delta$. Since the vertices $\{v_j\}_{j=1}^\Delta$ must be distinct the number of choices for $v_{s_1}$ is at most $\deg_{\bar \Graph}(u_t)-(\Delta-1)$.
\end{itemize}

Thus combining the above bounds we deduce that 
\begin{align}\label{C-2l-sum}
\sum_{{\bm u} \in U_1^t} N(\wh{\sf M}, \bar \Graph, {\bm u}) & \le \sum_{{\bm u} \in U_1^t}  \left[ \prod_{i=1}^{t-1}\deg_{\bar \Graph}(u_i) \right]\cdot (\deg_{\bar \Graph}(u_t) -(\Delta-1)) \notag\\
& \le \left(\sum_{u \in U_1} \deg_{\bar \Graph}(u) \right)^{t-1} \cdot \left[ \sum_{u \in U_1} \left(\deg_{\bar \Graph}(u)-\f\Delta2\right)\right]
 = e(\bar \Graph)^{t-1} \cdot \left(e(\bar \Graph) - \f\Delta2|U_1|\right),
\end{align}
where in the penultimate inequality we used the lower bound $\Delta \ge 2$, and the last equality is a consequence of the fact that $\bar \Graph$ is a bipartite graph. Furthermore, using \eqref{eq:deg-2-lbd}-\eqref{eq:deg-1-ubd} we see that
\[
\Delta |U_1| \le \sum_{u \in U_1} \deg_{\bar \Graph}(u)  + \Delta  |U_1 \setminus U_{1,1}| \le (1+\Delta \vep_2)e(\bar \Graph)
\]
and therefore
\[
e(\bar \Graph) \le \f{2}{1-\Delta \vep_2} \left(e(\bar \Graph) - \f\Delta2|U_1|\right).
\]
Plugging this bound in \eqref{C-2l-sum} we find that
\beq\label{C-2l-sum-1}
 \sum_{{\bm u} \in U_1^t} N_1(\wh{\sf M}, \bar \Graph, {\bm u}) \le \left(\f{2}{1-\Delta \vep_2}\right)^{t-1} \cdot \left(e(\bar \Graph) - \f\Delta2|U_1|\right)^{t}.
\eeq
To finish the proof let us denote $N_2(\wh{\sf M}, \bar \Graph, {\bm u})$ to be the number of copies of $\wh{\sf M}$ in $\Graph$ where the vertex $u_i$ gets mapped to the vertex $\bar i \in V^2({\sf H})$, for $i \in \llbracket t \rrbracket$. Reversing the roles of $V^1({\sf H})$ and $V^2({\sf H})$ we see that the bound \eqref{C-2l-sum} continues to hold when $N_1(\wh{\sf M}, \bar \Graph, {\bm u})$ is replaced by $N_2(\wh{\sf M}, \bar \Graph, {\bm u})$. Therefore, upon observing that
\[
 \sum_{{\bm u} \in U_1^t} \left[N_1(\wh{\sf M}, \bar \Graph, {\bm u}) + N_2(\wh{\sf M}, \bar \Graph, {\bm u}) \right] = N(\wh{\sf M}, \bar \Graph)
\]
the proof completes from \eqref{wh-sf-M} and \eqref{eq:four-cycle-bipartite-bd-1}. 
\end{proof}

\corAB{In the next section we derive a} bound on the number of core graphs (and hence also for strong-core graphs) in terms of its number of edges and the number of vertices of small degree. This combinatorial lemma will be used to derive the entropic stability of $\Graph\setminus \Graph_\gb$. 

\subsection{\corAB{Bound on the number of core graphs}}\label{sec:G-G_b}
\corAB{The following is main result of this section.}
\begin{lem}\label{lem:large-count}
Fix $\vep \in (0,1)$, ${\sf H}$ a $\Delta$-regular graph with $\Delta \ge 2$, and integers $\gD:=\gD(\vep)$ and $D:=D(\vep):= \lceil 16\Delta /\vep \rceil$ such that $\gD \ge D$. Let
\[
 \cV_1:= \{v \in V(\Graph): \deg_\Graph(v) \le \gD\}
\]
and set $\ol \cV_1:=V(\Graph)\setminus  \cV_1$. 
Further let $\cN_0({\bm e}, {\bm v}, \gD)$ be the number of core graphs with $e(\Graph) = {\bm e}$ and $|\cV_1| = {\bm v}$. 
Then, for any $p \le n^{-1/\Delta}$ and all large $n$, 
\[
\cN_0({\bm e},{\bm v}, \gD) \le \binom{n}{{\bm v}}\cdot \exp( \vep {\bm e} \log(1/p)). 
\]
\end{lem}


\begin{proof}
We split the proof into two parts. First let us consider the easier case $n p^{\Delta/2} = O((\log n)^{v_{\sf H}})$, and then we consider the case $np^{\Delta/2} = \Omega((\log n)^{v_{\sf H}})$ ({for two sets of positive reals $\{a_n\}$ and $\{b_n\}$ the notation $a_n = \Omega(b_n)$ means $\liminf_{n \to \infty} a_n/ b_n >0$.}), where we remind the reader that $v_{\sf H}$ is the number of vertices in ${\sf H}$. The proof for the latter regime is more involved. 

Since $\min_{v \in \ol \cV_1} \deg_{\Graph}(v) \ge \gD \ge D$ we have that
\[
|\ol \cV_1| \le 2 {\bm e} /D,
\]
where $e(\Graph)={\bm e}$. Thus, using the lower bound on $D$, we find that the number of ways to choose the vertices in $\ol\cV_1$ can be bounded by
\beq\label{eq:large-deg-vertex-bd}
\sum_{{\bm u}=0}^{\min\{2{\bm e}/D, n\}} n^{\bm u}\le n^{4{\bm e}/D} \le n^{\f{\vep}{4 \Delta}{\bm e}} \le p^{-\f\vep4 {\bm e}},
\eeq
where in the last step we use the fact that $p \le n^{-1/\Delta}$. For ease of writing let us denote $\Graph_1\subset \Graph$ be the subgraph induced by the edges in $E(\Graph)$ that are incident to some vertex in $\cV_1$. 

Next we note that the number of ways to choose the edges of $\Graph$ that are not adjacent to any vertex in $\cV_1$ (and hence both end points must be in $\ol \cV_1$) can be trivially bounded by
\beq\label{eq:edge-bd-1}
|\ol \cV_1|^{2 {\bm e}}  \le \exp\left( O(\log \log n) \cdot {\bm e}\right) \le \exp\left(\f\vep8 {\bm e} \log(1/p)\right),
\eeq
for all large $n$, where we have used the fact that $np^{\Delta/2} = O((\log n)^{v_{\sf H}})$ and $|V(\Graph)| \le {\bm e} = O(n^2 p^\Delta \log (1/p))$. Now we need to bound the number of ways to choose the edges of $\Graph_1$, which we denote by ${\bm e}_1$. It is easy to see that, applying Stirling's approximation, this can be bounded by
\beq\label{eq:edge-bd-2}
\binom{{|\cV_1|} \cdot |V(\Graph) |}{{\bm e}_1} \le \left(2 e |V(\Graph) |\right)^{{\bm e}_1} \le \exp\left(\f\vep8 {\bm e} \log(1/p)\right),
\eeq
for all large $n$, where we have used the fact that ${|\cV_1|} \le 2 {\bm e}_1$. Finally the number of ways to choose the vertices in ${\cV_1}$ such that $|\cV_1|={\bm v}$ is bounded by $\binom{n}{{\bm v}}$. Therefore, combining the bounds in \eqref{eq:large-deg-vertex-bd}-\eqref{eq:edge-bd-2} we derive that the number of core graphs with $|\cV_1|={\bm v}$, $e(\Graph_1) = {\bm e}_1$ and $e(\Graph)={\bm e}$ is bounded by 
\[
\binom{n}{{\bm v}} \cdot \exp\left(\f{\vep}{2} {\bm e} \log(1/p)\right).
\]
Since
\[
\log ({\bm e}_1) \le \log ({\bm e}) \le \f\vep 2 {\bm e} \log(1/p),
\]
for all large $n$, finally taking an union bound over the ranges of  ${\bm e}_1$ we derive the desired upper bound on $\cN_0({\bm e}, {\bm v},\gD)$. 


Next we consider the regime $np^{\Delta/2} = \Omega((\log n)^{v_{\sf H}})$. As we have already seen in \eqref{eq:large-deg-vertex-bd} that the choices of the number of vertices of $\ol \cV_1$ can be adequately bounded for the entire regime $p \le n^{-1/\Delta}$. However, the arguments used to derive the bounds \eqref{eq:edge-bd-1}-\eqref{eq:edge-bd-2} becomes ineffective when $np^{\Delta/2}$ is a fractional power of $n$. In this regime, we use the bound derived in Lemma \ref{lem:bad-graph-bd} to deduce that any two arbitrary vertices cannot be connected. This significantly reduces the cardinality of the possible edge set.   

Turning to implement this idea we split the vertices in $\ol \cV_1$ as follows: For $k=1,2,\ldots, k_\star$, we consider the following nested sequence of sets of vertices
\[
W_{k_\star} \subset W_{k_\star-1} \subset \cdots \subset W_2 \subset W_1,
\]
where
\[
W_k:= \{v \in \ol \cV_1: \deg_\Graph(v) \ge 2^{k-1} \gD\},
\]
and
\[
k_\star := \left\lfloor \f{\log (n/\gD)}{\log 2} \right\rfloor.
\]
For ease of writing we set $W_{k_\star+1}:=\emptyset$. Let $u \in W_{i_0-1}\setminus W_{i_0}$ for some $i_0 \ge 2$. This implies that $\deg_\Graph(u) \le 2^{i_0-1} \gD$. Therefore, if $(u,v) \in E(\Graph)$ for some $v \in V(\Graph)$ then using Lemma \ref{lem:bad-graph-bd} we have that 
\beq\label{eq:lem-lbd}
\deg_\Graph(u) \cdot \deg_\Graph(v) \ge \f{\wt c_0(\vep) \cdot e(\Graph)}{(\log n)^{v_{\sf H}}}.
\eeq
We claim that the above implies that $v \in W_{j_0}$, where $j_0:=j_0(i_0)$ is the smallest integer satisfying
\beq\label{eq:j-0-def}
2^{i_0+j_0-1} \gD^2 \ge \f{\wt c_0(\vep) \cdot e(\Graph)}{(\log n)^{v_{\sf H}}}.
\eeq
If not, then 
\[
\deg_\Graph(u) \cdot \deg_\Graph(v) \le 2^{i_0+j_0-2} \gD^2 < \f{\wt c_0(\vep) \cdot e(\Graph)}{(\log n)^{v_{\sf H}}},
\]
where the last step is a consequence of the definition of $j_0(i_0)$. However, this contradicts \eqref{eq:lem-lbd}.

A similar argument also shows that if $u \in \cV_1=\{v \in V(\Graph): \deg_\Graph(v) \le \gD\}$ then any of its adjacent vertices must be in $W_{j_\star}$, where $j_\star$ is the smallest integer satisfying
\beq\label{eq:j-star-def}
2^{j_\star} \gD^2 \ge \f{\wt c_0(\vep) \cdot e(\Graph)}{(\log n)^{v_{\sf H}}}.
\eeq
\corAB{Note that the number of edges in a core graph must be $\Omega(n^2 p^\Delta)$ (this is again a consequence of \eqref{eq:copy-H-star-bd}). Therefore, as by assumption $np^{\Delta/2} = \Omega((\log n)^{v_{\sf H}})$ we have $j_\star \gg 1$. This in particular implies that any vertex in $\cV_1$ cannot connect to another vertex in $\cV_1$.} Furthermore, it is easy to note that for any $i \ge 1$. 
\beq\label{eq:markov-vertex}
|W_i| \le (4/\gD) \cdot 2^{-i} \cdot e(\Graph).
\eeq
Equipped with these observations we bound the number of core graphs with $|\cV_1|={\bm v}$ and $e(\Graph) = {\bm e}$ as follows:
\begin{enumerate}
\item[(1)] Choose the vertices in $\cV_1$.
\item[(2)] Choose the vertices in $\{W_{i_0-1}\setminus W_{i_0}\}_{i_0=2}^{k_\star+1}$. 
\item[(3)] Choose ${\bm e}$ edges from the set of possible edges
\[
\gM:=\left\{(u,v): u \in \cV_1, v \in W_{j_\star}  \right\} \cup \bigcup_{i_0=2}^{k_\star+1}   \left\{(u,v): u \in W_{i_0-1}\setminus W_{i_0}, v \in W_{j_0(i_0)}  \right\}.
\]
\end{enumerate}
Let us find a  bound on the cardinality of $\gM$. 
Using the bound $|\cV_1| \le {\bm e}$, and \eqref{eq:j-star-def}-\eqref{eq:markov-vertex} we find that
\[
\left| \left\{(u,v): u \in \cV_1, v \in W_{j_\star}  \right\} \right| \le {\bm e} \cdot |W_{j_\star}| \le \f{4\gD (\log n)^{v_{\sf H}}}{\wt c_0(\vep)} \cdot {\bm e}. 
\]
Similarly using \eqref{eq:j-0-def} and \eqref{eq:markov-vertex}, for any $i_0 \ge 2$ we deduce that 
\[
\left|\left\{(u,v): u \in W_{i_0-1}\setminus W_{i_0}, v \in W_{j_0(i_0)}  \right\}\right| \le \f{8(\log n)^{v_{\sf H}}}{\wt c_0(\vep)} \cdot {\bm e}.
\]
Thus 
\beq\label{eq:edgeset-gM}
|\gM| \le (k_\star+1) \cdot \f{8\gD (\log n)^{v_{\sf H}}}{\wt c_0(\vep)} \cdot {\bm e} \le \f{8 \gD (\log n)^{v_{\sf H}+1}}{\log 2 \cdot \wt c_0(\vep)} \cdot {\bm e}. 
\eeq
Next we aim to obtain a bound on the number of choices of the vertices belonging to $\{W_{i_0-1}\setminus W_{i_0}\}_{i_0=2}^{k_\star+1}$. Using \eqref{eq:markov-vertex} we find that, for any $i_0=2,3,\ldots, k_\star+1$,
\[
{\bm w}_{i_0}:= |W_{i_0-1}\setminus W_{i_0}| \le (8/\gD) \cdot 2^{-i_0} \cdot {\bm e}.
\]
Thus the number of ways to choose the vertices $\{W_{i_0-1}\setminus W_{i_0}\}_{i_0=2}^{k_\star+1}$ is bounded by 
\beq\label{eq:large-deg-vertex-bd-w}
\prod_{i_0=2}^{k_\star+1} \left(\sum_{{\bm w}_{i_0} =0}^{ \lfloor (8/\gD) \cdot 2^{-i_0} \cdot {\bm e}\rfloor} n^{{\bm w}_{i_0}}\right) \le \prod_{i_0=2}^{k_\star+1} n^{(16/\gD) \cdot 2^{-i_0} \cdot {\bm e}} \le n^{8 {\bm e}/D} \le p^{-\f\vep2 {\bm e}},
\eeq
where in the last step we again use the fact that $p \le n^{-1/\Delta}$ and also the lower bound on $\gD\geq D =\lceil 16\Delta /\vep \rceil$. 
Therefore proceeding as in steps (1)-(3) and applying \eqref{eq:edgeset-gM}-\eqref{eq:large-deg-vertex-bd-w} we now derive hat
\begin{multline*}
\cN_0({\bm e}, {\bm v}, \gD) \le \binom{n}{{\bm v}}  \cdot p^{-\f\vep2 {\bm e}} \cdot \binom{|\gM|}{{\bm e}} \le  \binom{n}{{\bm v}}  \cdot p^{-\f{\vep}{2} {\bm e}} \cdot \left(\f{16e\gD(\log n)^{v_{\sf H}+1}}{\wt c_0(\vep) \cdot \log 2}\right)^{\bm e} \le \binom{n}{{\bm v}} \cdot p^{-{\vep} {\bm e}},
\end{multline*}
for all large $n$, where in the last step we once again use $p \le n^{-1/\Delta}$. This completes the proof of the lemma. 
\end{proof}

\corAB{
Finally in the following section, upon combining the results of this and previous section we prove the entropic stability of the strong-core graphs for regular connected bipartite graphs.}  
\subsection{Proposition of \ref{prop:strong-core} for bipartite graphs}\label{sec:proof-sc-even}
Before going to the proof of Proposition \ref{prop:strong-core} we remind the reader that we would like to choose $\Graph_\gb$ in such a way so that almost all copies of ${\sf H}$ are either completely contained in $\Graph_\gb$ or in $\Graph\setminus \Graph_\gb$. This necessitates the following decomposition of the vertices of the subgraph $ \Graph_{\rm low}:=\Graph_{\rm low}(\vep)$. 

\subsection*{Decomposition of the vertex set} Let ${\sf H}$ be a $\Delta$-regular connected bipartite graph with $\Delta \ge 2$. For ease in writing denote $t:= \f{v_{\sf H}}{2}$. Let $D:=D(\vep)$ be as in Lemma \ref{lem:large-count}. Set
\beq\label{eq:dfn-C_3}
C_3:=C_3(\vep) := (t-1) \left(\left\lceil \f{(2 \bar C_\star)^t}{\delta \vep}\right\rceil +2\right).
\eeq
Now for $i \in \llbracket C_3 \rrbracket$ define
\[
V_i:=\{ v \in V( \Graph_{\rm low}): D_{i-1}+1 \le\deg_\Graph(v) \le D_i\},
\]
where 
\[
D_i:=D_i(\vep):= D \cdot \left( \f{C_0(\vep)}{c_0(\vep)} \right)^{i-1},
\]
with $C_0(\vep)$ and $c_0(\vep)$ as in Lemma \ref{lem:bad-edges-bounds}, and we set $D_0:=0$. Denote
\[
\wt V_1:= \{ v \in V( \Graph_{\rm low}): (u,v) \in E( \Graph_{\rm low}) \text{ for some } u \in V_1\}.
\]
For $i =2,3,\ldots, C_3$, we then iteratively define
\[
\wt V_i:= \{ v \in V( \Graph_{\rm low}): (u,v) \in E( \Graph_{\rm low}) \text{ for some } u \in V_i\}\setminus \wt V_{i-1},
\]
and let
\[
V_\sharp:= V( \Graph_{\rm low}) \setminus \left(\bigcup_{i=1}^{C_3} V_i \cup \wt V_i\right).
\]
For ease of writing, for $i \in \llbracket C_3\rrbracket$ let us also denote $\Graph_{i,{\rm g}}$ to be subgraph spanned by the edges that are incident to some vertex in $\cup_{j=1}^i V_j$, and $\bar \Graph_{i, {\rm g}}$ to be its complement graph when $\Graph_{i, {\rm g}}$ is viewed as a subgraph of $\Graph_{\rm low}$. 

Equipped with the above notation we now proceed to the proof of Proposition \ref{prop:strong-core} when ${\sf H}$ is a connected regular bipartite graph.

\begin{proof}[Proof of Proposition \ref{prop:strong-core} for bipartite graphs]
We begin the proof by first identifying the subgraph $\Graph_\gb$ having the desired property mentioned above. To this end, for any strong-core graph $\Graph$, as $e(\Graph) \le \bar C_\star n^2 p^\Delta$, it follows from \eqref{eq:copy-H-star-bd} and the observation $\upalpha_{\sf H}= \f{v_{\sf H}}{2}=t$ that
\beq\label{eq:ubd-c-2l}
N({\sf H}, \Graph) \le (2\bar C_\star n^2 p^\Delta)^t = (2\bar C_\star)^t \cdot n^{v_{\sf H}} p^{e({\sf H})}.
\eeq
Therefore there exists $i_\star \in \llbracket C_3-t+1 \rrbracket$ such that
\beq\label{eq:cross-term-bd-1}
N({\sf H}, \Graph_{i_\star+(t-1), {\rm g}}) - N({\sf H}, \Graph_{i_\star, {\rm g}}) \le \vep \delta n^{v_{\sf H}} p^{e({\sf H})}.
\eeq
Otherwise, as $\Graph_{C_3, {\rm g}} \subset \Graph$ and $t \ge 2$,
\begin{multline*}
N({\sf H}, \Graph) \ge \left[\sum_{i=1}^{C_3/(t-1) -1} N({\sf H}, \Graph_{(t-1)i+1, {\rm g}}) - N({\sf H}, \Graph_{(t-1)(i-1)+1, {\rm g}}) \right]+ N({\sf H}, \Graph_{1, {\rm g}}) \\
\ge \left(\f{C_3}{t-1}-1\right) \vep \delta n^{v_{\sf H}} p^{e({\sf H})} > (2 \bar C_\star)^t n^{v_{\sf H}} p^{e({\sf H})},
\end{multline*}
yielding a contradiction to \eqref{eq:ubd-c-2l}, where the last inequality follows by recalling the definition of $C_3$ (see \eqref{eq:dfn-C_3} above). Next we make the another observation:

\begin{claim}\label{claim:cross-term}
Fix $i \in \llbracket C_3 - t+1\rrbracket$. Any copy of ${\sf H}$ in $\Graph_{\rm low}$ that uses an edge of $\Graph_{i, {\rm g}}$ must be contained in $\Graph_{i+(t-1), {\rm g}}$. 
\end{claim}

Equipped with Claim \ref{claim:cross-term} we note that any copy of ${\sf H}$ that uses edges of both $\Graph_{i_\star, {\rm g}}$ and $\bar \Graph_{i_\star, {\rm g}}$ must be contained in $\Graph_{i_\star+(t-1), {\rm g}}$ but not in $\Graph_{i_\star, {\rm g}}$. Hence by \eqref{eq:cross-term-bd-1} the number of such copies of ${\sf H}$ is at most $\vep \delta n^{v_{\sf H}} p^{e({\sf H})}$. Observe that any labelled copy of ${\sf H}$ in $\Graph_{\rm low}$ must be either contained in $\Graph_{i_\star, {\rm g}}$, or $\bar \Graph_{i_\star, {\rm g}}$, or must use edges of both $\Graph_{i_\star, {\rm g}}$ and $\bar \Graph_{i_\star, {\rm g}}$. Therefore from Lemma \ref{lem:bad-edges-bounds}(b) it now follows that 
\begin{align}\label{eq:cross-term-bd-2}
N({\sf H}, \Graph_{i_\star, {\rm g}}) + N({\sf H}, \bar \Graph_{i_\star, {\rm g}})&  \ge N({\sf H}, \Graph_{\rm low}) - \vep \delta n^{v_{\sf H}} p^{e({\sf H})}\notag\\
& \ge (1-\vep)N({\sf H}, \Graph) -\vep \delta n^{v_{\sf H}} p^{e({\sf H})} \ge\delta(1-8\vep) n^{v_{\sf H}} p^{e({\sf H})},
\end{align}
where in the last step we use the fact that $\Graph$ is a strong-core graph. 

Thus setting $\Graph_\gb=\Graph_{i_\star, {\rm g}}$ we indeed have that almost all the copies of ${\sf H}$ in $\Graph$ are either contained in $\Graph_\gb$ in $\Graph\setminus \Graph_\gb$. However, we note that this splitting procedure is not identical for {\em all }strong-core graphs, which is captured by the existence of {\em some} $i_\star \in \llbracket C_3\rrbracket$ that may very well vary for different graphs.  This indeterminacy in the parameter $i_\star$ results in another union bound. Nevertheless, as we will see below this additional bound turns out to be harmless for our purpose. 

Before proceeding further let us prove Claim \ref{claim:cross-term}. Turning to do this we fix an edge $e=(u, \wt u) \in E(\Graph_{i_0, {\rm g}}) \subset E(\Graph_{\rm low})$ for some $i_0 \in \llbracket C_3-t+1\rrbracket$. Without loss of generality assume that $u \in \cup_{j=1}^{i_0} V_j$. From definition of the set $V_j$ we have that
\(
\deg_\Graph(u) \le D_{i_0}
\)
and therefore from Lemma \ref{lem:bad-edges-bounds}(a) it follows that
\beq\label{eq:deg-wt-u-lbd}
\deg_\Graph(\wt u) \ge (c_0(\vep)/D_{i_0}) \cdot n^2 p^\Delta.
\eeq
Let $w$ be any vertex adjacent to $\wt u$ in $\Graph_{\rm low}$. Using Lemma \ref{lem:bad-edges-bounds}(b) we deduce that
\[
\deg_\Graph(w) \le \left(\f{C_0(\vep)}{c_0(\vep)}\right) \cdot D_{i_0} = D_{i_0+1},
\]
which in particular implies that $w \in \cup_{j=1}^{i_0+1} V_j$. This shows that any edge adjacent to some edge in $\Graph_{i_0, {\rm g}}$ must be in $\Graph_{i_0+1, {\rm g}}$.

Next, ${\sf H}$ being a connected regular bipartite graph we observe that it is $2$-vertex connected, i.e.~removal of one vertex does not make it disconnected (see \cite[Chapter 3.7, Exercise 7.4]{BaRa}). Fix any vertex $v$ and and edge $e$ in ${\sf H}$. From \cite[Exercise 3.7(a)]{BaRa} we have that $e$ and $v$ must be in some common cycle. As the length of any cycle in ${\sf H}$ is at most $v_{\sf H}$ ($= 2t$), denoting
\[
{\rm diam}_E({\sf H}):= \max_{e  \in E({\sf H}), v \in V({\sf H})} {\rm dist}_E(e, v),
\]
where ${\rm dist}_E(e,v)$ denotes the minimum length of the path in ${\sf H}$ needed to connect an end of $e$ to $v$, we find that ${\rm diam}_E({\sf H}) \le t-1$. 

Now fix a labelled copy $\mathscr{H}$ of ${\sf H}$ that uses an edge  $e_\star=(w_1, w_2) \in E(\Graph_{i, {\rm g}})$. Without loss of generality, assume that $w_1 \in \cup_{j=1}^i V_j$. Let $e_\star'$ be any other edge in $\mathscr{H}$.  We have from above that 
\[
{\rm dist}_E(e_\star', w_1) \le {\rm diam}_E(\mathscr{H}) \le {\rm diam}_E({\sf H}) \le t-1.
\]
This means that there is a path $\mathscr{P}$ of length at most $t$ such that the starting vertex is $w_1$ and the last edge is $e_\star'$. As $w_1 \in \cup_{j=1}^i V_j$ the first edge of $\mathscr{P}$ is in $\Graph_{i, {\rm g}}$. Since we noted above that  an edge that is adjacent to some edge in $\Graph_{i_0, {\rm g}}$ must be in $\Graph_{i_0+1, {\rm g}}$ for any $i_0 \in \llbracket C_3 -t+1 \rrbracket$ it is now immediate, by induction, that $e_\star' \in E(\Graph_{i+t-1,{\rm g}})$. This shows that any edge of $\mathscr{H}$ must be in $\Graph_{i+t-1,{\rm g}}$ and thus Claim \ref{claim:cross-term} is proved.

We now return to the proof of the proposition. Observe that \eqref{eq:cross-term-bd-2} implies that
\beq\label{eq:sc-2ell-union-1}
\left\{ \exists \text{ a strong-core graph}\right\} \subset  \cup_{i_\star =1}^{C_3-t+1} \cup_{{\bm e}=\bar{\bm e}_0(\delta(1-6\vep))}^{\bar C_\star n^2 p^\Delta}\left\{ \exists \Graph \subset \G(n,p): \Graph \in \cI_{i_\star, {\bm e}}\right\},
\eeq
where 
\begin{multline*}
\cI_{i_\star, {\bm e}}:= \Big\{ \Graph: \Graph \text{ is a strong-core graph with } e(\Graph) = {\bm e} \text{ and } \\
N({\sf H}, \Graph_{i_\star, {\rm g}}) + N({\sf H}, \bar \Graph_{i_\star, {\rm g}}) \ge \delta(1-8\vep) n^{v_{\sf H}} p^{e({\sf H})} \Big\},
\end{multline*}
and we recall from \eqref{eq:e-0} that $\bar{\bm e}_0(\delta_0):=\bar{\bm e}_0(\delta_0, {\sf H})$ is the minimum number of edges that a graph must possess to have $\delta_0n^{v_{\sf H}} p^{e({\sf H})}$ labelled copies of ${\sf H}$. Thus, to find an upper bound on the probability of the \abbr{LHS} of \eqref{eq:sc-2ell-union-1} it suffices to prove the same for $\cI_{i_\star, {\bm e}}$, and then take a union bound over the allowable range of $i_\star$ and ${\bm e}$. 

We split $\cI_{i_\star, {\bm e}}$ further into two subsets: $N({\sf H}, \Graph_{i_\star, {\rm g}})$ is small and $N({\sf H}, \Graph_{i_\star, {\rm g}})$ is large. Let us first consider the case when $N({\sf H}, \Graph_{i_\star, {\rm g}})$ is small. 

\subsection*{Case $1$} $N({\sf H}, \Graph_{i_\star, {\rm g}}) \le \vep \delta n^{v_{\sf H}} p^{e({\sf H})}$. 

In this case as $N({\sf H}, \Graph_{i_\star, {\rm g}})$ is small the graph $\Graph_{i_\star, {\rm g}}$ can potentially be close to an empty graph. So we cannot use the entropic stability of it. We need to rely on the entropic stability of $\bar \Graph_{i_\star, {\rm g}}$. 

Turning to make this idea precise we fix ${\bm e}_\# \le {\bm e}$ and for ease of writing denote
\[
\cI_{i_\star, {\bm e}, {\bm e}_\#}^{(1)}:= \left\{\Graph: \Graph \in \cI_{i_\star, {\bm e}}, e(\Graph_{i_\star, {\rm g}})={\bm e}_\#, \text{ and } N({\sf H}, \Graph_{i_\star, {\rm g}}) \le \vep \delta n^{v_{\sf H}} p^{e({\sf H})}\right\}.
\]
We aim to derive a bound on the probability of the event $\{\exists \Graph \subset \G(n,p): \Graph \in \cI_{i_\star, {\bm e}, {\bm e}_\#}^{(1)}\}$. To achieve this goal we will apply Lemma \ref{lem:large-count} with 
\beq\label{eq:cV-1}
\cV_1= \left(\cup_{j=1}^{i_\star} V_j\right) \bigcup \left\{v \in V(\Graph_{\rm high}): \deg_\Graph(v) \le D_{i_\star}\right\}
\eeq
and 
\beq\label{eq:gD}
\gD= D_{i_\star}. 
\eeq
Before applying that lemma we need to make several observations. 

From the definition of $\cV_1$ it follows that $\cV_1\setminus (\cup_{j=1}^{i_\star} V_j) \subset V(\Graph_{\rm high})$. Thus 
 \beq\label{eq:cV-1-decompose}
 |\cV_1| \le |V(\Graph_{\rm high})| + \left|\left(\cup_{j=1}^{i_\star} V_j\right)\setminus V(\Graph_{\rm bad})\right| + |V(\Graph_{\rm bad})|,
 \eeq
 where we recall the definition of $\Graph_{\rm bad}$ from Definition \ref{dfn:bad-graph}. We next note that if $v \notin V(\Graph_{\rm bad})$ then there exists at least one copy ${\sf H}$ passing through $v$ that is contained in $\Graph_{\rm low}$. This, in particular implies that $\deg_{\Graph_{i_\star, {\rm g}}}(v) \ge \Delta$ for all $v \in (\cup_{j=1}^{i_\star} V_i)\setminus V(\Graph_{\rm bad})$. From \eqref{eq:deg-wt-u-lbd}, as $np^{\Delta/2} \gg 1$, we also observe that $\Graph_{i_\star, {\rm g}}$ is a bipartite graph with one part $\cup_{j=1}^{i_\star} V_j$. Hence using Lemma \ref{lem:bad-edges-bounds}(b), as
 \[
 |V(\Graph_{\rm high})| \le |V(\Graph_{\rm bad})| \le 2 e(\Graph_{\rm bad}),
 \]
 from \eqref{eq:cV-1-decompose} we derive that
\beq\label{eq:cV-1-decompose-1}
\Delta |\cV_1| \le \Delta \left(|V(\Graph_{\rm high})| + |\left(\cup_{j=1}^{i_\star} V_j\right)\setminus V(\Graph_{\rm bad})| + |V(\Graph_{\rm bad})|\right) \le  4\Delta \vep e(\Graph) + e(\Graph_{i_\star, {\rm g}}).
\eeq
We now apply Lemma \ref{lem:large-count} to find that the cardinality of the set of graphs belonging to $\cI_{i_\star, {\bm e}, {\bm e}_\#}^{(1)}$ with $|\cV_1| ={\bm v}$ is bounded by 
\[
n^{\bm v} p^{-\vep {\bm e}} \le p^{-\Delta {\bm v}} p^{-\vep {\bm e}} \le p^{-{\bm e}_\#} \cdot p^{-5 \Delta \vep  {\bm e}},
\]
for all large $n$, where in the first step we used the fact that $p \le n^{-1/\Delta}$, and in the last step we used \eqref{eq:cV-1-decompose-1} and the fact that for any $\Graph \in \cI_{i_\star, {\bm e}, {\bm e}_\#}^{(1)}$ one has $e(\Graph_{i_\star, {\rm g}})={\bm e}_\#$. 

As the probability of observing any graph with ${\bm e}$ edges is $p^{\bm e}$ taking an union over ${\bm v} \le 2 {\bm e} \le 2 \bar C_\star n^2 p^\Delta$ we conclude that
\begin{align}\label{eq:cI-1-bd-1}
\P(\Graph \subset \G(n,p): \Graph \in \cI_{i_\star, {\bm e}, {\bm e}_\#}^{(1)})  & \le2 \bar C_\star n^2 p^\Delta \cdot  \exp\left(-\log(1/p) \left\{ (1-5\Delta \vep) ({\bm e} -{\bm e}_\#) - 5\Delta \vep  {\bm e}_\#\right\}\right).
\end{align}
From the definition of $\cI_{i_\star, {\bm e}, {\bm e}_\#}$ it further follows that
\[
N({\sf H}, \bar\Graph_{i_\star, {\rm g}}) \ge \delta(1 - 9 \vep)n^{v_{\sf H}} p^{e({\sf H})}
\]
for any $\Graph \in \cI_{i_\star, {\bm e}, {\bm e}_\#}^{(1)}$. 
Thus
\beq\label{eq:e-ehashbd-1}
{\bm e} - {\bm e}_\# = e(\Graph) - e(\Graph_{i_\star, {\rm g}}) \ge e(\bar \Graph_{i_\star, {\rm g}}) \ge \f12 \delta^{\f1t} (1-9\vep)^{\f1t} n^2 p^\Delta,
\eeq
where the final lower bound is again a consequence of \eqref{eq:copy-H-star-bd} and recall $t=\f{v_{\sf H}}{2}$. Moreover, as $\Graph$ is a strong-core graph,
\beq\label{eq:e-ehashbd-2}
{\bm e}_\# \le {\bm e} \le \bar C_\star n^2 p^\Delta.
\eeq
Hence summing both sides of \eqref{eq:cI-1-bd-1} over the allowable range of ${\bm e}_\#$ and ${\bm e}$, given by \eqref{eq:e-ehashbd-1}-\eqref{eq:e-ehashbd-2}, we derive that
\begin{multline}\label{eq:P-cI-1-bd-1} 
\P\left(\bigcup_{{\bm e}, {\bm e}_\#}\left\{\exists \Graph \subset \G(n,p): \Graph \in \cI_{i_\star, {\bm e}, {\bm e}_\#}^{(1)}\right\}\right) \\
\le 2 (\bar C_\star n^2 p^\Delta)^3 \cdot \exp\left( -\log(1/p) \left\{\f12 \delta^{\f1t} (1-5\Delta\vep)(1-9\vep)^{\f1t}n^2 p^\Delta - 5\Delta \vep\cdot \bar C_\star n^2 p^\Delta\right\} \right)\\
\le \exp\left(-\log(1/p)\cdot \f12 \delta^{\f1t} (1-\gf^{(1)}_{\sf H}(\vep)) n^2 p^\Delta\right),
\end{multline}
for all large $n$, where $\gf^{(1)}_{\sf H}(\cdot)$ is some function with the property $\lim_{\vep \downarrow 0} \gf^{(1)}_{\sf H}(\vep)=0$. 

This gives the desired bound when $N({\sf H}, \Graph_{i_\star, {\rm g}})$ is small. Next we consider the other case.

\subsection*{Case $2$} $N({\sf H}, \Graph_{i_\star, {\rm g}}) > \vep \delta n^{v_{\sf H}} p^{e({\sf H})}$.

In this case we need to use the entropic stability of both $\Graph_\gb$ and $\Graph\setminus \Graph_\gb$. As before, let us set 

\[
\cI_{i_\star, {\bm e}, {\bm e}_\#}^{(2)}:= \left\{\Graph: \Graph \in \cI_{i_\star, {\bm e}}, e(\Graph_{i_\star, {\rm g}})={\bm e}_\#, \text{ and } N({\sf H}, \Graph_{i_\star, {\rm g}}) \ge \vep \delta n^{v_{\sf H}} p^{e({\sf H})}\right\}.
\]

To carry out the argument effectively we need to discretize the range of $N({\sf H}, \Graph_{i_\star, {\rm g}})$. To this end, denote 
\[
\cS:=\{\vep, 2\vep, 3\vep, \ldots, s_0\vep\},
\] 
where $s_0:=\lfloor (1-9\vep)/\vep\rfloor$. For $\upeta \in \cS$ let
 \[
\cI_{i_\star, {\bm e}, {\bm e}_\#, \upeta}^{(2)}:= \left\{\Graph: \Graph \in \cI_{i_\star, {\bm e}}, e(\Graph_{i_\star, {\rm g}})={\bm e}_\#, \text{ and } N({\sf H}, \Graph_{i_\star, {\rm g}}) \in[ \upeta \delta n^{v_{\sf H}} p^{e({\sf H})}, (\upeta+\vep) \delta n^{v_{\sf H}} p^{e({\sf H})}]\right\}.
\]

Finally let
\[
\wh\cI_{i_\star, {\bm e}, {\bm e}_\#}^{(2)}:=\left\{\Graph: \Graph \in \cI_{i_\star, {\bm e}}, e(\Graph_{i_\star, {\rm g}})={\bm e}_\#, \text{ and } N({\sf H}, \Graph_{i_\star, {\rm g}}) \geq (1-9\vep) \delta n^{v_{\sf H}} p^{e({\sf H})}\right\}.
\]

Let us proceed to bound the cardinality of $\cI_{i_\star, {\bm e}, {\bm e}_\#, \upeta}^{(2)}$. This will be done by applying Lemma \ref{lem:large-count}. To apply that lemma we need several estimates. 

Since $i_\star \le C_3$ recalling the definition of the sets $\{V_j\}$ and that $\Graph_{i_\star, {\rm g}}$ is a bipartite graph, with $\cup_{j=1}^{i_\star} V_j$ as one of its parts, we find that 
\beq\label{eq:v-sharp-bd}
D\cdot \left(\f{C_0(\vep)}{c_0(\vep)}\right)^{C_3} \cdot |\cup_{j=1}^{i_\star} V_j| \ge e(\Graph_{i_\star, {\rm g}}).
\eeq
If $\Graph \in \cI_{i_\star, {\bm e}, {\bm e}_\#, \upeta}^{(2)}$ then $N({\sf H}, \Graph_{i_\star, {\rm g}}) \ge \upeta \delta n^{v_{\sf H}} p^{e({\sf H})}$ which in turn, by yet another application of \eqref{eq:copy-H-star-bd}, implies that
\beq\label{eq:e-i-star-bd}
e(\Graph_{i_\star, {\rm g}}) \ge \f{1}{2}\upeta^{\f1t} \delta^{\f1t} n^2 p^\Delta \ge \f{1}{2}\vep^{\f1t} \delta^{\f1t} n^2 p^\Delta.
\eeq
Thus, as $n p^{\Delta/2} \gg 1$, by \eqref{eq:v-sharp-bd} we have
\beq\label{eq:vhash-bd}
{\bm v}_\# := |\cup_{j=1}^{i_\star} V_j| \ge e np^{\Delta/2},
\eeq
for all large $n$. Furthermore, we recall that for $\cV_1$ as in \eqref{eq:cV-1} the set of vertices $\cV_1 \setminus (\cup_{j=1}^{i_\star} V_j) \subset V(\Graph_{\rm high})$. Thus using Lemma \ref{lem:bad-edges-bounds}(b) we also have that
\beq\label{eq:cV-1-bad-bd}
0 \le |\cV_1| - |\cup_{j=1}^{i_\star} V_j| \le 2 \vep e(\Graph).
\eeq
Using the lower bound on ${\bm v}_\#$ we now apply Lemma \ref{lem:large-count} with 
$\cV_1$ and $\gD$ as in \eqref{eq:cV-1}-\eqref{eq:gD} to find that the number of graphs in $\cI_{i_\star, {\bm e}, {\bm e}_\#, \upeta}^{(2)}$ with $|\cV_1| ={\bm v}$ and $|\cup_{j=1}^{i_\star} V_j| = {\bm v}_\#$ is bounded by 
\[
\binom{n}{{\bm v}} p^{-\vep {\bm e}} \le \left(\f{en}{{\bm v}}\right)^{{\bm v}} p^{-\vep {\bm e}}  \le \left(\f1p\right)^{(\Delta/2) \cdot {\bm v}+\vep {\bm e}} \le \left(\f1p\right)^{{\bm v}_\# + 2\Delta \vep {\bm e}}, 
\]
where the first step is due to Stirling's approximation, \corAB{and the second step is due to the fact that ${\bm v} \ge {\bm v}_\#$ and \eqref{eq:vhash-bd}. While the last inequality is due to \eqref{eq:cV-1-bad-bd}.} Thus
\begin{multline}\label{eq:P-cI-2-bd-1}
\P\left(\exists \Graph \subset \G(n,p): \Graph \in \cI_{i_\star, {\bm e}, {\bm e}_\#, \upeta}^{(2)} \text{ such that } |\cV_1| ={\bm v} \text{ and } |\cup_{j=1}^{i_\star} V_j| = {\bm v}_\# \right) \\
\le \exp\left(-\log(1/p) \cdot \left\{({\bm e}_\# -(\Delta/2) \cdot {\bm v}_\#) + (1-2\Delta \vep) ({\bm e}-{\bm e}_\#) - 2\Delta \vep {\bm e}_\#\right\}\right).
\end{multline}
To simplify the \abbr{RHS} we need lower bounds on $({\bm e}_\# -(\Delta/2) \cdot {\bm v}_\#)$ and $({\bm e}-{\bm e}_\#)$. Here we will use the lower bound on $N({\sf H}, \Graph_{i_\star, {\rm g}})$ and Lemma \ref{lem:small-count}.

 To this end, we remind the reader that we already noted above that the graph $\Graph_{i_\star, {\rm g}}$ is a bipartite graph with one part $\cup_{j=1}^{i_\star} V_j$. We also recall from above that for any $v \in \cup_{j=1}^{i_\star} V_j \setminus V(\Graph_{\rm bad})$ its degree $\deg_{\Graph_{i_\star, {\rm g}}}(v) \ge \Delta$. By Lemma \ref{lem:bad-edges-bounds}(b) again we have that
\[
|V(\Graph_{\rm bad})| \le \vep e(\Graph) \le \vep \cdot \bar C_\star n^2 p^\Delta \le K_0 \vep' \cdot \f12 \vep^{\f1t} \delta^{\f1t} n^2 p^\Delta \le K_0 \vep' \cdot e(\Graph_{i_\star, {\rm g}}),
\]
for some large constant $K_0$ (depending only on $\delta$) and $\vep'=\vep^{1-\f1t}$, where the first inequality is due to the fact that $\Graph$ is a strong-core graph, and the last step follows from \eqref{eq:e-i-star-bd}. Therefore $\Graph_{i_\star, {\rm g}}$ satisfies Assumption \ref{ass:bipartite-ass} with $\vep_2=K_0 \vep'$. Hence applying Lemma \ref{lem:small-count}  with $\delta_1=\upeta \delta$ and $\vep_2$ as above we find that
\beq\label{eq:ehash-vhash}
{\bm e}_\# - (\Delta/2) \cdot {\bm v}_\# \ge \f12 \upeta^{\f1t} \delta^{\f1t} \cdot (1-\Delta \vep_2) n^2 p^\Delta. 
\eeq
On the other hand the assumption $N({\sf H}, \Graph_{i_\star, {\rm g}}) \le (\upeta+\vep) \delta n^{v_{\sf H}}p^{e({\sf H})}$ together with \eqref{eq:cross-term-bd-2} implies that 
\[
N({\sf H}, \bar \Graph_{i_\star, {\rm g}}) \ge (1-9\vep-\upeta)\delta n^{v_{\sf H}} p^{e({\sf H})}. 
\]
This in turn yields the lower bound
\beq\label{eq:e-ehash-bd-2}
{\bm e} - {\bm e}_\# \ge \f12 \delta^{\f1t} \cdot (1-9\vep -\upeta)^{\f1t}n^2 p^\Delta. 
\eeq
Since for any $x \in [0,1]$ and $t \ge 1$ one has $x^{\f1t} + (1-x)^{\f1t} \ge 1$, we deduce from \eqref{eq:ehash-vhash} and \eqref{eq:e-ehash-bd-2} that any $\upeta \in \cS$  we have
\[
({\bm e}_\# - (\Delta /2) \cdot {\bm v}_\#) + (1-2 \Delta \vep) ({\bm e}-{\bm e}_\#) \ge \f12 \delta^{\f1t} \cdot (1-\Delta \vep_2) \cdot (1-2\Delta \vep) \cdot (1-9\vep)^{\f1t} \cdot n^2 p^\Delta.
\]
Plugging this bound in the \abbr{RHS} of \eqref{eq:P-cI-2-bd-1} and then taking a union over $\upeta \in \cS$, and ${\bm v}$, ${\bm v}_\#$, ${\bm e}$ and ${\bm e}_\#$ over their respective allowable ranges we derive that 
\begin{multline}\label{eq:P-cI-2-bd-2} 
\P\left(\bigcup_{{\bm e}=\bar{\bm e}_0(\delta(1-6\vep))}^{\bar C_\star n^2 p^\Delta} \bigcup_{{\bm e}_\#, \upeta}\left\{\exists \Graph \subset \G(n,p): \Graph \in \cI_{i_\star, {\bm e}, {\bm e}_\#, \upeta}^{(2)}\right\}\right) \\
\le 4 (\bar C_\star n^2 p^\Delta)^4 \cdot |\cS| \cdot \exp\left( -\log(1/p) \left\{\f12 \delta^{\f1t} (1-\Delta \vep_2)(1-2 \Delta \vep)  (1-9\vep)^{\f1t}  n^2 p^\Delta - 2 \Delta \vep\cdot \bar C_\star n^2 p^\Delta\right\} \right)\\
\le \exp\left(-\log(1/p) \cdot \f12 \delta^{\f1t}(1-\gf^{(2)}_{\sf H}(\vep)) n^2 p^\Delta\right),
\end{multline}
for all large $n$, where $\gf_{\sf H}^{(2)}(\cdot)$ is some other function satisfying $\lim_{\vep \downarrow 0} \gf_{\sf H}^{(2)}(\vep)=0$. 

Next we need to obtain a bound for the case $\Graph\in \wh\cI_{i_\star, {\bm e}, {\bm e}_\#}^{(2)}$. Observe that only the lower bound on $N({\sf H}, \Graph_{i_\star, {\rm g}})$ was used in deriving \eqref{eq:P-cI-2-bd-1} and hence the same argument gives (ignoring the term $(1-2 \Delta \vep) ({\bm e}-{\bm e}_\#)$
\begin{multline}\label{eq:P-cI-2-bd-3}
\P\left(\exists \Graph \subset \G(n,p): \Graph \in  \wh\cI_{i_\star, {\bm e}, {\bm e}_\#}^{(2)} \text{ such that } |\cV_1| ={\bm v} \text{ and } |\cup_{j=1}^{i_\star} V_j| = {\bm v}_\# \right) \\
\le \exp\left(-\log(1/p) \cdot \left\{({\bm e}_\# - (\Delta/2) \cdot {\bm v}_\#) - 2\Delta \vep {\bm e}_\#\right\}\right).
\end{multline}
Also, as in the derivation of \eqref{eq:ehash-vhash}, using Lemma \ref{lem:small-count} with $\delta_1=(1-9\vep)\delta$ we get
\beq\label{eq:ehash-vhash-2}
{\bm e}_\# - {\bm v}_\# \ge \f12 (1-9\vep)^{\f1t} \delta^{\f1t} \cdot (1-\Delta \vep_2) n^2 p^\Delta. 
\eeq
Taking a union bound over ${\bm v}$, ${\bm v}_\#$, ${\bm e}$ and ${\bm e}_\#$ over their respective allowable ranges we derive as in \eqref{eq:P-cI-2-bd-2} that
\begin{multline}\label{eq:P-cI-2-bd-4} 
\P\left(\bigcup_{{\bm e}=\bar{\bm e}_0(\delta(1-6\vep))}^{\bar C_\star n^2 p^\Delta} \bigcup_{{\bm e}_\#\leq {\bm e}}\left\{\exists \Graph \subset \G(n,p): \Graph \in \wh\cI_{i_\star, {\bm e}, {\bm e}_\#}^{(2)}\right\}\right) \\
\le 4 (\bar C_\star n^2 p^\Delta)^4  \cdot \exp\left( -\log(1/p) \left\{\f12 \delta^{\f1t} (1-\Delta \vep_2) (1-9\vep)^{\f1t}  n^2 p^\Delta - 2 \Delta \vep\cdot \bar C_\star n^2 p^\Delta\right\} \right)\\
\le \exp\left(-\log(1/p) \cdot \f12 \delta^{\f1t}(1-\gf^{(3)}_{\sf H}(\vep)) n^2 p^\Delta\right),
\end{multline}
for all large $n$, where $\gf_{\sf H}^{(3)}(\cdot)$ is another function satisfying $\lim_{\vep \downarrow 0} \gf_{\sf H}^{(3)}(\vep)=0$.

As
\[
\left\{\exists \Graph \subset \G(n,p):\wh\cI_{i_\star, {\bm e}, {\bm e}_\#}^{(2)}\right\} \subset \left\{\exists \Graph \subset \G(n,p): \Graph \in \wh\cI_{i_{\star},{\bm e}, {\bm e}_\#}^{(2)}\right\} \bigcup_{\upeta \in \cS} \left\{\exists \Graph \subset \G(n,p): \Graph \in \cI_{i_{\star}, {\bm e}, {\bm e}_\#, \upeta}^{(2)}\right\},
\]
equipped with \eqref{eq:P-cI-1-bd-1}, \eqref{eq:P-cI-2-bd-2}, and\eqref{eq:P-cI-2-bd-4} we then take another union over $i_\star \in \llbracket C_3\rrbracket$,  set 
\[
\gf_{\sf H}(\vep):= \max\{\gf_{\sf H}^{(1)}(\vep), \gf_{\sf H}^{(2)}(\vep), \gf_{\sf H}^{(3)}(\vep)\} +\vep,
\] 
and use \eqref{eq:sc-2ell-union-1} to derive the desired bound on the probability of the existence of a strong-core graph. This finally finishes the proof of the proposition.  
\end{proof}

\begin{rmk}
Let us remark that the reverse pigeonhole principle was also used in \cite{BGLZ} to find the solution of the variational problem in \eqref{eq:var-prblm-0}. In particular, there it was used to deduce that if the subgraph ${\sf H}$ is bipartite (in particular a cycle of even length) and $U$ is a {\em graphon}, then the contribution to the normalized homomorphism density $t({\sf H}, U)$ when some neighboring vertices in ${\sf H}$ are mapped to {\em vertices} of low and high degrees in $U$, is negligible compared to the leading term. See the proofs of \cite[Lemmas 8.1 and 8.2]{BGLZ}. 
\end{rmk}

\section{Entropic stability of core graphs with many edges}\label{sec:core-many-edges}
It remains to show that the set of core graphs with many edges are entropically stable. This is the content of Proposition \ref{p:upper-notstrong}. We recall from Section \ref{sec:proof-main-thm} that the proof of that result splits into two cases: {$np^{\Delta/2} \ge (\log n)^{v_{\sf H}}$} and {$np^{\Delta/2} \le (\log n)^{v_{\sf H}}$}. First we consider the easier case of large $p$. 



\subsection{Core graphs with many edges in the large $p$ regime}\label{s:bipartite}
We remind the reader that in this case the proof of Proposition \ref{p:upper-notstrong} was provided in Section \ref{sec:notstronglarge-p} assuming Lemma \ref{lem:core-large-edge}. So it suffices to prove the latter result. 
\begin{proof}[Proof of Lemma \ref{lem:core-large-edge}]
\corAB{We begin by reminding ourselves of the definitions of $\cW(\Graph)$, $E_{1,2}(\Graph)$, $E_{2,2}(\Graph)$, $e_{1,2}(\Graph)$, and $e_{2,2}(\Graph)$ (see \eqref{eq:cW}, and \eqref{eq:e-121}-\eqref{eq:e-123}). Recall} that $\Graph_\cW$ is the subgraph of $\Graph$ induced by edges that are incident to some vertex in $\cW$. Since by assumption $\Graph_\cW$ is bipartite we have that $e(\Graph)= e_{1,2}(\Graph) + e_{2,2}(\Graph)$.

Equipped with this observation we now proceed as follows: we fix $\gh:=\{{\bm w}, {\bm e}_{1,2}, {\bm e}_{2,2}\}$ with ${\bm e}_{1,2}+{\bm e}_{2,2}={\bm e}\geq {\bm e}_{\star}$ (recall from Section \ref{sec:notstronglarge-p} that ${\bm e}_\star = \bar C_\star n^2 p^\Delta$) and let 
\[
\sA_{{\bm e}, \gh}:=\left\{\exists \Graph \subset \G(n,p): \Graph \text{ is a core graph with } (|\cW(\Graph)|, e_{1,2}(\Graph), e_{2,2}(\Graph)) = \gh \right\}.
\]
We bound the probability of $\sA_{{\bm e}, \gh}$ for each fixed choice of $\gh$ and then take a union bound over the allowable range of $\gh$.

Observe that for $\sA_{{\bm e}, \gh}$ to be non-empty the following constraint needs to be satisfied:
\begin{equation}
\label{e:V1bound}
\Delta {\bm w}\leq {\bm e}_{1,2}\leq D{\bm w}.
\end{equation}
Since $\Graph_\cW$ is bipartite the upper bound is immediate as the maximal degree among the vertices in $\cW(\Graph)$ is at most $D$. On the other hand $\Graph$ being a core graph each edge must participate in at least one copy of ${\sf H}$. This yields that the minimum degree of the vertices in $\Graph$ is at least $\Delta$ which in turn implies the lower bound in \eqref{e:V1bound}. 

We now split the proof into two cases: (1) ${\bm e}_{2,2} \geq {\bm e}_{1,2}$, and (2) ${\bm e}_{2,2} \leq {\bm e}_{1,2}$.

\subsection*{Case 1} ${\bm e}_{2,2} \geq {\bm e}_{1,2}$. 

We invoke {Lemma \ref{lem:large-count}} to upper bound the total number of core graphs with $(|\cW|, e_{1,2}(\Graph), e_{2,2}(\Graph))=\gh$ by 
\[
\binom{n}{{\bm w}}\left( \frac{1}{p}\right)^{\vep {\bm e}} \le n^{\bm w} \cdot \left( \frac{1}{p}\right)^{\vep {\bm e}} \le \left( \frac{1}{p}\right)^{\vep {\bm e} + {\bm e}_{1,2}}
\]
where we have used that $p \le n^{-1/\Delta}$ and the lower bound from \eqref{e:V1bound}. Thus 
\begin{equation}\label{eq:e22large}
\P(\sA_{{\bm e},\gh}) \leq  \left( \frac{1}{p}\right)^{\vep {\bm e} + {\bm e}_{1,2}} \cdot p^{{\bm e}_{1,2}+ {\bm e}_{2,2}  } \le p^{{\bm e}_{2,2} - \vep {\bm e}} \le p^{\left(\f12 -\vep\right) {\bm e}},
\end{equation}
where the last inequality follows upon noting that ${\bm e}_{2,2} \ge \f12 {\bm e}$.

\subsection*{Case 2} ${\bm e}_{2,2}\leq {\bm e}_{1,2}$.

In this case using the upper bound in \eqref{e:V1bound} and the fact that ${\bm e}_{1,2} \ge \f12 {\bm e} \ge\f12 {\bm e}_\star$  we find that
\[
\frac{n}{{\bm w}} \leq \frac{n D}{{\bm e}_{1,2}} \leq \frac{2 n D}{\bar C_\star n^2p^\Delta} \leq \frac{1}{e p^{\Delta/2}},
\]
for all large $n$, where the last inequality is a consequence of the fact that $np^{\Delta/2} \gg 1$. Therefore applying Lemma \ref{lem:large-count} again and using Stirling's approximation we deduce that the number of core graphs with $(|\cW|, e_{1,2}(\Graph), e_{2,2}(\Graph))=\gh$ is bounded by
\[
\left(\f{e n}{{\bm w}}\right)^{\bm w} \cdot \left( \frac{1}{p}\right)^{\vep {\bm e}} \le \left( \frac{1}{p}\right)^{\vep {\bm e} + (\Delta/2) \cdot {\bm w}}.
\]
Now using the lower bound \eqref{e:V1bound} and the fact that ${\bm e}_{1,2} \ge \f12 {\bm e}$ we derive from above that
\begin{equation}\label{eq:e22small}
\P(\sA_{{\bm e}, \gh}) \leq  \left( \frac{1}{p}\right)^{\vep {\bm e} +(\Delta/2) \cdot  {\bm w}} \cdot p^{{\bm e}_{1,2}} 
\leq  p^{\f12 {\bm e}_{1,2} - \vep {\bm e}} \le p^{\left(\f14 - \vep\right) {\bm e}}.
\end{equation}
To complete the proof we first sum over all possible choices of $\gh$ for a given ${\bm e}$. Observe that the total number of choices for each of ${\bm w}$ and ${\bm e}_{1,2}$ are trivially upper bounded by ${\bm e}$. Thus, using \eqref{eq:e22large}-\eqref{eq:e22small} and applying  an union bound we find 
\[
\P(\sA_{\bm e}) = \P\left(\cup_{\gh} \sA_{{\bm e}, \gh}\right) \leq {\bm e}^2 p^{\left(\frac{1}{4} -\vep\right){\bm e}}. 
\]
Now summing the over all {${\bm e}_{\star}\leq {\bm e}\leq \bar C n^2p^{\Delta}\log(1/p)$} we derive that 
\[ 
\P\left( \cup_{{\bm e} \ge {\bm e}_\star} \sA_{\bm e}\right) \le  (\bar C n^2p^{\Delta}\log(1/p))^2 \cdot p^{\left(\frac{1}{4}-\vep\right){\bm e}_{\star}}\leq \exp\left(-\frac{1}{16}\bar{C}_\star n^2p^{\Delta}\log(1/p)\right),
\]
for all large $n$, where in the last step we have used the facts $n^2p^{\Delta}\log(1/p)\gg 1$ and $\vep \le \f18$. This completes the proof of the lemma. 
\end{proof}


\corAB{We now turn to the case of $np^{\Delta/2} \le (\log n)^{v_{\sf H}}$. We remind the reader that in this case} for a core graph $\Graph$ the subgraph $\Graph_\cW$ need not be a bipartite graph. Nevertheless, as we show below in the next section if we assume that $e_{1,2}(\Graph)+ e_{2,2}(\Graph)$ is sufficiently large then the set of those core graphs are entropically stable.

\subsection{Entropic stability for graphs with large $e_{1,2}(\Graph)+e_{2,2}(\Graph)$}\label{sec:combin-graph-bd}

In this section we prove Lemmas \ref{l:core11} and \ref{l:core2jb}. Both proofs will use an argument analogus to the proof of Lemma \ref{lem:large-count}. However, we remind the reader these lemmas find bounds on the probabilities of certain events that may involve graphs which are no longer core graphs. Therefore we cannot directly apply Lemma \ref{lem:large-count}. To this end, we have the following general lemma. Its proof is similar in nature to that of Lemma \ref{lem:large-count}. Before stating the lemma we introduce a few more notation. 

For every graph $\Graph$ we denote $\cW_\cP(\Graph)$ to be the subset of vertices satisfying some property $\cP_\Graph$. That is, $\cP_\Graph: V(\Graph) \mapsto \{0,1\}$ is a map and $\cW_\cP(\Graph):= \{v \in V(\Graph): \cP_\Graph(v) = 1\}$. Set $\ol{\cW}_\cP(\Graph):= V(\Graph)\setminus \cW_\cP(\Graph)$. Thus $\{\cW_\cP(\Graph), \ol{\cW}_\cP(\Graph)\}$ is a partition of the vertices of $\Graph$ which may be determined by some properties of the graph $\Graph$. Denote
\[
\wt e_{1,1}(\Graph) := \left| \left\{(u,v) \in E(\Graph): u,v \in \cW_\cP(\Graph)\right\}\right|,
\]
\[
\wt e_{1,2}(\Graph) := \left| \left\{(u,v) \in E(\Graph): u \in \cW_\cP(\Graph), v\in \ol\cW_\cP(\Graph)\right\}\right|,
\]
and
\[
\wt e_{2,2}(\Graph) := \left| \left\{(u,v) \in E(\Graph): u,v \in \ol\cW_\cP(\Graph)\right\}\right|.
\]
We now state the relevant lemma. 

\begin{lem}\label{lem:graph-spec-prop-bd}
Fix non-negative integers $\wt{\bm e}_{1,1}, \wt{\bm e}_{1,2},\wt{\bm e}_{2,2}$, $\wt{\bm w}$, and $\wt \vep >0$. Set $\ul{\bm e}:= (\wt{\bm e}_{1,1}, \wt{\bm e}_{1,2}, \wt{\bm e}_{2,2})$. 
Let $\cN_\#(\ul{\bm e}, \wt{\bm w})$ be the number of graphs $\Graph$ with $|\cW_\cP(\Graph)|=\wt{\bm w}$,  
\beq\label{eq:wte-cond}
\wt{e}_{1,1}(\Graph) = \wt{\bm e}_{1,1}, \quad \wt{e}_{1,2}(\Graph) = \wt{\bm e}_{1,2}, \quad \text{ and } \quad \wt{e}_{2,2}(\Graph) = \wt{\bm e}_{2,2},
\eeq
such that  $\wt {\bm e}:=\wt {\bm e}_{1,1} + \wt {\bm e}_{1,2} + \wt {\bm e}_{2,2}\leq \bar{C}n^2 p^{\Delta} \log (1/p)$, and
\beq\label{eq:barcWcP}
|\ol \cW_\cP(\Graph)| \le \wt\vep (\wt e_{1,2}(\Graph) +\wt e_{2,2}(\Graph)).
\eeq
If $np^{\Delta/2} \le (\log n)^{v_{\sf H}}$ then
\[
\cN_\#(\ul{\bm e}, \wt{\bm w}) \le  \exp \left(\log(1/p) \cdot (\Delta/2) \cdot \left\{\wt{\bm w}+ \wt \vep(\wt{\bm e}_{1,2} +\wt{\bm e}_{2,2})\right\} + K (\log \log n) \cdot \wt{\bm e} \right)
\]
for all large $n$, where {$K:=K(v_{\sf H})$} is some constant depending only on the number of vertices of ${\sf H}$.
\end{lem}

\begin{proof}
The proof uses simple combinatorial bounds. For ease of writing let us denote $\wt{\bm w}_1:= |\bar \cW_\cP(\Graph)|$. We write $\cN_\#(\ul{\bm e}, \wt{\bm w}, \wt{\bm w}_1)$ to denote the number of graphs with $|\cW_\cP(\Graph)|=\wt{\bm w}$, $|\cW_\cP(\Graph)| = \wt{\bm w}_1$, $\wt{\bm e} \le \bar C n^2 p^{\Delta} \log(1/p)$, and satisfies \eqref{eq:wte-cond}-\eqref{eq:barcWcP}. First we a find a bound on $\cN_\#(\ul{\bm e}, \wt{\bm w}, \wt{\bm w}_1)$ and then take a union over the allowable range of $\wt{\bm w}_1$ to derive a bound on $\cN_\#(\ul{\bm e}, \wt{\bm w})$. 

To this end, consider any $\wt{\bm w}_1 \leq \wt\vep (\wt {\bm e}_{1,2} +\wt {\bm e}_{2,2})$ (note this bound is imposed by \eqref{eq:barcWcP}).  The vertices in $\cW_\cP(\Graph)$ and $\ol \cW_\cP(\Graph)$, and hence all the vertices in $\Graph$ can be chosen in at most $n^{\wt{\bm w}+\wt{\bm w}_1}$ ways. Once these vertices are chosen the number of ways to specify edges is at most $(\wt{\bm w}+\wt{\bm w}_1)^{2\wt {\bm e}}$. 

As $np^{\Delta/2} \le (\log n)^{v_{\sf H}}$ and $\wt{\bm e} \le \bar C n^2 p^{\Delta} \log(1/p)$ we also get that $\wt{\bm w}+\wt{\bm w}_1\leq 2\wt {\bm e}\leq  (\log n)^{2v_{\sf H}+2}$ for all large $n$. This implies that $(\wt{\bm w}+\wt{\bm w}_1)^{2\wt {\bm e}}\leq \exp(K'\log \log n \cdot \wt {\bm e})$ for some $K'>0$ and all $n$ sufficiently large. Combining these estimates we obtain that 
\[
\cN_\#(\ul{\bm e}, \wt{\bm w}, \wt{\bm w}_1)\leq n^{\wt{\bm w}_+\wt{\bm w}_1} \exp(K'\log \log n \cdot \wt {\bm e})\le \exp (\log(1/p)(\wt{\bm w}+\wt{\bm w}_1) + 2 K'\log \log n \cdot \wt {\bm e}),
\]
where in the last step we again used the facts $np^{\Delta/2} \le (\log n)^{v_{\sf H}}$ and $\wt{\bm w}+\wt{\bm w}_1\leq 2{\bm e}$. 
Finally taking a union over all $\wt{\bm w}_1 \leq \wt\vep (\wt e_{1,2} +\wt e_{2,2})$ we arrive at the desired result. 
\end{proof}

Using Lemma \ref{lem:graph-spec-prop-bd} we now prove Lemma \ref{l:core11}.

\begin{proof}[Proof of Lemma \ref{l:core11}]
Fix a vector $\wh{\ul{\bm e}}:= ({\bm e}_{1,1}, {\bm e}_{1,2}, {\bm e}_{2,2})$ and ${\bm w}$. Set $e_{1,1}(\Graph):=|E_{1,1}(\Graph)|$. Let $\mathscr{N}(\wh{\ul{\bm e}},{\bm w})$ denote the number of possible graphs satisfying the hypothesis of the event $\wt{\mathrm{Core}}_{1,1}$ (recall its definition from \eqref{eq:wtcore11}), i.e., the set of graphs $\Graph$ with 
\begin{enumerate}
\item $e_{1,1}(\Graph)={\bm e_{1,1}}, \, e_{1,2}(\Graph)={\bm e_{1,2}}, \, e_{2,2}(\Graph)={\bm e_{2,2}}$, $|\cW(\Graph)| = {\bm w}$,
\item ${\bm e_{1,1}}+{\bm e_{1,2}}+{\bm e_{2,2}}={\bm e} \leq (\log \log n)^{-2}\bar{C}n^2p^{\Delta} \log (1/p)$,
\end{enumerate}
and
\begin{enumerate}
\item[(3)] ${\bm e_{1,2}}+{\bm e_{2,2}} \geq \bar{C}_{\star}n^2p^{\Delta}$ and {$d_{\min}(\Graph) \ge \Delta$}. 
\end{enumerate} 
We refer the reader to \eqref{eq:cW}, and \eqref{eq:e-121}-\eqref{eq:e-123} to recall the definitions of $\cW(\Graph)$, $e_{1,2}(\Graph)$, and $e_{2,2}(\Graph)$.
Now set $\cP_\Graph: V(\Graph) \mapsto \{0,1\}$ to be 
\beq\label{eq:cP-dfn-1}
\cP_\Graph(v) := \left\{\begin{array}{ll} 1 & \mbox{ if } \deg_\Graph(v) \le D,\\
0 & \mbox{ otherwise},
\end{array}
\right.
\eeq
where $D$ is as in \eqref{eq:D}. With this choice of $\cP_\Graph$ we have $\cW(\Graph)= \cW_\cP(\Graph)$. Thus any edge incident to some vertex in $\ol \cW_\cP(\Graph)$ must either be in $E_{1,2}(\Graph)$ or be in $E_{2,2}(\Graph)$. Therefore, using the fact any vertex in $V(\Graph)\setminus \cW(\Graph)$ has degree at least $D$, we have   
\[
|\ol \cW_\cP(\Graph)| \le (2/D) \cdot (e_{1,2}(\Graph)+ e_{2,2}(\Graph)) \le (\vep/8 \Delta)\cdot (e_{1,2}(\Graph)+ e_{2,2}(\Graph)),
\]
indicating that \eqref{eq:barcWcP} holds with $\wt \vep = \vep/(8\Delta)$. Hence, we can now apply Lemma \ref{lem:graph-spec-prop-bd} with $\cP_\Graph$ as in \eqref{eq:cP-dfn-1} and $\wt \vep$ as above. To obtain a usable bound we further note that any edge adjacent to some vertex in $\cW(\Graph)$ must be in $E_{1,1}(\Graph) \cup E_{1,2}(\Graph)$. Moreover, it follows from the definition that both end points of an edge in $E_{1,2}(\Graph)$ cannot be in $\cW(\Graph)$. As $d_{\min}(\Graph) \ge \Delta$ we deduce that
\beq\label{eq:cWE112}
\Delta|\cW(\Graph)| \le \sum_{v \in \cW(\Graph)} \deg_\Graph(v) \le 2 |E_{1,1}(\Graph)| + |E_{1,2}(\Graph)|.
\eeq
By Lemma \ref{lem:graph-spec-prop-bd} we have that 
\begin{align}\label{eq:mathscrN}
\mathscr{N}(\wh{\ul {\bm e}}, {\bm w}) &  \le \exp\left(\log(1/p) \cdot (\Delta/2) \cdot \left\{{\bm w}+ (\vep/(8\Delta)) \cdot ({\bm e}_{1,2}+ {\bm e}_{2,2})\right\} + K \log \log n \cdot {\bm e}\right) \notag\\
& \le \exp\left(\log(1/p) \left\{{\bm e}_{1,1} + \f12 {\bm e}_{1,2} + (\vep/16) \cdot ({\bm e}_{1,2}+ {\bm e}_{2,2})\right\} + K \log \log n \cdot {\bm e}\right),
\end{align}
where we the last step is a consequence of \eqref{eq:cWE112}. Furthermore, as ${\bm e}_{1,2}+ {\bm e}_{2,2} \ge \bar C_\star n^2 p^{\Delta}$ and  ${\bm e} \le \bar C  n^2 p^{\Delta} \log(1/p)/(\log \log n)^2$ we have
\[
K \log \log n \cdot {\bm e} \le (\vep /16) \cdot ({\bm e}_{1,2}+ {\bm e}_{2,2}) \cdot \log(1/p), 
\]
for all large $n$. Thus, using \eqref{eq:mathscrN} and the above inequality we derive
\begin{multline}\label{eq:core11-pre-bd}
\P(\wt{\rm Core}_{1,1}) \le \sum_{\wh{\ul{\bm e}}, {\bm w}} \mathscr{N}(\wh{\ul {\bm e}}, {\bm w}) \cdot p^{{\bm e}_{1,1}+{\bm e}_{1,2}+ {\bm e}_{2,2}} \\
\le \sum_{\wh{\ul{\bm e}}, {\bm w}} \exp\left(-\log(1/p) \left\{  \f12 {\bm e}_{1,2} + {\bm e}_{2,2} - (\vep/8) \cdot ({\bm e}_{1,2}+ {\bm e}_{2,2})\right\} \right),
\end{multline}
for all large $n$, where the sums over $\wh{\ul{\bm e}}$ and ${\bm w}$ are to be taken over their allowable respective ranges. 


Since ${\bm w} \le 2 {\bm e} \le 2 \bar C n^2 p^{\Delta} \log(1/p)/ (\log \log n)^2$ and ${\bm e}_{1,2} + {\bm e}_{2,2} \ge \bar C_\star n^2 p^{\Delta}$ it is now immediate from \eqref{eq:core11-pre-bd} that
\beq\label{eq:wtcore11p-bd}
\P(\wt{\rm Core}_{1,1}) \le \exp\left(-\f{\bar C_\star}{8}n^2 p^{\Delta} \log(1/p)\right),
\eeq
for all large $n$. Recalling that $\bar C_\star \ge 32 \delta^{\f2{v_{\sf H}}}$ we obtain the desired probability upper bound. This completes the proof. 
%
%
\end{proof}

Next, using Lemma \ref{lem:graph-spec-prop-bd} again we prove Lemma \ref{l:core2jb}. 

\begin{proof}[Proof of Lemma \ref{l:core2jb}]
Fix any $j \in \llbracket L_n \rrbracket$ and let us recall the definition of ${\rm Core}_{2,j,\b}$ from \eqref{eq:core2jb}. We see that any graph $\Graph_0$ satisfying the hypothesis of ${\rm Core}_{2,j,\b}$ must be contained in the set 
\beq\label{eq;core2jbset}
\Omega:=\{\Graph_0: \exists \Graph':=\Graph'(\Graph_0) \supset \Graph_0 \text{ such that } \Graph_0, \Graph' \in \cJ_j \text{ and } \varpi_\Delta(\Graph') = {\Graph_0} \}, \notag
\eeq
where we recall that $\varpi_\Delta(\Graph')$ is the $\Delta$-core of the subgraph of $\Graph'$ obtained upon removing the edges in $E_{1,1}(\Graph')= \{(u,v) \in E(\Graph'): \deg_{\Graph'}(u), \deg_{\Graph'}(v) \le D\}$ and $D$ is as in \eqref{eq:D}. For a $\Graph_0 \in\Omega$ there may be more than one $\Graph' \in \cJ_j$ such that $\Graph' \in \cJ_j$ and $\varpi_\Delta(\Graph')=\Graph_0$. Choose any one of them arbitrarily and fix it for the rest of the proof of this lemma. 

Our goal would be to bound the cardinality of $\Omega$. To this end, set $\cP_{\Graph_0}: V(\Graph_0) \mapsto \{0,1\}$ to be
\beq\label{eq:cP-dfn-2}
\cP_{\Graph_0}(v) := \left\{\begin{array}{ll} 
0 & \mbox{ if }  \deg_{\Graph'(\Graph_0)}(v) \ge D,\\
1 & \mbox {otherwise}.
\end{array}
\right. \notag
\eeq
With  this choice of $\cP_{\Graph_0}$ we now apply Lemma \ref{lem:graph-spec-prop-bd}. Since $\Graph_0, \Graph' \in \cJ_j$ we find that $e(\Graph_0) \ge \f12 e(\Graph')$. Therefore, as $D \ge 16 \Delta/\vep$ we derive that
\[
|\ol{\cW}_{\cP}(\Graph_0)|\leq |\{v\in V(\Graph): {\rm deg}_{\Graph'}(v)\geq D\}|\leq (\vep/8 \Delta)\cdot e(\Graph') \leq (\vep/4 \Delta) \cdot e(\Graph_{0}).
\]
Thus \eqref{eq:barcWcP} is satisfied with $\wt \vep=\vep/(4 \Delta)$. Hence denoting $\wh{\mathscr{N}}(\ul{\bm e},\wt{\bm w})$ to be the number of graphs $\Graph_0 \in \Omega$ with $|\cW_{\cP}(\Graph_0)| = \wt{\bm w}$ and satisfying \eqref{eq:wte-cond}-\eqref{eq:barcWcP}, and applying Lemma \ref{lem:graph-spec-prop-bd} we deduce that
\begin{align}\label{eq:whmathscrN}
\wh{\mathscr{N}}(\ul{\bm e}, \wt{\bm w})  & \le \exp \left(\log(1/p) \cdot (\Delta/2)  \cdot \left\{\wt{\bm w}+  (\vep/(4 \Delta)) \cdot (\wt{\bm e}_{1,2} +\wt{\bm e}_{2,2})\right\} + K (\log \log n) \cdot \wt{\bm e} \right) \notag\\
& \le  \exp \left(\log(1/p) \left\{(\Delta/2) \cdot \wt{\bm w}+  (\vep/4) \cdot (\wt{\bm e}_{1,2} +\wt{\bm e}_{2,2})\right\}  \right), 
\end{align}
where the last step is due to the fact as $\varpi_\Delta(\Graph')=\Graph_0$, it follows from the definition of $\cP_{\Graph_0}$ and $E_{1,1}(\Graph')$ that $\wt{\bm{e}}_{1,1}=\wt e_{1,1}(\Graph_0)=0$, and thus $\wt{\bm e}=\wt{\bm e}_{1,2}+\wt{\bm e}_{2,2}$.

Equipped with \eqref{eq:whmathscrN} we observe that
\begin{multline*}
\P({\rm Core}_{2,j,\b}) \le \sum_{\ul{\bm e}, \wt{\bm w}} \wh{\mathscr{N}}(\ul{\bm e}, \wt{\bm w}) \cdot p^{\wt{\bm e}_{1,2} + \wt{\bm e}_{2,2}} \le  \sum_{\ul{\bm e}, \wt{\bm w}} \exp \left(-\log(1/p) \left\{\f{\wt{\bm e}_{1,2}}{2}+\wt{\bm e}_{2,2}-  (\vep/4) \cdot (\wt{\bm e}_{1,2} +\wt{\bm e}_{2,2})\right\}  \right),
\end{multline*}
where the sum is over allowable ranges of $\wt{\bm w}$ and $\ul{\bm e}$, and the last step follows from the fact that for any $\Graph_0 \in \Omega$ we have {$d_{\min}(\Graph_0) \ge \Delta$} \corAB{and thus $\Delta \wt{\bm w} \le  \wt{\bm e}_{1,2} + 2 \wt{\bm e}_{1,1}=\wt{\bm e}_{1,2}$.} 

\corAB{Finally using that 
\[
\wt {\bm w} \le 2(\wt{\bm e}_{1,2} + \wt{\bm e}_{2,2}) \le 2 \wt{\bm e} \le 2 \bar C n^2 p^{\Delta} \log(1/p),
\] 
and the lower bound 
\[
(1/2) \cdot(\wt{\bm e}_{1,2} + \wt{\bm e}_{2,2}) \ge \f12 \bar C n^2 p^{\Delta} \log(1/p) /(\log \log n)^2, 
\]
induced} by the fact $\Graph_0 \in \cJ_j$, for $j \in \llbracket L_n \rrbracket$, one evaluates the above sum to obtain a desired bound. We omit further details. This completes the proof of the lemma.  
\end{proof}

\subsection{Graphs with large $\cN_{1,1}$}\label{sec:largecN-11}
 In this section our goal is to prove Lemmas \ref{l:core12} and \ref{l:core2ja}. That is, to show that seed graphs with large $\cN_{1,1}({\sf H}, \Graph)$ are entropically stable, where we recall that the notation $\cN_{1,1}({\sf H}, \Graph)$ denotes the number of copies of ${\sf H}$ that uses at least one edge from $E_{1,1}(\Graph)$ with $E_{1,1}(\Graph)$ is as in \eqref{eq:E11} (edges where both endpoints have low degree). Similar to Section \ref{sec:combin-graph-bd} here we will also rely on Lemma \ref{lem:graph-spec-prop-bd}. We remind the reader that to complete the proofs of Lemmas \ref{l:core11} and \ref{l:core2jb} we needed lower bounds on $\wt{\bm e}_{1,2}(\Graph) + \wt{\bm e}_{2,2}(\Graph)$. The following result shows that such bounds follow once we assume a lower bound on $\cN_{1,1}({\sf H}, \Graph)$. 
 
\begin{prop}
\label{l:seqcounting}
Fix $\tau >0$ and a $\Delta$-regular connected graph ${\sf H}$ with $\Delta \ge 2$. Let $\Graph$ be a graph with 
\beq\label{eq:large-e-bd}
e(\Graph) \le c_n \cdot \bar C n^2 p^\Delta \log(1/p),
\eeq
for some $\f12 (\log \log n)^{-2} \le c_n \le 1$. Assume
\beq\label{eq:cN-0-lbd}
\cN_{1,1}({\sf H}, \Graph) \ge \f{c_n}{2} \cdot \tau n^{v_{\sf H}} p^{e({\sf H})}.
\eeq
If $np \gg (\log n)^{1/(v_{\sf H}-2)}$ then we have 
\[
e_{1,2}(\Graph) +e_{2,2}(\Graph)\ge (n^2 p^2)^{1+\f{1}{4(v_{\sf H}-1)}},
\]
for all large $n$.
\end{prop}

The proof of Proposition \ref{l:seqcounting} is deferred to Section \ref{s:comb-reg}. In the remainder of this section using Proposition \ref{l:seqcounting} we prove Lemmas \ref{l:core12} and \ref{l:core2ja}.

\begin{proof}[Proof of Lemma \ref{l:core12} (assuming Proposition \ref{l:seqcounting})]
This proof is an easy consequence of Lemma \ref{lem:graph-spec-prop-bd} and Proposition \ref{l:seqcounting}. Fix $\wh{\ul{\bm e}}:= ({\bm e}_{1,1}, {\bm e}_{1,2}, {\bm e}_{2,2})$, a non-negative integer ${\bm w}$, and let $\mathscr{N}_\#(\ul{\wh{\bm e}}, {\bm w})$ be the number of graphs with
\beq\label{eq:graph-cond-12}
(e_{1,1}(\Graph), e_{1,2}(\Graph), e_{2,2}(\Graph)) = \wh{\ul{\bm e}}, \qquad \cW(\Graph)={\bm w},
\eeq
and satisfying the hypothesis of the event $\wt{\rm Core}_{1,2}$ (see \eqref{eq:wtcore12} for the definition of the event $\wt{\rm Core}_{1,2}$), where we refer the reader to \eqref{eq:cW}, and \eqref{eq:e-121}-\eqref{eq:e-123} to recall the definitions of $\cW(\Graph)$, $e_{1,2}(\Graph)$, and $e_{2,2}(\Graph)$, and recall that $e_{1,1}(\Graph)= |E_{1,1}(\Graph)|$. Before applying Lemma \ref{lem:graph-spec-prop-bd} we note that any graph $\Graph$ satisfying the hypothesis of $\wt{\rm Core}_{1,2}$ must also satisfy the inequality $\cN_{1,1}({\sf H}, \Graph) \ge \vep \delta n^{v_{\sf H}} p^{e({\sf H})}$. Therefore applying Proposition \ref{l:seqcounting} with $\tau=\vep \delta$ and $c_n=1$, as $np^{\Delta/2} \gg (\log n)^{1/(v_{\sf H}-2)}$, yields that for any such graph $\Graph$
\beq\label{eq:e122-lbd}
e_{1,2}(\Graph) + e_{2,2}(\Graph) \ge n^2 p^\Delta   \cdot (\log n)^\gamma,
\eeq
for some $\gamma >0$. Thus setting $\cP_\Graph$ as in \eqref{eq:cP-dfn-1}, applying Lemma \ref{lem:graph-spec-prop-bd}, and  proceeding as in the steps leading to \eqref{eq:mathscrN} we find that 
\begin{align}\label{eq:mathscrNs}
\mathscr{N}_\#(\wh{\ul {\bm e}}, {\bm w}) & \le \exp\left(\log(1/p) \left\{{\bm e}_{1,1} + \f12 {\bm e}_{1,2} + (\vep/16) \cdot ({\bm e}_{1,2}+ {\bm e}_{2,2})\right\} + K \log \log n \cdot {\bm e}\right)\notag\\
& \le \exp\left(\log(1/p) \left\{{\bm e}_{1,1} + \f12 {\bm e}_{1,2} + (\vep/8) \cdot ({\bm e}_{1,2}+ {\bm e}_{2,2})\right\} \right),
\end{align}
where the last step is a consequence of \eqref{eq:e122-lbd} and the fact that $e(\Graph) \le \bar C n^2 p^\Delta \log (1/p)$. Having obtained \eqref{eq:mathscrNs} we now again proceed similar to the steps leading to \eqref{eq:wtcore11p-bd} and use the lower bound \eqref{eq:e122-lbd} to derive the desired upper bound on the probability of $\wt{\rm Core}_{1,2}$. To avoid repetition we omit further details. This finishes the proof of this lemma.  
\end{proof}

Next we prove Lemma \ref{l:core2ja}. It is quite similar to that of Lemma \ref{l:core12}. 

\begin{proof}[Proof of Lemma \ref{l:core2ja} (assuming Proposition \ref{l:seqcounting})]
Fix $j \in \llbracket L_n \rrbracket$. As before we fix $\wh{\ul{\bm e}}:=({\bm e}_{1,1}, {\bm e}_{1,2}, {\bm e}_{2,2})$ and ${\bm w}$, and set $\wh{\mathscr N}_\#(\wh{\ul{\bm e}}, {\bm w})$ to be the set of the graphs satisfying \eqref{eq:graph-cond-12} and the hypothesis of the event $\wt{\rm Core}_{2,j,\a}$ (see \eqref{eq:core2ja} for its definition). Once again we apply Proposition \ref{l:seqcounting}. As any graph $\Graph$ satisfying the hypothesis of the event $\wt{\rm Core}_{2,j\a}$ have 
\beq\notag
e(\Graph) \le 2^{-(j-1)} \vep \delta n^{v_{\sf H}} p^{e({\sf H})}, \, N({\sf H}, \Graph) \ge (1 - 3 \vep - \gs_j \vep)\delta n^{v_{\sf H}} p^{e({\sf H})}, \, \text{ and } \, 
\cN_{1,1}({\sf H}, \Graph) \ge 2^{-j} \vep \delta n^{v_{\sf H}} p^{e({\sf H})},
\eeq
we apply Proposition \ref{l:seqcounting} with $\tau=\vep \delta$ and $c_n = 2^{-(j-1)}$ to see that \eqref{eq:e122-lbd} continues to hold, uniformly for any $j \in \llbracket L_n \rrbracket$.

Note that moving from a graph $\Graph$ to its $\Delta$-core does not change $N({\sf H}, \Graph)$ and $\cN_{1,1}({\sf H}, \Graph)$. It only decreases the number of edges. So without loss of generality we may also assume that any graph $\Graph$ satisfying the hypothesis of $\wt{\rm Core}_{2,j,\a}$ must have $d_{\min}(\Graph) \ge \Delta$. Therefore arguing exactly same as above we notice that \eqref{eq:mathscrNs} holds for $\wh{\mathscr N}_\#(\wh{\ul{\bm e}}, {\bm w})$ as well, uniformly for all $j \in \llbracket L_n \rrbracket$. Thus repeating the same arguments as in the proof of Lemma \ref{l:core12} the desired bounds is derived. 
Further details are omitted. 
\end{proof}


\section{Proof of Proposition \ref{l:seqcounting}}
\label{s:comb-reg}
We first consider the case $\Delta=2$ or equivalently ${\sf H}=C_\ell$ for some $\ell \ge 3$. In this case we derive Proposition \ref{l:seqcounting} using a careful counting argument. We then generalize this result for any $\Delta \ge 3$ by decomposing a regular connected graph ${\sf H}$ into vertex disjoint cycles and edges. 
\subsection{Proof of Proposition \ref{l:seqcounting} for cycles}
{As a preliminary step we will show that the number of copies of $C_\ell$ in $\Graph$ that is contained in the subgraph spanned only by the vertices of low degree is $o(n^\ell p^\ell)$. To carry out this step we need the following notation.} 

For any graph ${\sf H}$ we write $\wt \cN_{1,1}({\sf H},\Graph)$ to denote the number of labelled copies of ${\sf H}$ consisting of only edges belonging to $E_{1,1}(\Graph)$ and set
\[
\bar \cN_{1,1}( {\sf H}, \Graph) := \cN_{1,1}({\sf H}, \Graph) - \wt \cN_{1,1}({\sf H}, \Graph).
\]
That is, $\bar \cN_{1,1}({\sf H}, \Graph) $ is the number of labeled copies of ${\sf H}$ in $\Graph$ that uses at least one edge from $E_{1,1}(\Graph)$ and at least one belonging to $E_{1,2}(\Graph) \cup E_{2,2}(\Graph)$. 

{The next lemma yields a bound on $\wt \cN_{1,1}(C_\ell, \Graph)$. Later in Section \ref{sec:general-counting} we will need an analogous result for general connected regular graph. The proof for the general connected regular graph being identical with that for $C_\ell$ we state and prove the result for the general case.} 


\begin{lem}\label{lem:wtN11}
Let $\tau, c_n, \Delta$, and ${\sf H}$ be as in Proposition \ref{l:seqcounting}. Suppose $\Graph$ be a graph satisfying \eqref{eq:large-e-bd}. If $np^{\Delta/2} \gg (\log n)^{1/(v_{\sf H}-2)}$ then 
\[
\wt \cN_{1,1}({\sf H}, \Graph) \le c_n \cdot \f\tau4 n^{v_{\sf H}} p^{e({\sf H})},
\]
for all large $n$. 
\end{lem}

\begin{lem}\label{lem:e-bar-ub}
For any $\ell \ge 3$ and a graph $\Graph$ we have
\beq\label{eq:bar-cN-0-ubd}
\bar \cN_{1,1}(C_\ell,\Graph) \le  \ell 2^{\ell+1} D^\ell \cdot \left(2 ({\bm e}_{1,2}(\Graph)+ {\bm e}_{2,2}(\Graph))\right)^{\lfloor (\ell-1)/2 \rfloor}.
\eeq
\end{lem} 

Using Lemmas \ref{lem:wtN11} and \ref{lem:e-bar-ub} let us now prove Proposition \ref{l:seqcounting} for cycle graphs $C_\ell$. 

\begin{proof}[Proof of Proposition \ref{l:seqcounting} for cycles]
Since by assumption $\cN_{1,1}(C_\ell,\Graph) \ge (c_n/2) \cdot \tau n^\ell p^\ell$, applying Lemma \ref{lem:wtN11} for ${\sf H}=C_\ell$ we deduce that
\[
\bar \cN_{1,1}(C_\ell, \Graph) \ge c_n \cdot \f\tau 4 n^\ell p^\ell.
\]
Therefore Lemma \ref{lem:e-bar-ub} now implies that there exists some constant $\wt c > 0$, depending on $\tau$, so that  
\begin{align*}
{\bm e}_{1,2}(\Graph)+ {\bm e}_{2,2}(\Graph)  & \ge\f12 \cdot  \left( c_n \cdot \f\tau4 \cdot D^{-\ell} \ell^{-1} 2^{-(\ell+1)}\right)^{1/\lfloor (\ell-1)/2\rfloor} \cdot \left(n^2p^2\right)^{\f{\ell}{\ell-1}}\\
& \ge \wt c \cdot (\log \log n)^{-\f{4}{\ell-2}} \cdot \left(n^2 p^2\right)^{1+ \f{1}{\ell-1}} \ge \left(n^2 p^2\right)^{1+\f{1}{2(\ell-1)}},
\end{align*}
for all large $n$, where the penultimate step uses the fact that $c_n \ge \frac{1}{2}(\log \log n)^{-2}$ and the last step uses that 
\[
(\log \log n)^{4/(\ell-2)} \ll (\log n)^{\f{1}{2(\ell-1)(\ell-2)}} \ll (np)^{\f{1}{2(\ell-1)}}.
\]
This completes the proof. 
\end{proof}
We now proceed to the proof of Lemma \ref{lem:wtN11}. The proof is straightforward.  
 
\begin{proof}[Proof of Lemma \ref{lem:wtN11}]
We construct a labelled copy of ${\sf H}$ contributing to $\wt\cN_{1,1}({\sf H},\Graph)$ as follows:  Fix any edge $e_0:=(u_0,u_0') \in E({\sf H})$. This edge needs to be mapped to an edge in $E_{1,1}(\Graph)$. It can be done in at most $2|E_{1,1}(\Graph)|$. Having chosen this edge, as both end points of $e_0$ are mapped to vertices of degree at most $D$, we note that each of the neighbors of $u_0$ and $u_0'$ can be chosen at most in $D$ ways. Furthermore by the definition of $\wt\cN_{1,1}({\sf H},\Graph)$ each such neighbor must also be mapped vertices of low degree. Since ${\sf H}$ is connected we carry on this argument by induction to deduce that each of the remaining $(v_{\sf H}-2)$ vertices of ${\sf H}$, excluding $u_0$ and $u_0'$, can be chosen in at most $D$ ways. This shows that
\[
\wt\cN_{1,1}({\sf H}, \Graph) \leq 2|E_{1,1}(\Graph)|\cdot D^{v_{\sf H}-2} \le 2 e(\Graph) \cdot D^{v_{\sf H}-2}=c_n \cdot o(n^{v_{\sf H}} p^{e({\sf H})}),
\]
where in the last step we used that $e(\Graph) \le c_n \cdot \bar C n^2 p^\Delta \log(1/p)$, and the fact that $np^{\Delta/2} \gg (\log n)^{1/(v_{\sf H}-2)}$ implies $n^2 p^\Delta \log (1/p) \ll (np^{\Delta/2})^{v_{\sf H}} = n^{v_{\sf H}} p^{e({\sf H})}$. This completes the proof of the lemma. 
\end{proof}

We need a few combinatorial definitions before proving Lemma \ref{lem:e-bar-ub}. For any labelled copy ${\sf C}$ of either the $\ell$-cycle or the path of length $\ell$ in $\Graph$, we associate it with an element ${\bm s}({\sf C}):=(s_1,s_2,\ldots, s_{\ell})$ of $\{0,1\}^{\ell}$ as follows. We set $s_{i}=0$ if the $i$-th edge of ${\sf C}$ (according to the labelling) is in $E_{1,1}(\Graph)$ and $s_{i}=1$ otherwise. For any fixed ${\bm s}\in \{0,1\}^{\ell}$ and $v_1, v_2 \in V(\Graph)$ we write $\cN(P_\ell, \Graph, {\bm s}, v_1, v_2)$ to denote the number of labelled copies ${\sf C}$ of $P_\ell$, the path of length $\ell$, such that ${\bm s}({\sf C})={\bm s}$ with the starting and the ending vertices being $v_1$ and $v_2$ respectively. 

We have the following counting lemma which will be key in proving Lemma \ref{lem:e-bar-ub}.

\begin{lem}
\label{path}
For any integer $\ell\ge 3$, ${\bm s}\in \{0,1\}^{\ell}$, and $v_1, v_2 \in V(\Graph)$ we have 
\[
\cN(P_\ell, \Graph, {\bm s}, v_1, v_2) \leq D^{\ell}\left(2 (e_{1,2}(\Graph)+e_{2,2}(\Graph))\right)^{\lfloor \ell/2 \rfloor}.
\] 
\end{lem}

\begin{proof}
The proof is done by an induction argument. Clearly for $\ell=1$, there can be at most one copy of $P_{1}$ with the fixed starting and ending points. 

For $\ell=2$, let  us consider different possible choices of ${\bm s}$. Clearly if ${\bm s}=00$, since the starting vertex is fixed, then there are at most $D$ possibilities for each of the edges and  hence $\cN(P_2, \Graph, {\bm s}, v_1, v_2)\leq D^2$. If ${\bm s}\in \{10, 01, 11\}$ then, as the starting and the ending vertices are fixed, there are at most $2 \bar {e}(\Graph)$ choices for the edge corresponding to the $1$, where for ease in writing we use the shorthand $\bar e(\Graph):= e_{1,2}(\Graph) + e_{2,2}(\Graph)$. Having chosen this edge, because the two leaf vertices are fixed, there are at most one choices for the remaining edge. This gives the desired bound for $\ell=2$, any ${\bm s} \in \{0,1\}^2$ and $v_1, v_2 \in V(\Graph)$.

Let us suppose that the statement of the lemma is true for all $\ell=1,2,\ldots, q-1$, any ${\bm s} \in \{0,1\}^\ell$, and $v_1, v_2 \in V(\Graph)$. We now establish the lemma for $\ell=q\geq 3$, any ${\bm s} \in \{0,1\}^q$, and $v_1, v_2 \in V(\Graph)$. Write ${\bm s}=s_1s_2{\bm s}'$, where $s_1$ and $s_2$ are the first two digits of ${\bm s}$ and ${\bm s}'$ is the remaining substring of length $q-2$.

If $s_1=0$ then there are at most $D$ choices for the first edge $e=(v_1,v_1')$, and for each such choice there are at most $\cN(P_{q-1}, \Graph, s_2{\bm s}', v_1', v_2)$ choices for the remaining edges. Hence by induction hypothesis we obtain 
\[
\cN(P_q, \Graph, {\bm s}, v_1,v_2)\leq D \cdot D^{q-1} (2\bar e(\Graph))^{\lfloor (q-1)/2\rfloor} \leq D^{q}(2 \bar {e}(\Graph))^{\lfloor q/2 \rfloor}.
\]
If $s_1=1$ we need to consider two cases depending on whether $s_2=0$ or $1$. If $s_1s_2=10$, arguing as above, we see that there are at most $2 \bar {e}(\Graph) D$ many choices for the first two edges $e_1 =(v_1,v_1')$ and $e_2=(v_1', v_1'')$. For each of these choices there at most $\cN(P_{q-2}, \Graph, {\bm s}', v_1'', v_2)$ many choices for the remaining edges. So by induction hypothesis in this case we derive 
\[
\cN(P_q, \Graph, {\bm s}, v_1, v_2) \leq (2 \bar {e}(\Graph) D) \cdot D^{q-2} (2 \bar {e}(\Graph))^{\lfloor (q-2)/2 \rfloor} \leq D^{q}(2 \bar {e}(\Graph))^{\lfloor q/2 \rfloor}. 
\]
If $s_1s_2=11$ we first choose the second edge $e_2=(v_1', v_1'')$, where there are at most $2 \bar {e}(\Graph)$ many choices. Given this choice, as the first vertex is fixed at $v_1$ there are at most one choice for the first edge $e_1=(v_1, v_1')$. Now there are at most a total of $N(P_{q-2}, \Graph, {\bm s}', v_1'', v_2)$ many choices for the third edge onwards. Hence, by induction hypothesis we get 
\[
\cN(P_q, \Graph, {\bm s}, v_1, v_2) \leq (2 \bar {e}(\Graph) ) \cdot D^{q-2} (2 \bar {e}(\Graph))^{\lfloor (q-2)/2 \rfloor} \leq D^{q}(2 \bar {e}(\Graph))^{\lfloor q/2 \rfloor}.
\] 
This completes the induction and the proof of the lemma. 
\end{proof}
 
We are now ready to prove Lemma \ref{lem:e-bar-ub}. 

\begin{proof}[Proof of Lemma \ref{lem:e-bar-ub}]
Let $\cN^*(C_\ell, \Graph)$ denote the number of labelled copies of $C_\ell$ in $\Graph$ such that the first edge (according to the labelling) belongs to $E_{1,2}(\Graph) \cup E_{2,2}(\Graph)$ and the second edge belongs to $E_{1,1}(\Graph)$. Since $\bar \cN_{1,1}(C_\ell, \Graph)$ counts the number of labelled copies of $C_\ell$ in $\Graph$ that have at least one edge in $E_{1,1}(\Graph)$ and at least one belonging to $E_{1,2}(\Graph) \cup E_{2,2}(\Graph)$, every such labelled copy must contain an edge in $E_{1,1}(\Graph)$ that is adjacent to some edge $E_{1,2}(\Graph) \cup E_{2,2}(\Graph)$ also contained in that copy. Therefore we find that  
\[
\bar \cN_{1,1}(C_\ell, \Graph)\leq 2\ell \cN^*(C_\ell, \Graph),
\]
where the factor $\ell$ is due to the choice of the location of the edge belonging to $E_{1,2}(\Graph) \cup E_{2,2}(\Graph)$ that is adjacent to the edge in $E_{1,1}(\Graph)$ and the factor two is due to the orientation of that edge. So it suffices to prove
\[
\cN^*(C_\ell, \Graph) \leq  2^{\ell} D^\ell \cdot (2 \bar {\bm e})^{\lfloor (\ell-1)/2 \rfloor}.
\]
Recall the $\ell$-bit string ${\bm s}({\sf C})$ associated  with every labelled copy ${\sf C}$ of $C_{\ell}$ in $\Graph$. Clearly for any labelled copy that is counted in 
$\cN^*(C_\ell, \Graph)$ the string ${\bm s}({\sf C})$ must start with the substring $10$. For any such ${\bm s}$, for ease of explanation, we further introduce the notation $\cN^*(C_\ell, \Graph, {\bm s})$ to denote the number of labelled copies ${\sf C}$ of $C_{\ell}$ that are counted in $\cN^*(C_\ell, \Graph)$ such that ${\bm s}({\sf C})={\bm s}$. Equipped with this notation it now suffices to prove that for any $\ell$-bit string ${\bm s}$ that starts with $10$ we have 
\begin{equation}
\label{e:sbound}
\cN^*(C_\ell, \Graph, {\bm s}) \leq D^{\ell} \cdot (2 \bar {\bm e})^{\lfloor (\ell-1)/2 \rfloor}.
\end{equation}
Fix an ${\bm s}$ as above and let ${\bm s}'$ be the substring of ${\bm s}$ that ends with the second $1$ of ${\bm s}$. Observe that by our construction the last edge of any cycle contributing to $\cN^*(C_\ell, \Graph)$ must belong to $E_{1,2}(\Graph) \cup E_{2,2}(\Graph)$ because the first edge also belonging to $E_{1,2}(\Graph) \cup E_{2,2}(\Graph)$ can have only one of its end point adjacent to an edge in $E_{1,1}(\Graph)$. Hence for any $\ell$-bit starting ${\bm s}$ starting with $10$ containing only one $1$ we trivially have $N^*(C_\ell, \Graph, {\bm s})=0$. Thus such strings can be safely ignored and the substring ${\bm s}'$ is well defined. Let us now write ${\bm s}={\bm s}'{\bm s}''$. Note that by definition the length of the substring ${\bm s}'$ is at least three. 

If ${\bm s}''$ is an empty string we get 
 \[
 \cN^*(C_\ell, \Graph, {\bm s})\leq (2 \bar {\bm e}) D^{\ell-1},
 \]
as there are at most $2 \bar {\bm e}$ many choices for the first edge whereas each subsequent edge must start at a small degree vertex and hence the total number of choices is upper bounded by $D$. This completes the proof of \eqref{e:sbound} in the case ${\bm s}''$ is empty.

If ${\bm s}''$ is non-empty we let $3\leq q<\ell$ to be the length of ${\bm s}'$. Note that arguing as before there are at most $(2 \bar {\bm e}) D^{q-1}$ many choices for the first $q$ edges, and for each such choice the total number of choices for the remaining $(\ell-q)$ edges is upper bounded by $N(P_{\ell-q}, \Graph, {\bm s}'', v_t, v_1)$, where $v_q$ is the end point of the $q$-th edge and $v_1$ is the appropriate end point of the first edge. This observation together with Lemma  \ref{path} now yields that 
\[
 \cN^*(C_\ell, \Graph, {\bm s}) \leq (2 \bar {\bm e}) D^{q-1} \cdot D^{\ell-q}(2 \bar {\bm e})^{\lfloor (\ell-q)/2 \rfloor}
\]
Noting that $q\geq 3$ implies $1+\lfloor (\ell-q)/2\rfloor \leq \lfloor (\ell-1)/2 \rfloor$ the proof completes.   
\end{proof}

\subsection{Proof of Proposition \ref{l:seqcounting} for general regular graphs}\label{sec:general-counting}
Equipped with Lemma \ref{lem:wtN11} we note that to extend Proposition \ref{l:seqcounting} for any general regular connected graph ${\sf H}$ it suffices to prove the following result. 

\begin{lem}\label{lem:e-bar-ub-H}
Let ${\sf H}$ be  $\Delta$-regular connected graph with $\Delta \ge 3$. Fix a graph $\Graph$. Then
\beq\label{eq:bar-cN-0-ubd}
\bar \cN_{1,1}({\sf H},\Graph) \le  C_{{\sf H}, D} \cdot \left(2 ({\bm e}_{1,2}(\Graph)+ {\bm e}_{2,2}(\Graph))\right)^{(v_{\sf H}/2)-1}
\eeq
for some constant $C_{{\sf H},D} <\infty$ depending only on the graph ${\sf H}$ and the threshold $D$.
\end{lem} 

To prove Lemma \ref{lem:e-bar-ub-H} we will use the bounds derived in Lemmas \ref{lem:e-bar-ub} and \ref{path} for cycles and paths, and the following graph decomposition result. 

\begin{prop}\label{prop:graph-decompose}
Let ${\sf H}$ be a $\Delta$-regular graph with $\Delta \ge 3$. Fix an edge $e:=(u_1,u_2) \in E({\sf H})$. Then there exists a collection of vertex disjoint cycles $\{\mathscr{C}_i\}_{i=1}^{\lambda_1}$ and edges $\{\mathfrak{e}_i\}_{i=1}^{\lambda_2}$, for some $\lambda_1, \lambda_2 \ge 0$, with the property
\beq\label{eq:e-not-there}
e \notin \bigcup_{i=1}^{\lambda_1}E(\mathscr{C}_i)   \cup \bigcup_{j=1}^{\lambda_2} \{\mathfrak{e}_j\},
\eeq
such that vertices $V({\sf H})$ can be covered by the vertices of the collection $\mathcal C:= \{\{\mathscr{C}_i\}_{i=1}^{\lambda_1}, \{\mathfrak{e}_j\}_{j=1}^{\lambda_2}\}$. 
\end{prop}

Proposition \ref{prop:graph-decompose} essentially follows by applying K\"{o}nig's proper edge coloring theorem on the bipartite double cover of ${\sf H}$. The proof is postponed to Appendix \ref{sec:decompose}.


Let us now proceed to the proof of Lemma \ref{lem:e-bar-ub-H}. To this end, we let $\mathcal{S}_0$ to be the set of all ordered pairs of neighbouring edges of ${\sf H}$. Fix ${\bm q}=\{(u_1,u_2), (u_2,v)\}\in \mathcal{S}_0$ and denote $\bar{\cN}_{1,1}({\sf H}, \Graph, {\bm s})$ to be the number of labelled copies of ${\sf H}$ in $\Graph$ such that $(u_1,u_2)$ is mapped to an edge in $E_{1,1}(\Graph)$ and $(u_1,v)$ is mapped to an edge in $E_{1,2}(\Graph)$. Recall that $\bar{\cN}_{1,1}({\sf H}, \Graph)$ counts number of labelled copies of ${\sf H}$ in $\Graph$ that uses at least one edge from $E_{1,1}(\Graph)$ and one from $E_{1,2}(\Graph) \cup E_{2,2}(\Graph)$. Since an edge in $E_{1,1}(\Graph)$ cannot be adjacent to that in $E_{2,2}(\Graph)$, for any such copy there must exist a pair of adjacent edges in ${\sf H}$ that are mapped to edges in $E_{1,1}(\Graph)$ and $E_{1,2}(\Graph)$, respectively. Hence
\[
\bar{\cN}_{1,1}({\sf H}, \Graph)\leq \sum_{{\bm q}\in \mathcal{S}_0}\bar{\cN}_{1,1}({\sf H}, \Graph, {\bm q}).
\]
Since $|\mathcal{S}_0|$ depends only on ${\sf H}$ the following result immediately implies Lemma \ref{lem:e-bar-ub-H}. 

\begin{lem}\label{lem:e-bar-ub-H-fixed}
Let ${\sf H}, \Graph$, and $\Delta$ be as in Lemma \ref{lem:e-bar-ub-H}. For each fixed ${\bm q} \in \cS_0$ we have the following:
\beq\label{eq:bar-cN-0-ubd-H-fixed}
\bar \cN_{1,1}({\sf H},\Graph, {\bm q}) \le  C'_{{\sf H},D} \cdot \left(2 ({\bm e}_{1,2}(\Graph)+ {\bm e}_{2,2}(\Graph))\right)^{(v_{\sf H}/2)-1},
\eeq
 where $C'_{{\sf H},D}< \infty$ is some positive constant depending only on ${\sf H}$ and $D$.
\end{lem} 

The remainder of the section is devoted to the proof of Lemma \ref{lem:e-bar-ub-H-fixed}. This requires introduction of some more notation. We make this following claim which is a consequence of Proposition \ref{prop:graph-decompose}. 

\begin{claim}\label{claim:F} There exists a collection $\{F_1, F_2,\ldots, F_m\}$ of edges and cycles in ${\sf H}$  satisfying the following properties:

\begin{itemize}
\item The collection covers all vertices of ${\sf H}$, i.e., $\cup_{i=1}^{m} V(F_i)=V({\sf H})$.
\item The edge $(u_1,u_2)$ is not contained in any of the $F_i$'s.
\item $u_1\in V(F_1), u_2\in V(F_2), v\in V(F_3)$.
\item {The sub collection $\{F_i\}_{i \ge 4}$ is vertex disjoint. The sub collection $\{F_i\}_{i=1}^3$ is also vertex disjoint from $\{F_i\}_{i \ge 4}$. Furthermore for any pair of indices $j_1, j_2 =1,2,3$, either $F_{j_1}=F_{j_2}$ or $V(F_{j_1}) \cap V(F_{j_2}) = \emptyset$.} 
\item For each $i\geq 4$, there exists $v_i'\in V(F_i)$ and $v_i''\in \cup_{j=1}^{i-1} V(F_j)$ such that $(v_i',v_i'')\in E({\sf H})$.
\end{itemize} 
\end{claim}

{By Proposition \ref{prop:graph-decompose} there exists a vertex disjoint collection of edges and cycles $\{\wt F_i\}$ such that $\cup_i V(\wt F_i) = V({\sf H})$ and the edge $(u_1, u_2)$ does not belong to the collection $\{\wt F_i\}$. Since Proposition \ref{prop:graph-decompose} does not guarantee that $u_1, u_2$, and $v$ belong to three different edges or cycles, after relabelling $\{\wt F_i\}$ we now obtain the collection $\{F_i\}$ satisfying the first four properties in Claim \ref{claim:F}.} The last property is due to the connectedness of ${\sf H}$ and another relabelling if necessary.  

Let ${\sf H}_*$ be the subgraph spanned by the edges $E(F_{1}) \cup E(F_2) \cup E(F_3) \cup\{(u_1,u_2), (u_2,v)\}$. For a later use we also define $\wh{\sf H}$ to be subgraph of ${\sf H}$ spanned by edge set $E({\sf H}_*) \cup (\cup_{i=4}^m E(F_i) )\cup  \{(v_i', v_i''), i=4,5,\ldots, m\}$.

Now we denote $\bar{\cN}_{1,1}({\sf H}_*, \Graph, {\bm q})$ to be the number of labelled copies of ${\sf H}_*$ in $\Graph$ such that $(u_1,u_2)$ is mapped to an edge in $E_{1,1}(\Graph)$ and $(u_2,v)$ is mapped to an edge in $E_{1,2}(\Graph)$. We now have the following lemma which will be used in the proof of Lemma \ref{lem:e-bar-ub-H-fixed}.

\begin{lem}\label{lem:e-bar-ub-H--star-fixed}
Let ${\sf H}, \Graph, \Delta, {\bm q}$, and $\cS_0$ be as in Lemma \ref{lem:e-bar-ub-H-fixed}. Then we have 
\beq\label{eq:bar-cN-0-ubd-H-fixed}
\bar \cN_{1,1}({\sf H}_*,\Graph, {\bm q}) \le  C^*_{{\sf H},D} \cdot \left(2 ({\bm e}_{1,2}(\Graph)+ {\bm e}_{2,2}(\Graph))\right)^{(v_{{\sf H}_*}/2)-1}
\eeq
for some positive constant $C^*_{{\sf H},D}$ depending only on ${\sf H}$ and $D$.
\end{lem} 

\begin{proof}
{First we begin with the case $F_1=F_2$. In this case, as the edge $(u_1,u_2)$ does not belong to the collection $\{F_i\}_{i=1}^m$ and $u_1, u_2 \in V(F_1)$, we note that $F_1$ must be a cycle. Now choose the edge $(u_2,v)$. Since $(u_2,v)$ needs to be mapped to an edge in $E_{1,2}(\Graph)$ this  has at most ${\bm e}_{1,2}(\Graph)$ choices. Note that the orientation is automatically fixed as we know $u_2$ must be mapped to a vertex of low degree. We next pick the vertex $u_1$ for which we have at most $D$ choices. We claim that the remaining vertices of $F_1$ can be chosen in at most $C_1^\star\left( {\bm e}_{1,2}(\Graph)+ {\bm e}_{2,2}(\Graph)\right)^{(v_{F_1}-4)/2}$ ways for some constant $C_1^\star<\infty$.} 

{To see this, next choose $\{v_1,v_2\}$ and $\{v_3,v_4\}$, the neighbors of $u_1$ and $u_2$, respectively, that are in $V(F_1)$. This can be done in at most $D^4$ ways (in the case when the edge $(u_2,v)$ belongs to the cycle $F_1$ this bound is even smaller equaling $D^3$).  Having chosen the neighbors the graphs spanned by the edges $E(F_1) \setminus \{(u_1, v_1), (u_1,v_2), (u_2,v_3), (u_2,v_4)\}$ is a disjoint union of (at most) two paths such that the sum of their lengths is $v_{F_1}-4$. Since the starting and the ending points of both of these are now fixed, upon applying Lemma \ref{path} we derive that the number of choices of the remaining vertices of $F_1$ is bounded above by $\wt C_1 \left({\bm e}_{1,2}(\Graph)+ {\bm e}_{2,2}(\Graph)\right)^{(v_{F_1}-4)/2}$ for some $\wt C_1 < \infty$. This yields the desired bound on the number of choices of the vertices $V(F_1)$.}

{
It remains to choose the remaining vertices of $F_3$. Let us consider the easier case when either $F_3=F_1$. In this case, noting that $v_{{\sf H}_*}=v_{F_1}$ we already have the desired bound from above.} 

{
So, without loss of generality let us assume that $F_3$ is distinct from $F_1$. We aim to show that in this case the total number of choices for the vertices in $V(F_3)\setminus \{v\}$ is bounded by 
\beq\label{eq:F-3-bd}
C_3\left({\bm e}_{1,2}(\Graph)+ {\bm e}_{2,2}(\Graph)\right)^{v_{F_3}/2},
\eeq
for some other constant $C_3 <\infty$. This together with the claim on the number of choices of vertices in $F_1$ yields the desired conclusion for the case $F_1=F_2$.} 

Let us now turn to prove \eqref{eq:F-3-bd}. If $F_3$ is an edge this is obvious as in that case $v \in V(F_3)$ being mapped to a vertex of high degree the edge $F_3$ must be mapped to an edge in $E_{1,2}(\Graph) \cup E_{2,2}(\Graph)$. On the other hand, if $F_3$ is a cycle then there are two possibilities:
\begin{enumerate}
\item The edges of $F_3$ are all mapped to $E_{1,2}(\Graph)\cup E_{2,2}(\Graph)$
\end{enumerate}
and
\begin{enumerate}
\item[(2)] The edges of $F_3$ are mapped to both $E_{1,1}(\Graph)$ and $E_{1,2}(\Graph)\cup E_{2,2}(\Graph)$. 
\end{enumerate}
In the first subcase, as $F_3$ is a cycle the bound \eqref{eq:F-3-bd} follows from \eqref{eq:copy-H-star-bd}. Whereas in the second case, as $v$ is mapped to high degree vertex, all edges of $F_3$ cannot be mapped to edges in $E_{1,1}(\Graph)$. Thus in the case the bound \eqref{eq:F-3-bd} follows from Lemma \ref{lem:e-bar-ub}. This proves the bound \eqref{eq:F-3-bd}.

{Now we proceed to provide a proof for the case $V(F_1) \cap V(F_2) = \emptyset$. As before we choose the edge $(u_2,v)$ and then pick the vertex $u_1$. The number of choices is bounded by $D e_{1,2}(\Graph)$. Upon fixing these vertices,} we claim that there are at most $C_1\left( {\bm e}_{1,2}(\Graph)+ {\bm e}_{2,2}(\Graph)\right)^{(v_{F_1}-2)/2}$ and $C_2\left({\bm e}_{1,2}(\Graph)+ {\bm e}_{2,2}(\Graph)\right)^{(v_{F_2}-2)/2}$ many choices for the remaining vertices of $F_1$ and $F_2$, respectively, where $C_1, C_2 < \infty$ are some constants.

First let us derive the bound for the remaining vertices of $F_1$. 

\begin{itemize}

\item If $F_1$ is an edge, as $u_1\in V(F_1)$ is mapped to a low degree vertex, there are at most $D$ choices for its remaining vertex.

\item If $F_1$ is a cycle letting $u_3$ and $u_4$ be the two neighbouring vertices of $u_1$ in $F_1$, and arguing as before there are at most $D^2$ possible choices for $u_3$ and $u_4$. Once $u_3$ and $u_4$ has been fixed we need to choose the remaining vertices of ${F_1}$. We note that the subgraph spanned by the edges $E(F_1)\setminus \{(u_1,u_3) \cup (u_1,u_4)\}$ is a path of length $v_{F_1}-2$. Therefore using Lemma \ref{path} to conclude that there are at most  $C_1'\left({\bm e}_{1,2}(\Graph)+ {\bm e}_{2,2}(\Graph)\right)^{(v_{F_1}-2)/2}$ many choices for the remaining vertices, for some $C_1' <\infty$. This yields the bound when $F_1$ is a cycle. 

\end{itemize}

As $u_2$ is also mapped to a low degree vertex in $\Graph$, the argument for $F_2$ is identical. Hence we have the claim. 

{Building on this claim to complete the proof we need to bound the number of choices of the remaining vertices of $F_3$. If $F_3=F_1$ or $F_3=F_2$ then $v_{{\sf H}_*}=v_{F_1}+v_{F_2}$. Thus we already have the desired bound from the above claim. When $F_3$ is distinct from $F_1$ and $F_2$ the bound \eqref{eq:F-3-bd} continues to hold. Therefore, also in that case we have desired bound. The proof of the lemma is now complete.} 
\end{proof}

Equipped with Lemma \ref{lem:e-bar-ub-H--star-fixed} we are now ready to prove Lemma \ref{lem:e-bar-ub-H-fixed}.

\begin{proof}[Proof of Lemma \ref{lem:e-bar-ub-H-fixed}]
We begin by recalling the definition of the subgraph $\wh{\sf H} \subset {\sf H}$ and noting that $v_{\sf H}=v_{\wh{\sf H}}=v_{{\sf H}_*}+\sum_{i=4}^{m} v_{F_{i}}$. This, in particular implies that we only need to show that \eqref{eq:bar-cN-0-ubd-H-fixed} holds with ${\sf H}$ replaced by $\wh{\sf H}$. 

Thus in light of Lemma \ref{lem:e-bar-ub-H--star-fixed} it suffices to prove that for each $i\geq 4$, once the vertices in $\cup _{j=1}^{i-1} V(F_j)$ are chosen, there are at most $C \left({\bm e}_{1,2}(\Graph)+ {\bm e}_{2,2}(\Graph)\right)^{v_{F_i}/2}$ many choices for the vertices in $V(F_i)$, where $C< \infty$ some constant. 

To this end, fix $i\geq 4$, and assume that the vertices in $\cup _{j=1}^{i-1} V(F_j)$ have been fixed. We consider three subcases: 
\begin{enumerate}

\item[(i)] All edges of $F_i$ are mapped to $E_{1,2}(\Graph) \cup E_{2,2}(\Graph)$.

\item[(ii)] Edges of $F_i$ are mapped to both $E_{1,1}(\Graph)$ and $E_{1,2}(\Graph) \cup E_{2,2}(\Graph)$

\item[(iii)] All edges of $F_i$ are mapped to edges in $E_{1,1}(\Graph)$

\end{enumerate}

In the first case the desired bound is a consequence of \eqref{eq:copy-H-star-bd}. In the second case the bound follows from Lemma \ref{lem:e-bar-ub}. Turning to treat the last case, we recall the properties of the collection $\{F_i\}_{i=1}^m$ from Claim \ref{claim:F}. As $v_i''\in \cup_{j=1}^{i-1} V(F_j)$ has already been fixed, and $v_i'$ has to be mapped to a low degree vertex it follows that there are at most $D$ or ${\bm e}_{1,2}(\Graph)$ choices for $v_i'$ depending on whether $v_i''$ is mapped to a low degree or a high degree vertex. Once $v'_i$ has been fixed, as all the edges of $F_i$ have to be mapped to vertices of low degree, the rest of the vertices of $F_i$ can be chosen in at most $D^{v_{F_i}-1}$ ways. As $v_{F_i} \ge 2$,we conclude that, in the last case as well, the  number of ways to choose all the vertices of $F_i$ is indeed upper bounded by $C \left({\bm e}_{1,2}(\Graph)+ {\bm e}_{2,2}(\Graph)\right)^{v_{F_i}/2}$ for some $C< \infty$. This completes the proof of the lemma.
\end{proof}

\appendix

\section{Decomposition of regular graphs}\label{sec:decompose}

In this short section we provide the proofs of Lemma \ref{lem:perfect-match} and Proposition \ref{prop:graph-decompose}. Recall that  these two results yield decompositions of regular graphs with certain properties that were crucially used in Section \ref{sec:G_b} and in the proof of Proposition \ref{l:seqcounting}. First let us recall the standard notion of {\em proper edge coloring} of a graph.






\begin{dfn}[Proper edge coloring]
For any graph a proper edge coloring is a coloring of its edges so that edges with a common vertex receive different colors.
\end{dfn}

The following is a classical result on the proper edge coloring of a bipartite graph. For its proof we refer the reader to \cite[Theorem 1.4.18]{LP}.

\begin{thm}[K\"{o}nig's edge coloring theorem]
Any bipartite graph with maximum degree $\Delta$ has a proper edge coloring with $\Delta$ colors.
\end{thm}


We now proceed to the proof of Proposition \ref{prop:graph-decompose}. The idea behind this proof stems from that of \cite[Lemma 5.2]{hms}. The key difference is the use of K\"{o}nig's edge coloring theorem which 
allows us to produce a decomposition that avoids a prescribed edge (see \eqref{eq:e-not-there}). 

\begin{proof}[Proof of Proposition \ref{prop:graph-decompose}]
Let ${\sf H}'$ be the bipartite double cover of ${\sf H}$. That is, $V({\sf H}') := V({\sf H}) \times \{1,2\}$ and $((v_1,1), (v_2,2)) \in E({\sf H}')$ if and only if $(v_1, v_2) \in E({\sf H})$. 
By construction ${\sf H}'$ is a bipartite graph. Since ${\sf H}$ is a $\Delta$-regular graph so is ${\sf H}'$. Thus applying K\"{o}nig's edge coloring theorem we deduce that ${\sf H}'$ has a proper edge coloring with $\Delta$ colors. For $j \in \llbracket \Delta \rrbracket$ we let  $\mathscr{S}_j$ be the subset of the edges $E({\sf H}')$ that receive color $j$ in this proper edge coloring. 

We claim that 
\beq\label{eq:mathscr-S}
|\mathscr{S}_j| = e({\sf H}')/\Delta \qquad \text{ for all } j \in \llbracket \Delta \rrbracket. 
\eeq
To see this claim, if possible let us assume that $|\mathscr{S}_{j_\star}| > e({\sf H}')/\Delta$ for some $j_\star \in \llbracket \Delta \rrbracket$. Since the color of all edges in $\mathscr{S}_{j_\star}$ is same and it is obtained via a proper edge coloring we deduce that all the edges in $\mathscr{S}_{j_\star}$ must be vertex disjoint. Otherwise it violates the definition of a proper edge coloring. Therefore, by a slight abuse of notation viewing $\mathscr{S}_{j_\star}$ as a subgraph of ${\sf H}'$ we find that
\[
v_{\mathscr{S}_{j_\star}} = 2 e(\mathscr{S}_{j_\star}) > \f{2 e({\sf H}')}{\Delta} = v_{{\sf H}'},
\]
where the last equality is due to the fact that ${\sf H}'$ is a $\Delta$-regular graph. As $\mathscr{S}_{j_\star} \subset {\sf H}'$, this is indeed a contradiction. Thus we deduce that 
\[
|\mathscr{S}_j| \le e({\sf H}')/\Delta \qquad \text{ for all } j \in \llbracket \Delta \rrbracket. 
\] 
As $\sum_j |\mathscr{S}_j| = e({\sf H}')$ we arrive at \eqref{eq:mathscr-S}. Having proven this claim we now turn to complete the proof of the proposition. 

Since $\Delta \ge 3$ there exists $j_0 \in \llbracket \Delta \rrbracket$ such that
\beq\label{eq:e-not-there-1}
 e_1:=((u_1,1), (u_2,2)), e_2:=((u_1,2), (u_2,1)) \notin \mathscr{S}_{j_0}. 
 \eeq
 During the rest of the proof we will work with $\mathscr{S}_{j_0}$ to produce a matching (recall its definition from Definition \ref{dfn:matching}) of the original graph ${\sf H}$ with the desired property \eqref{eq:e-not-there}. 
 
 To this end, let us define the projection map $\pi: V({\sf H}') \mapsto V({\sf H})$ as follows: $\pi((v,\dagger)) = v$ for any $v \in V({\sf H})$ and $\dagger \in \{1,2\}$. Thus $\pi(\cdot)$ naturally induces a map from $E({\sf H}')$ to $E({\sf H})$ as well. By an abuse of notation we continue to denote the latter map by $\pi(\cdot)$. Equipped with these notation we recall from above that $\mathscr{S}_{j_0}$ is a matching in ${\sf H}'$. Therefore $\mathscr{M}:= \pi(\mathscr{S}_{j_0})$ is a subgraph of ${\sf H}$ with degree of its every non-isolated vertices is at most two. Thus every connected component of $\mathscr{M}$ is either a path or a cycle. That is $\mathscr{M}:= \{\{\mathscr{C}_i\}_{i=1}^{\lambda_1},  \{\mathscr{P}_i\}_{i=1}^{\lambda_2}\}$, for some $\lambda_1, \lambda_2 \ge 0$, where $\mathscr{C}_i$'s are cycles and $\mathscr{P}_j$'s are paths. Furthermore, by \eqref{eq:e-not-there-1} we deduce that $e=(u_1,u_2) \notin E(\mathscr{M})$.
 
 Hence to complete the proof it remains to establish that $\mathscr{P}_j$'s must be paths of length one and 
 \beq\label{eq:cover-vertex}
 \bigcup_{i=1}^{\lambda_1} V(\mathscr{C}_i) \cup \bigcup_{j=1}^{\lambda_2} V(\mathscr{P}_j) \supset V({\sf H}). 
 \eeq
 As $\mathscr{S}_{j_0}$ is a matching in ${\sf H}'$ and by \eqref{eq:mathscr-S} we have $v_{\mathscr{S}_{j_0}} = 2 e({\sf H}')/\Delta = v_{{\sf H}'}$ we conclude that for any $v \in V({\sf H})$ and $\dagger \in \{1,2\}$ the degree of $(v, \dagger)$ in $\mathscr{S}_{j_0}$ is one. Recalling that $\mathscr{M}=\pi(\mathscr{S}_{j_0})$ it immediately establishes \eqref{eq:cover-vertex}. Finally, for any $v \in V({\sf H})$ the degrees of $(v,1)$ and $(v,2)$ both being one in $\mathscr{S}_{j_0}$ also implies that degree of $v$ in $\mathscr{M}$ must be two unless the connected component containing $v$ is a path of length one.  Therefore all $\mathscr{P}_j$'s have to be paths of length one. 
 This completes the proof of the proposition. 
\end{proof}

We now proceed to the proof of Lemma \ref{lem:perfect-match}. It proof is similar to that of Proposition \ref{prop:graph-decompose}. Hence only a brief outline is provided. We remind the reader that in this lemma we need to prove that any regular bipartite graph ${\sf H}$ admits a perfect matching that avoids a collection of prescribed edges.

\begin{proof}[Proof of Lemma \ref{lem:perfect-match}]
As ${\sf H}$ is a bipartite graph with maximum degree $\Delta$ an application of K\"{o}nig's edge coloring theorem yields that ${\sf H}$ has a proper edge coloring with $\Delta$ colors. Denoting $\wh{\mathscr{S}}_j$ to be the collection of edges of $E({\sf H})$ that receive color $j$ and proceeding similarly as in the proof of the claim \eqref{eq:mathscr-S} we find that
\beq\label{eq:mathscr-S-1}
|\wh{\mathscr{S}}_j| = E({\sf H})/\Delta \qquad \text{ for all } j \in \llbracket \Delta \rrbracket. 
\eeq
As $\wh{\mathscr{S}}_j$ is obtained via a proper edge coloring all its edges must be vertex disjoint and thus it is matching. On other hand, \eqref{eq:mathscr-S-1} implies that $\wh{\mathscr{S}}_j$ is a perfect matching for any $j \in \llbracket \Delta \rrbracket$. 

To complete the proof we note that given any collection of $(\Delta-1)$ edges $\{e_1,e_2,\ldots, e_{\Delta-1}\}$ there exists a color $j_0 \in \llbracket \Delta \rrbracket$ such that those $(\Delta-1)$ edges receive a color different from $j_0$. Hence, the perfect matching induced by $\wh{\mathscr{S}}_{j_0}$ indeed avoids those $(\Delta-1)$ edges. This finishes the proof. 
\end{proof}

\section{Proofs of Lemmas \ref{lem:bad-graph-bd} and \ref{lem:bad-edges-bounds} }\label{app:app1}
Let us remind the reader that Lemmas \ref{lem:bad-graph-bd} and \ref{lem:bad-edges-bounds} provide a lower bound on the product of the degrees of adjacent vertices of any (strong)-core graph. Furthermore, Lemma \ref{lem:bad-edges-bounds} shows that for any strong-core graph there exists a large subgraph, containing most of the copies of ${\sf H}$ of the whole graph, such that the end points of all of its edges satisfy a tight upper and lower bound on the product of their degrees. This observation was crucial to the proof of the fact that the strong-core graphs are entropically stable.    

\begin{proof}[Proof of Lemma \ref{lem:bad-edges-bounds}]
First let us prove the lower bound on the product of the degrees. To this end, from \cite[Lemma 5.13]{hms}, for any edge $e =(u,v) \in E(\Graph)$ we have that 
\[
N({\sf H}, \Graph, e) \le 4 e({\sf H}) \cdot (2 e_\Graph)^{\f{v_{\sf H}}{2} - \f{2\Delta-1}{\Delta}} \cdot \left(4 \deg_\Graph (u) \cdot \deg_\Graph(v)\right)^{\f{\Delta-1}{\Delta}}.
\]
Since for any strong-core graph $\Graph$ and an edge $e \in E(\Graph)$ we have that
\beq\label{eq:copy-e-lbd}
N({\sf H}, \Graph, e) \ge ({\vep}/{\bar C_\star}) \cdot  (np^{\Delta/2})^{v_{\sf H}-2},
\eeq
it now follows from above that
\begin{multline}\label{eq:prod-lbd-pre}
\deg_\Graph(u) \cdot \deg_\Graph(v) \\
\ge \f14 \cdot \left(\f{\vep}{\bar C_\star}\right)^{\f{\Delta}{\Delta-1}} \cdot \left(\f{1}{4 e({\sf H})}\right)^{\f{\Delta}{\Delta-1}}\cdot \left[\f{(np^{\Delta/2})^{v_{\sf H}-2}}{(np^{\Delta/2})^{v_{\sf H}- 4 +\f{2}{\Delta}}} \right]^{\f{\Delta}{\Delta-1}} \cdot \left( \f{n^2 p^\Delta}{2 e(\Graph)}\right)^{(\f{v_{\sf H}}{2} - \f{2\Delta-1}{\Delta})\f{\Delta}{\Delta-1}} \\
 \ge \f14 \cdot \left(\f{\vep}{\bar C_\star}\right)^{\f{\Delta}{\Delta-1}} \cdot \left(\f{1}{4 e({\sf H})}\right)^{\f{\Delta}{\Delta-1}}\cdot n^2 p^\Delta \cdot (2 \bar C_\star)^{-(\f{v_{\sf H}}{2} - \f{2\Delta-1}{\Delta})\f{\Delta}{\Delta-1}}\ge c_0(\vep) n^2 p^\Delta,
\end{multline}
where the last step follows upon choosing $c_0(\vep)$ sufficiently small and the fact that $e(\Graph) \le \bar C_\star n^2 p^\Delta$. This completes the proof of part (a). 

Turning to proof part (b) let us fix some $C_0 < \infty$ and let $\Graph_{\rm high}$ be as in Definition \ref{dfn:bad-graph}. We claim that 
\beq\label{eq:e-G-high-ubd}
e(\Graph_{\rm high}) \le \f{5 e(\Graph)^2}{C_0 n^2 p^\Delta}.
\eeq
To see the above using \eqref{eq:copy-H-star-bd} we have that 
\[
N(P_3, \Graph) \le (2 e(\Graph))^2.
\]
On the other hand it is easy to see that
\begin{multline*}
N(P_3, \Graph) \ge 2 \sum_{(u,v) \in E(\Graph)} (\deg_\Graph(u)-1) \cdot (\deg_\Graph(v) -2) \\
\ge  2 \sum_{(u,v) \in E(\Graph)} \deg_\Graph(u) \cdot \deg_\Graph(v) - 3 \sum_{v \in V(\Graph)} (\deg_\Graph(v))^2 
\ge 2 C_0 n^2 p^\Delta \cdot e(\Graph_{\rm high})  - 6 e(\Graph)^2,
\end{multline*}
where the last step follows from the lower bound \eqref{eq:sc-ubd}. Combining the upper and lower bounds on $N(P_3, \Graph)$ the inequality \eqref{eq:e-G-high-ubd} is now immediate. 

Using \eqref{eq:e-G-high-ubd} we next proceed to find the desired upper bound on $e(\Graph_{\rm bad})$. As the lower bound \eqref{eq:copy-e-lbd} holds for every $e \in E(\Graph)$ we deduce that 
\beq\label{eq:e-G-bad-bd1}
e(\Graph_{\rm bad}) \cdot (\vep/\bar C_\star) (np^{\Delta/2})^{v_{\sf H}-2} \le \sum_{e \in E(\Graph_{\rm bad})} N({\sf H}, \Graph, e) \le \sum_{e' \in E(\Graph_{\rm high})} N({\sf H}, \Graph, e'),
\eeq
where the rightmost inequality is due to the fact by definition for any edge $e \in E(\Graph_{\rm bad})$ every copy of ${\sf H}$ passing through $e$ uses at least one edge belonging to $E(\Graph_{\rm high})$. 
From \cite[Lemma 5.15]{hms} we have that
\begin{align}\label{eq:e-G-bad-bd2}
\sum_{e \in E(\Graph_{\rm high})} N({\sf H}, \Graph, e)  & \le e({\sf H}) \cdot (2 e(\Graph))^{v_{\sf H}/2} \cdot \left(\f{e(\Graph_{\rm high})}{e(\Graph)}\right)^{1/\Delta} \notag\\
& \le e({\sf H}) \cdot (2 e(\Graph))^{v_{\sf H}/2} \cdot \left(\f{5 \bar C_\star}{C_0}\right)^{1/\Delta} 
\le 2^{v_{\sf H}/2} e({\sf H}) \cdot (n p^{\Delta/2})^{v_{\sf H} -2} \cdot \left(\f{5 \bar C_\star}{C_0}\right)^{1/\Delta} \cdot e(\Graph),
\end{align}
where the penultimate inequality follows from \eqref{eq:e-G-high-ubd} and the fact that $e(\Graph) \le \bar C_\star n^2 p^\Delta$. The last step is again a consequence of the upper bound on $e(\Graph)$. As, by definition, $\Graph_{\rm high} \subset \Graph_{\rm bad}$, now \eqref{eq:e-G-bad-bd1} together with \eqref{eq:e-G-bad-bd2} yield that  
\[
e(\Graph_{\rm high}) \le e(\Graph_{\rm bad}) \le \vep e(\Graph),
\]
upon choosing $C_0$ sufficiently large. Finally noting that
\[
N({\sf H}, \Graph) - N({\sf H}, \Graph_{\rm low})  \le \sum_{e \in E(\Graph_{\rm high})} N({\sf H}, \Graph, e)
\]
the lower bound 
\[
N({\sf H}, G_{\rm low})\ge (1-\vep) N({\sf H}, \Graph)
\]
follows from \eqref{eq:e-G-bad-bd2} upon using the upper bound on $e(\Graph)$, the fact that 
\[
N({\sf H}, \Graph) \ge \delta(1-6\vep) n^{v_{\sf H}} p^{e({\sf H})} = \delta(1-6\vep) (n p^{\Delta/2})^{v_{\sf H}},
\] 
and enlarging $C_0$ if necessary. This completes the proof of the lemma. 
\end{proof}

The proof of Lemma \ref{lem:bad-graph-bd} follows from a same line reasoning as that in Lemma \ref{lem:bad-edges-bounds}. Indeed, replacing \eqref{eq:copy-e-lbd} by the lower bound
\[
N({\sf H}, \Graph, e) \ge \vep n^{v_{\sf H}} p^{e({\sf H})} /(\bar C n^2 p^\Delta \log(1/p))
\]
which holds for any edge $e$ in a core graph, arguing similarly as in \eqref{eq:prod-lbd-pre}, and using the upper bound $e(\Graph) \le \bar C n^2 p^\Delta \log(1/p)$ Lemma  \ref{lem:bad-graph-bd} follows. We omit further details.

\end{document}